\documentclass[12pt, twoside]{article}

\usepackage{tocloft}
\renewcommand{\cftsecleader}{\cftdotfill{\cftdotsep}}
\usepackage{titletoc}
\usepackage{pdfpages}
\usepackage{amsmath,amssymb}
\usepackage{amsthm}
\usepackage{epsfig}
\usepackage{graphics}
\usepackage{graphicx}
\usepackage{colordvi}
\usepackage{mathrsfs} 
\usepackage{stmaryrd}
\usepackage{geometry}
\usepackage{dsfont}
\usepackage{tikz}
\usetikzlibrary{cd}
\usetikzlibrary{automata}
\usepackage{url}
\usepackage[hidelinks]
{hyperref}
\hypersetup{
    colorlinks=false,
   linkcolor=blue!75!black,
    filecolor=blue!75!black,      
   urlcolor=blue!75!black, 
   citecolor=blue!75!black,
}
\usepackage{enumitem}
\usepackage[textsize=small]{todonotes}
\usepackage[margin=1cm, size=small]{caption}
\usepackage{multirow}
\allowdisplaybreaks[4]
\oddsidemargin
-.2in \evensidemargin -.2in\marginparwidth 0.4in
\usepackage{color,soul}
\usepackage{xcolor}

\usepackage{etoolbox}
\patchcmd{\thebibliography}{\section*}{\section}{}{}
\usepackage{fancyvrb}
\usepackage{fvextra}
\usepackage{fancyhdr}
\usepackage{todonotes}

\usepackage{chngcntr}
\counterwithin{table}{section}
\counterwithin{figure}{section}
 
\usepackage{empheq}
\usetikzlibrary{automata}
\definecolor{blue2}{cmyk}{.94,.11,0,0}
\usepackage{enumitem}
\setitemize{itemsep=0pt}
\definecolor{myblue}{rgb}{.8, .8, 1}
\newlength\mytemplen
\newsavebox\mytempbox

\renewcommand{\Re}{{\rm Re}}
\newcommand{\BES}{{\rm BES}}

\renewcommand{\Im}{{\rm Im}}

\newcommand{\less}{\lesssim}
\newcommand{\more}{\gtrsim}
\newcommand{\e}{{\mathrm e}}
\newcommand{\1}{\mathds 1}

\newcommand{\C}{{\mathscr C}}
\newcommand{\F}{\mathscr F}
\newcommand{\B}{\mathscr B}

\newcommand{\vep}{{\varepsilon}}

\newcommand{\da}{{\downarrow}}

\newcommand{\lra}{\longrightarrow}

\newcommand{\la}{{\langle}}
\newcommand{\ra}{{\rangle}}

\newcommand{\dint}{{\int\!\!\!\int}}
\newcommand{\ind}{{\perp\!\!\!\perp}}

\newcommand{\supp}{{\rm supp}}

\newcommand{\ms}{\mathscr}
\renewcommand{\P}{{\mathbb P}}
\newcommand{\E}{{\mathbb E}}

\newcommand{\mc}{\mathcal}
\renewcommand{\d}{{\mathrm d}}

\newcommand{\al}{(-\alpha)}

\newcommand{\albe}{(-\alpha),\beta \da }
\newcommand{\zbe}{(0),\beta \da }
\newcommand{\R}{{\Bbb R}}

\renewcommand{\i}{{\mathtt i}}
\newcommand{\defeq}{{\stackrel{\rm def}{=}}}

\renewenvironment{proof}[1][\proofname]{\noindent {\bfseries #1.}\;}{\hfill\ensuremath{\blacksquare}\\}

\newcommand{\s}{{\mathfrak s}}

\newcommand{\cc}{{^\circ}}

\newcommand{\eqspace}{{\quad\;}}

\newcommand{\bi}{\mathbf i}

\newcommand{\two}{{\sqrt{2}}}

\tikzcdset{scale cd/.style={every label/.append style={scale=#1},
    cells={nodes={scale=#1}}}}

\newtheoremstyle{slantthm}{10pt}{10pt}{\slshape}{}{\bfseries}{}{.5em}{\thmname{#1}\thmnumber{ #2}\thmnote{ (#3)}.}
\newtheoremstyle{slantrmk}{10pt}{10pt}{\rmfamily}{}{\bfseries}{}{.5em}{\thmname{#1}\thmnumber{ #2}\thmnote{ (#3)}.}

\begin{document}
\theoremstyle{slantthm}
\newtheorem{thm}{Theorem}[section]
\newtheorem{prop}[thm]{Proposition}
\newtheorem{lem}[thm]{Lemma}
\newtheorem{cor}[thm]{Corollary}
\newtheorem{defi}[thm]{Definition}
\newtheorem{disc}[thm]{Discussion}
\newtheorem{conj}[thm]{Conjecture}

\theoremstyle{slantrmk}
\newtheorem{ass}[thm]{Assumption}
\newtheorem{hypothesis}[thm]{Hypothesis}
\newtheorem{rmk}[thm]{Remark}
\newtheorem{eg}[thm]{Example}
\newtheorem{que}[thm]{Question}
\numberwithin{equation}{section}
\newtheorem{outline}[thm]{Outline}
\newtheorem{quest}[thm]{Quest}
\newtheorem{prob}[thm]{Problem}
 \newtheorem{nota}[thm]{Notation}
 \newtheorem{cond}[thm]{Condition}
 
\newcommand{\thetitle}{Two-dimensional delta-Bose gas: skew-product relative motions}

\title{\vspace{-1.2cm}
\bf \thetitle\footnote{Support from an NSERC Discovery grant is gratefully acknowledged.}}

\author{Yu-Ting Chen\,\footnote{Department of Mathematics and Statistics, University of Victoria, British Columbia, Canada.}\,\,\footnote{Email: \url{chenyuting@uvic.ca}}}

\date{\today \vspace{-1.1cm}
}

\maketitle
\abstract{We prove a Feynman--Kac-type formula for the relative motion of the two-body delta-Bose gas in two dimensions. The multiplicative functional is not exponential, and the process is a skew-product diffusion uniquely extended in law, in the sense of Erickson~\cite{Erickson}, from ${\rm BES}(0,\beta{\downarrow})$ of Donati-Martin and Yor~\cite{DY:Krein} as the radial part. We give two different proofs of the formula. The first uses the original excursion characterization of $\BES(0,\beta\da)$, and the second is via the lower-dimensional Bessel processes at the expectation level. The latter proof contrasts the long-standing approach for delta-function interactions
by adding mollifiers to the Laplacians since the present approximations are from ``lower, fractional dimensions.''  Moreover, the second proof conducts a new study of $\BES(0,\beta\da)$ as an SDE since we handle the drift via certain ratios of the Macdonald functions.  The properties proven include the strong well-posedness and comparison of the SDE of $\BES(0,\beta\da)$ for all initial conditions. In particular, this well-posedness contrasts the fact that the skew-product diffusion for the Feynman--Kac-type formula has a singular drift of $L^p_{\rm\tiny loc}$-integrability only for $ p\leq 2$. \smallskip

\smallskip

}

\setlength{\cftbeforesecskip}{5pt}
\renewcommand{\cftsecleader}{\cftdotfill{\cftdotsep}}

\setlength{\cftbeforesecskip}{0pt}
\setlength\cftaftertoctitleskip{0pt}
\renewcommand{\cftsecleader}{\cftdotfill{\cftdotsep}}
\setcounter{tocdepth}{1}
\tableofcontents

\section{Introduction}\label{sec:intro}
The main objective of this paper is to establish and describe a stochastic path integral of the following formal Hamiltonian of a quantum mechanical model: 
\begin{align}\label{def:L}
\ms L\,\defeq -\Delta_z-\Lambda\delta(z),\quad z\in \Bbb C,
\end{align}
where $\Delta_{z}=\Delta\,\defeq\,\partial_x^2+\partial_y^2$ is the Laplacian in $z=x+\i y$. The coupling constant $\Lambda$ and the delta function $\delta(\cdot)$ all together are meant to define a multiplication operator and model analytically how two particles are coupled by a delta-like interaction. Despite its formal simplicity, our motivation for studying the above Hamiltonian considers the applications.
The operator $\ms L$ defines the delta-like interaction of any two particles in the two-dimensional delta-Bose gas, exactly in the two-body case and locally in the many-body case. It is also an essential part of moments of the two-dimensional stochastic heat equation (SHE) at criticality. Below, we will discuss $\ms L$ in terms of these applications and then overview the paper's main results. See also the monograph of Albeverio, Gesztesy, H\o{}egh-Krohn and Holden~\cite{AGHH:Solvable} for reviews of many-center extensions of the Hamiltonian $\ms L$ as in the Kronig--Penney model \cite{KP} from solid state physics. We do not pursue any extension to this direction here, though.

\subsection{The two-dimensional many-body delta-Bose gas}\label{sec:intro_DBG}
The model has a formal Hamiltonian given by the following operator:
\begin{align}\label{def:H}
\ms H\,\defeq
 -\frac{1}{2}\sum_{j=1}^N\Delta_{z_j}-\Lambda\sum_{j<k}\delta(z_k-z_j),
\end{align}
where $z_j\in \Bbb C$ is the position of the $j$-th particle. The simplest relation between \eqref{def:L} and \eqref{def:H} occurs when $N=2$. In this case, a straightforward change of variables shows that $\ms L$ defines the dynamic of the relative motion $z=z_2-z_1$, and the complementary dynamic is given by $-\Delta_{z'}/4$ for the center of mass $z'=(z_2+z_1)/2$ (cf. \cite[Proposition~2.5]{C:DBG3+}). Henceforth, the {\bf relative motion} will also refer to the operator $\ms L$. 

More precise relations between the relative motion and the delta-Bose gas for $N\geq 3$ also exist, though in a much more involved manner. In the original quantum mechanical context, these relations have appeared more or less in the functional analytic constructions going back to Dell'Antonio, Figari and Teta~\cite{DFT:Schrodinger} and then Dimock and Rajeev \cite{DR:Schrodinger}. See \cite{GH:Short} for a recent extension of these functional analytic results.
Moreover, via certain diagrammatic expansions, the Schr\"odinger semigroups of the delta-Bose gas have been proven explicitly, analytically reducible to the relative motions and some free motions obeying Gaussian dynamics~\cite{CSZ:Mom,GQT,C:DBG3+}. In more detail, these diagrammatic expansions allow an interpretation of how the relative motions enter locally in the many-body delta-Bose gas, although this local property may be heuristically read off from \eqref{def:H}.

The construction of the relative motions is obtained by Albeverio, Gesztesy, H\o{}egh-Krohn and Holden~\cite{AGHH:2D}, extending the construction by Berezin and Faddeev~\cite{BF:Schrodinger} for
 the analogous Hamiltonians in three dimensions. 
 One of the main results in \cite{AGHH:2D} identifies realizations of $\ms L$ as a one-parameter family of self-adjoint  operators $\Delta_\beta$ with domain $\ms D(\Delta_\beta)\subset L^2(\Bbb C)$, where $\beta$ ranges over $[0,\infty)$ and $\Delta_0$ is just the Laplacian $\Delta_z$. These operators can be characterized by the resolvents $(-\Delta_\beta+q)^{-1}$ such that the kernels of the semigroup $P^\beta_t=\exp\{t\Delta_\beta\}$ satisfy the following formula \cite[Theorem~2.2]{AGHH:2D}:
\begin{align}
\begin{split}\label{def:DBG}
\int_0^\infty \e^{-q t}P^\beta_t(z,\tilde{z})\d t&=G_q(z-\tilde{z})+\frac{4\pi}{\log (q/\beta)}G_q(z)G_{q}(\tilde{z}),\quad\Re(q)>\beta.
\end{split}
\end{align}
Here, $G_q(z)$ denotes the integral density of $(-\Delta_z+q)^{-1}$:
\[
G_q(z)\,\defeq\int_0^\infty \e^{-q t}P_{2t}(z)\d t,
\]
which is defined by the two-dimensional Brownian transition density:
\begin{align}\label{def:Pt}
P_t(z,\tilde{z})=P_t(z-\tilde{z})\,\defeq\,\frac{1}{2\pi t}\exp\left(-\frac{\lvert z-\tilde{z}\rvert^2}{2t}\right).
\end{align}
See the first paragraph of Section~\ref{sec:FKformula} on the Laplace inversions of \eqref{def:DBG}.

The discrepancy between the obvious probabilistic interpretation of the relative motion
and the strong functional analytic flavor from \eqref{def:DBG} should raise the question of whether the Feynman--Kac formula for $\{P^\beta_t\}$ may exist. ($\{P^\beta_t\}$ is not Markovian or sub-Markovian since $P^\beta_t\1>1$ for $t>0$.) On the other hand, the physical interpretation of the Hamiltonian $\ms L$ is most natural to start with. This suggests mollifying the delta-function potential, which is an approximation scheme for $\{P^\beta_t\}$ established in \cite{AGHH:2D}. Specifically, \cite{AGHH:2D} considers the \emph{weakly attractive} Hamiltonians 
\begin{align}\label{def:approxH}
-\Delta-\Lambda_{(\vep)}\phi_\vep=-\Delta_z-\Lambda_{(\vep)}\phi_\vep(z)
\end{align}
in the limit of $\vep\searrow 0$. Here,
\begin{align}\label{def:Lambdaphi}
 \Lambda_{(\vep)}=\Lambda_{(\vep)}(\lambda)\,\defeq\,\frac{2\pi}{\log \vep^{-1}}+
\frac{2\pi\lambda}{(\log \vep^{-1})^2},\quad \phi_\vep(z)\,\defeq \,\vep^{-2}\phi(\vep^{-1}z),
\end{align}
$\lambda\in \R$, and $\phi$ is a $\C_c^\infty$-probability density, among other cases handled in \cite{AGHH:2D}. For the case in \eqref{def:Lambdaphi}, the relationship between $(\lambda,\phi)$ and the parameter $\beta$ can be restated as follows:
\begin{align}\label{def:beta}
\frac{\log \beta}{2}=-\dint \phi(z)\phi(\tilde{z})\log \lvert z-\tilde{z}\rvert \d z\d \tilde{z}+\log 2+\lambda-\gamma_{\sf  EM},
\end{align}
where $\gamma_{\sf  EM}=0.57721...$ denotes the Euler--Mascheroni constant. 
Also, the Feynman--Kac formula is available by classical theorems under \eqref{def:Lambdaphi}. It holds precisely that
\begin{align}\label{FK:approx}
\e^{(\Delta+\Lambda_{(\vep)}\phi_\vep)t}f(z_0)=\E\left[\exp\left(\Lambda_{(\vep)}\int_0^t \phi_\vep(\two B(s))\d s\right)f(\two B(t))\right],
\end{align} 
where $B$ is a two-dimensional standard Brownian motion such that $\sqrt{2}B(0)=z_0$.

Despite its appealing accessibility, a Brownian approximation starting with \eqref{FK:approx} is quite problematic for proving an \emph{exact} form of the Feynman--Kac formula of the semigroup $\{P^\beta_t\}$. The obstruction arises as soon as one attempts to pass a limit of the expectation
in \eqref{FK:approx} by a Portmanteau-theorem type argument relying on weak convergence of the random variable. This is because a nonzero distributional limit of the additive functionals $\Lambda_{(\vep)}\int_0^t \phi_\vep(\two B(s))\d s$ occurs only when $B(0)=0$. The limit is the standard exponential, say $\mathbf e$, according to the Kallianpur--Robbins law for scaling limits of additive functionals of two-dimensional Brownian motions~\cite{KR:1}. But then, even for 
$f\equiv 1$, the  limiting expectation $\E[\exp\{\mathbf e\}]=\infty$ from \eqref{FK:approx} is the uninformative conclusion. For the $N$-body delta-Bose gas with $N\geq 3$, attempting the Feynman--Kac formula 
 by this method is faced with not only the same issue but also the question of correlated interactions.

\begin{rmk}\label{rmk:FK1D}
(1$\cc$) The above approximation by \eqref{def:Lambdaphi} uses \emph{general mollifcations} for potentially broad applications and physical interpretations. Otherwise, the purpose of mollification alone can be met just by taking $\phi$ as the indicator function of a centered disc, since this allows exact solutions of the Schr\"odinger semigroups of $-\Delta-\Lambda \phi$. See \cite{Chan} for the exact formula if $\Lambda\leq 0$.  Note that the three-dimensional case under the choice of radially symmetric indicator mollifiers is discussed in \cite{AFH:SPNA,Nelson}. \smallskip

\noindent (2$\cc$) In stark contrast to the two-dimensional case,
the many-body delta-Bose gas in one dimension allows a precise Feynman--Kac formula as an expected Brownian exponential functional \cite{BC:1D}. In particular, the formula for the relative motion is given by
\begin{align}\label{sgp:1D}
\E\Bigg[\exp\left(\Lambda\int_0^t \delta(\two W(s))\d s\right)f(\two W(t))\Bigg].
\end{align}
Here, $W$ is a one-dimensional standard Brownian motion, and the additive functional $\int_0^t \delta(\two W(s))\d s$ is mathematically realizable as the local time of $ \two W$ at the point 0.  Note that the formula in \eqref{sgp:1D} takes the limiting form of the one-dimensional counterpart of the expectation in \eqref{FK:approx}, but any value of order $\mathcal O(1)$ for the coupling constant $\Lambda$ is now allowed. These characteristics are shared by the many-body formula in \cite{BC:1D}, and a Portmanteau-theorem type argument by approximations with mollified delta-function potentials
 can indeed be carried out. \qed 
\end{rmk}

\subsection{Statistical moments of the two-dimensional SHE at criticality}\label{sec:intro_SHE}
Our motivation for representing the relative motion probabilistically also considers the statistical properties of the two-dimensional stochastic heat equation (SHE). The equation is given by the following \emph{purely formal} SPDE:
\begin{align}\label{def:SHE}
\frac{\partial X}{\partial t}(z,t)=\frac{\Delta_z}{2}X(z,t)+\Lambda^{1/2}X(z,t)\xi (z,t),\quad z\in \Bbb C,
\end{align}
where $\xi(z,t)$ is space-time white noise, namely a  
centered distribution-valued Gaussian field with the correlation function 
\[
\E[\xi(z,t)\xi(\tilde{z},\tilde{t})]=\delta(z-\tilde{z})\delta(t-\tilde{t}). 
\]

In general, the SHE of all (spatial) dimensions has received a long-standing interest in the physics and probability literature. A formal solution of the SHE can be physically interpreted as the Cole--Hopf transformation of the Kardar--Parisi--Zhang equation for growing interfaces~\cite{KPZ}. Another interpretation is called the continuum directed random polymer by extending the original one-dimensional model \cite{AKQ} to all dimensions. See \cite{Comets, Quastel} for reviews of the physical relations in one dimension and \cite{Toninelli} for the case of two dimensions. On the other hand, 
 above one dimension as in the case of \eqref{def:SHE}, solving the SHE mathematically has to resolve
 the highly nontrivial issue of making sense of products of distributions in the noise term since the solution is expected to be distribution-valued. This situation is very different from the one-dimensional case, where the solutions are function-valued by extending It\^{o}'s theory~\cite{Walsh}. Solving the distributions of solutions is different, though.

Bertini and Cancrini~\cite{BC:2D} prove the existence of solutions to the two-dimensional SHE ``at criticality'' by relating its two-point moments to the relative motion. The approach in \cite{BC:2D} begins with the following approximate equations subject to regularized noise such that function-valued solutions also exist in It\^{o}'s sense~\cite{Walsh}:
\begin{align}\label{def:Xvep}
\frac{\partial X_\vep}{\partial t}(z,t)=\frac{\Delta_z}{2}X_\vep(z,t)+\Lambda_{(\vep)}^{1/2}X_\vep(z,t)\xi_\vep(z,t).
\end{align}
Here, $\Lambda_{(\vep)}$ is defined by \eqref{def:Lambdaphi}, and the noise is only regularized in space so that \eqref{def:Xvep} resembles short-range discrete models:
\[
\xi_\vep(z,t)\,\defeq \int_{\Bbb C}\varrho_\vep(z-\tilde{z})\xi(\tilde{z},t)\d \tilde{z},
\]
where $\varrho_\vep(z)\,\defeq \,\vep^{-2}\varrho(\vep^{-1}z)$ for a given probability density $\varrho\in \C_c^\infty$. Then, by duality, the two-point moments of $X_\vep$ are expressible as the semigroups of the approximate relative motion \eqref{def:Lambdaphi}, hence
as the Feynman--Kac formula in \eqref{FK:approx}. Specifically, we have
\begin{align}\label{SHE:2mom}
\E\Bigg[\prod_{j=1}^2X_\vep(z_j,t)\Bigg]=\E\Bigg[\exp\left(\Lambda_{(\vep)}\int_0^t \tilde{\phi}_\vep(B_2(s)-B_1(s))\d s\right)\prod_{j=1}^2X_\vep(B_j(t),0)\Bigg],
\end{align}
where $\tilde{\phi}_\vep(z)\,\defeq\int \varrho_\vep(z+\tilde{z})\varrho_\vep(\tilde{z})\d \tilde{z}$, and $B_1,B_2$  are independent two-dimensional standard Brownian motions such that
 $B_j(0)=z_j$. Accordingly, it is proven in \cite{BC:2D} that the family of laws of the approximate solutions $\{X_\vep\}$ is tight, and the two-point moments of any subsequential $\vep\to 0$ limit can be explicitly solved by using $\{P^\beta_t\}$ from \eqref{def:DBG}.  

 These results in \cite{BC:2D} on the \emph{approximate} moments and their applications have been extended in several ways. The extensions include some observations of non-Gaussian fluctuations of the SHE (in the form of random polymers) \cite{Feng}
 and the convergences of the approximate moments $\E[\prod_{j=1}^NX_\vep(z_j,t)]$ for general $N$ via approximations of the $N$-body delta-Bose gas and diagrammatic expansions
\cite{CSZ:Mom,GQT,C:DBG3+}.  Moreover, by studying some of these approximate moments,
Caravenna, Sun and Zygouras~\cite{CSZ:Flow} offer a rigorous meaning to the two-dimensional SHE at criticality in the sense of proving unique distributional limits of the discrete analogues. An extension of these unique distributional limits to \eqref{def:Xvep} as $\vep\to 0$ is outlined in \cite[Remark~1.4]{CSZ:Flow}. 

For the \emph{limiting} case, the very fast growth of the $N$-point moments in $N$ has been expected and will make using moment-generating functions difficult for determining the probability distribution of the two-dimensional SHE at criticality \cite{GQT,CSZ:Flow,CSZ:Mut}.
On the other hand, the moments are basic statistical objects and give some of the very few explicit characteristics for the two-dimensional SHE at criticality. Thus, it should be reasonable to expect their usefulness in proving more properties of the SHE. For example, Caravenna, Sun and Zygouras~\cite{CSZ:Mut} show that by proving some particular bounds for the limiting moments, the solution of the two-dimensional SHE at criticality is not a Gaussian multiplicative chaos. 
For more applications of the limiting $N$-point moments, 
stochastic path integrals 
 of the $N$-body delta-Bose gas are natural candidates of a useful tool 
if they exist. But this possible approach
requires solving the case of the relative motion at the very least, and then the difficulties discussed in Section~\ref{sec:intro_DBG} for proving the Feynman--Kac formula have to be resolved.

\subsection{Overview of main results}\label{sec:intro_overview}
We will prove the first form of stochastic path integral for the two-dimensional delta-Bose gas. Specifically, we obtain a Feynman--Kac-type formula such that the semigroup of the relative motion from \eqref{def:DBG} is stochastically represented as follows:
 \begin{align}\label{main:PR}
P^\beta_tf(z_0)=\E\left[\frac{\e^{\beta t}K_0(\sqrt{\beta}\lvert z_0 \rvert)}{K_0(\sqrt{\beta}\lvert \two Z_t\rvert)}f(\two Z_t)\right],\quad \forall\;f\geq 0,\;z_0\in \Bbb C\setminus\{0\},
\end{align}
and an extension to $z_0=0$ is also obtained (Theorem~\ref{thm:1}).
In \eqref{main:PR}, $K_\nu$ is the Macdonald function of index $\nu$  to be recalled in \eqref{def:K},
and the underlying diffusion $\{\two Z_t\}$, which we will call the {\bf stochastic relative motion} in the sequel,
 starts from $z_0$ and is a complex-valued skew-product diffusion. Additionally, the process $\{\two |Z_t|\}$ will be called the {\bf stochastic relative position}, since 
the {\bf relative position} refers to the radial part of the relative motion.  
We stress that \eqref{main:PR} is not a standard ``exponential form'' of the Feynman--Kac formula as in the one-dimensional case recalled in Remark~\ref{rmk:FK1D} (2$\cc$). See Hypothesis~\ref{hyp:FK1} and Corollary~\ref{cor:FK} for details. 

We will give two different proofs of the Feynman--Kac-type formula \eqref{main:PR} and its extension to $z_0=0$. See Section~\ref{sec:firstdr}, especially Outline~\ref{outline:main}. Each proof suggests the interpretation that the stochastic relative motion $\{\two Z_t\}$ is a ``two-minus dimensional'' diffusion, and these proofs rely on the fact that
$\{\lvert Z_t \rvert\}$ is chosen as $\BES(0,\beta\da)$, a \emph{singular} transformation of Bessel processes due to Donati-Martin and Yor~\cite{DY:Krein}. New properties of $\BES(0,\beta\da)$ and the lower-dimensional counterparts $\BES(-\alpha,\beta\da)$ will be obtained in the second proof.

The other basic characteristics of $\{\two Z_t\}$ are as follows. The generator of $\{ \two Z_t\}$ is given by
\begin{align}\label{def:gen}
\ms A^{\beta\da}_0\,\defeq\,\Delta_z+b_\beta(z)\cdot \nabla_z,\quad z\in \Bbb C\setminus\{0\},
\end{align}
where, with $\nabla_z$ denoting the gradient in $z$,
\begin{align}\label{def:bbeta}
b_\beta(z)\,\defeq\, -2\sqrt{\beta}\frac{K_1}{K_0}(\sqrt{\beta}|z|)\frac{z}{|z|}=\nabla_z \log K_0(\sqrt{\beta}|z|)^2.
\end{align}
{\bf Here and in what follows, dot products and differentiations with respect to complex variables are understood by identifying $\Bbb C$ with $\R^2$ as in (\ref{def:L}).} Also, in contrast to planar Brownian motion, $0$ \emph{can} be reached by $\{ \two Z_t\}$ so that it realizes the heuristic attractive interactions due to the negative potential from \eqref{def:approxH}. (The classical Feynman--Kac formula does not reflect the potential in the associated process.)
In particular, $\{\two Z_t\}$ is neither Gaussian nor absolutely continuous to planar Brownian motion. See Sections~\ref{sec:aug} and \ref{sec:FKformula} for details.

The singular drift $b_\beta$ \eqref{def:bbeta} of $\{\two Z_t\}$ appears to make the process special. 
Naturally, the process $\{\two Z_t\}$  can be viewed in terms of the SDE with singular drift, a subject dating at least back to \cite{Ver}. Moreover, by the second equality in \eqref{def:bbeta}, the process is also a distorted Brownian motion in the sense of Ezawa, Klauder and Shepp~\cite{EKS:DBM}, a relationship brought to our attention by an anonymous reviewer after we announced the first version of this paper. Here, a distorted Brownian motion in $\R^d$ is a Markov process with generator 
\begin{align}\label{def:DBM}
\Delta_z+\nabla_z \log \rho\cdot \nabla_z,\quad z=(z_1,\cdots,z_d)\in \R^d,
\end{align}
where $\rho=\rho(z,t)$ is positive a.e. Specifically,  to relate \eqref{def:gen} to \eqref{def:DBM}, we take 
\begin{align}\label{choice:DBM}
\mbox{$d=2$\quad and \quad $\rho(z,t)\equiv K_0(\sqrt{\beta} |z|)^2$.}
\end{align}
Nevertheless, it seems that no earlier existence or uniqueness results for SDEs with singular drift or distorted Brownian motions apply to $\ms A^{\beta\da}_0$, since the drift is quite singular around the origin. See Section~\ref{sec:Markovgen} for various ways to quantify the singularity, and Eberle~\cite{Eberle} for versions of uniqueness for distorted Brownian motions. In particular, we notice that
\cite{AHS:DBM} established a general theorem on the existence of distorted Brownian motions (Theorem~4.1 \emph{ibid.}), and some connections between delta-function potentials and distorted Brownian motions for \emph{$d=3$} (Examples~3 and 4 and Lemma~2.6 \emph{ibid.}). (Another result related to \cite{AHS:DBM} on the three-dimensional case can be found in \cite[Example~2.1]{Streit}.)
But the conditions of \cite[Theorem~4.1]{AHS:DBM} are not satisfied by $b_\beta$. See the second part of 
Proposition~\ref{prop:DBM} (2$\cc$). In this paper,
we shall not digress further to the question of 
 well-posedness in general situations 
including $\ms A^{\beta\da}_{0}$ as a special case, however.

Let us close this introduction by revisiting the relation in Section~\ref{sec:intro_SHE} to the two-dimensional SHE at criticality. Note that according to those properties mentioned below \eqref{def:bbeta}, \eqref{main:PR} also shows a non-Gaussian representation of the two-point moments of the SHE. Recall the approximations of these moments discussed below \eqref{SHE:2mom}. We leave to future investigation the central mechanism of the SHE that generates this non-Gaussianity.
\smallskip 

\noindent {\bf Organization of the remaining of this paper.} Section~\ref{sec:MAINRESULTS} gives the precise formulations of the main results. Section~\ref{sec:EXCURSION} gives the first proof of \eqref{main:PR} and the extension to $z_0=0$. The second proof is divided into Sections~\ref{sec:SDE} and~\ref{sec:DBG}. Section~\ref{sec:BES} collects several related properties of the lower-dimensional Bessel processes. Finally, Section~\ref{sec:other} gives the proofs of Corollary~\ref{cor:FK}, and Propositions~\ref{prop:esp}, \ref{prop:GEN}, \ref{prop:singular}, \ref{prop:DBM} and \ref{prop:osgood}; these are auxiliary properties to be discussed in Section~\ref{sec:MAINRESULTS}.
 \smallskip

\noindent {\bf Frequently used notation.} $C(T)\in(0,\infty)$ is a constant depending only on $T$ and may change from inequality to inequality. Other constants are defined analogously, but constants labelled by the same equation number are always equal to each other. We write $A\less B$ or $B\more A$ if $A\leq CB$ for a universal constant $C\in (0,\infty)$. $A\asymp B$ means both $A\less B$ and $B\less A$. For a process $Y$, the expectations $\E^Y_y$ and $\E^Y_\nu$ and the probabilities $\P^Y_y$ and $\P^Y_\nu$ mean that the initial conditions of $Y$ are the point $x$ and the probability distribution $\nu$, respectively. Also, $\log $ is defined with base `$\e$' throughout this paper, and we continue to work with $\Bbb C$ for the two-dimensional space, only to save notations.\smallskip

\section{Formulations of main results}\label{sec:MAINRESULTS}
\subsection{The stochastic relative position}
The process $\BES(0,\beta\da)$, $\beta\in(0,\infty)$, found by Donati-Martin and Yor~\cite{DY:Krein} is at the heart of our main results. Under this diffusion, the point $0$ is instantaneously reflecting, and
 the infinitesimal generator on $(0,\infty)$ is given by  
\begin{align}\label{gen:BESab}
\frac{1}{2}\frac{\d^2}{\d x^2}+\left(\frac{1}{2x}-\sqrt{2\beta}\frac{K_1}{K_0}(\sqrt{2\beta}x)\right)\frac{\d}{\d x}.
\end{align}
It shall become clear why we think of $\BES(0,\beta\da)$ as a Bessel process of index $(0-)$, equivalently of dimension $(2-)$, conditioned to hit zero with strength $\beta$.

Now, given the presence of the Macdonald functions in \eqref{gen:BESab} and the earlier Brownian approximations discussed in Section~\ref{sec:intro}, it may not be immediately clear why $\BES(0,\beta\da)$ can be related to $\{P^\beta_t\}$ in \eqref{def:DBG}. The paper \cite{DY:Krein} nevertheless grabbed our attention because of the following two observations obtained in \cite{C:DBG0,C:DBG3+}:

\begin{itemize}[leftmargin=3\labelsep]
\item \emph{The unusual prefactor $4\pi /\log (q/\beta)$ in the characterization \eqref{def:DBG} of $\{P^\beta_t\}$ can be identified as a double-Laplace transform of the transition density of the gamma subordinator.} See \cite[the proof of Proposition~5.1]{C:DBG3+} and \eqref{Inv:sbeta}.

On the other hand, the diffusion $\BES(0,\beta\da)$ is precisely what Donati-Martin and Yor choose to explicitly represent the gamma subordinator in  \cite{DY:Krein}. The inverse local time of $\BES(0,\beta\da )$ for the point $0$ is the gamma subordinator. This representation goes back to the It\^o--McKean problem \cite[p.217]{IM:Diffusion}, which inquires a description of the class of L\'evy measures of subordinators such that they are also the L\'evy measures of inverse local times of diffusions. See Knight~\cite{Knight:Krein} and Kotani and Watanabe~\cite{KW:Krein} for complete, theoretical solutions of the It\^{o}--McKean problem by Krein's theory of strings. A review of the processes $\BES(-\alpha,\beta\da)$, to be recalled in \eqref{def:BESab} for the case of $\alpha>0$, and the background of the It\^{o}--McKean problem can be found in 
 \cite[Chapter~15, especially pp.284--285]{SSV:Bern}.

\item \emph{Even if $\phi_\vep$ is not radially symmetric,
the two-dimensional Bessel process alone via \eqref{FK:approx} is enough to approximate the nontrivial part of $P^\beta_tf$, defined by 
\[
[4\pi/\log(q/\beta)] G_q(z)G_q(\tilde{z})
\]
in \eqref{def:DBG}, where $\beta$ relates to the approximation by \eqref{def:beta}.} See \cite[Section~2]{C:DBG0}.

This reduction in \cite{C:DBG0}, effectively only at the expectation level, applies the excursion theory of the two-dimensional Bessel process and is motivated by the pathwise approach of Kasahara and Kotani~\cite{KK} for the Kallianpur--Robbins law \cite{KR:1} and a generalization to the fluctuations of the additive functionals. See Cs\'aki, F\"oldes and Hu \cite{CFH} for further extensions and also \cite[Section~2]{AGHH:2D} for some analysis of $\{P^\beta_t\}$ by radial symmetry. 
\end{itemize}
These two observations suggest to us that $\BES(0,\beta\da)$ be a stochastic counterpart of the relative position.

The technical motivation explained in \cite{DY:Krein} for the construction of $\BES(0,\beta\da )$ spurs on our further investigation. It shows how this diffusion arises from exponentials of local times, which is reminiscent of \eqref{FK:approx} and Remark~\ref{rmk:FK1D} (2$\cc$). 
To describe precisely how these exponentials enter, first,
given $\alpha\in [0,1)$ and $x_0\in \R_+$, let $\{\rho_t\}$ under $\P^{(-\alpha)}_{x_0}$ start from $x_0$ and be a Bessel process of index $-\alpha$, denoted by $\BES(-\alpha)$. Equivalently, $\{\rho_t\}\sim \BES(-\alpha)$ is of dimension 
\begin{align}\label{def:delta}
\delta=2(1-\alpha).
\end{align}
Then $\BES(0,\beta\da)$ is considered  in \cite{DY:Krein} as the limit of the diffusion  $\BES(-\alpha,\beta\da )$ as $\alpha\searrow 0$, where $\BES(-\alpha,\beta\da )$ has a law defined by
\begin{align}\label{def:BESab}
\d\P_{x_0}^{(-\alpha),\beta \da}\rvert_{\sigma(\rho_s;s\leq t)}\,\defeq\; \frac{\widehat{K}_\alpha(\sqrt{2\beta}\rho_t)}{\widehat{K}_\alpha(\sqrt{2\beta}x_0)}
\e^{\Lambda_\alpha(\beta) L_t-\beta t}\d \P_{x_0}^{(-\alpha)}\rvert_{\sigma(\rho_s;s\leq t)},\quad\alpha\in (0,1).
\end{align}
Here in \eqref{def:BESab}, $\e^{\Lambda_\alpha(\beta)L_t}$ is the exponential of a local time that draws our attention, and the Radon--Nikod\'ym derivative process in \eqref{def:BESab} uses the following definitions:
\begin{itemize}[leftmargin=3\labelsep]
\item For all $\nu\in \R$, 
\begin{align}\label{def:hK}
\widehat{K}_\nu(x)\,\defeq\,
x^\nu K_{\nu}(x),\quad x\in (0,\infty),
\end{align}
and $K_\nu$ is the Macdonald function of index $\nu$ satisfying 
\begin{align}\label{def:K}
K_\nu(x)=\frac{x^\nu}{2^{\nu+1}}\int_0^\infty \e^{-t-\frac{x^2}{4t}}t^{-\nu-1}\d t.
\end{align}
Moreover, for $\nu\in (0,\infty)$, $\widehat{K}_\nu$ can be continuously extended to $x=0$:
\begin{align}\label{def:hK0}
\widehat{K}_\nu(0)=\lim_{x\searrow 0}\widehat{K}_\nu(x)=2^{\nu-1} \Gamma (\nu).
\end{align}
See \cite[(5.10.25) on p.119]{Lebedev} for \eqref{def:K}, and \eqref{asymp:K0} for more details of \eqref{def:hK0}. 
\item For $\alpha\in (0,1)$, set
\begin{align}\label{choice}
C_\alpha^\star\,\defeq \, \pi \cdot \frac{2^{1-\alpha}}{\Gamma(\alpha)},\quad \Lambda_\alpha=\Lambda_\alpha(\beta)\,\defeq\, C^\star_\alpha \beta^\alpha=\frac{\pi}{ \widehat{K}_\alpha(0)}\beta^\alpha,
\end{align}
and let $\{L_t\}$ denote the diffusion local time of $\{\rho_t\}$ at zero such that 
\begin{align}\label{def:Lt}
\E^{(-\alpha)}_0\left[\int_0^\infty \e^{-q t}\d L_t\right]=\frac{1}{C^\star_\alpha q^\alpha},\quad \forall\;q\in(0,\infty).
\end{align} 
See Section~\ref{sec:BES} for a collection of properties of the diffusion local times of Bessel processes, especially \eqref{eq:LTLT} where the case of $x_0=0$ recovers \eqref{def:Lt} by taking $C'_\alpha=C^\star_\alpha$. 
\end{itemize}
See
 Watanabe~\cite{Watanabe:TimeInversion} and Pitman and Yor~\cite{PY:BESINF} for the earlier studies of transformations of Bessel processes on which \eqref{def:BESab} is based. A closely related discussion will appear in Remark~\ref{rmk:beta}.

\begin{rmk}
The probability measure of $\BES(-\alpha,\beta\da)$ in \cite[(2.1) on p.882]{DY:Krein} and 
 the probability measure in \eqref{def:BESab} are the same, although the former 
is defined with a different local time and a different definition of $\Lambda_\alpha(\beta)$, which we denote by $\{L^{\rm DY}_t\}$ and $\Lambda_\alpha^{\rm DY}(\beta)$, respectively. 
 This identity of probability measures follows since
\begin{align}\label{LT:scaling}
\Lambda_\alpha(\beta) L_t=\Lambda_\alpha^{\rm DY}(\beta) L^{\rm DY}_t.
\end{align}
To see \eqref{LT:scaling}, note that
\[
C_\alpha^{\rm DY}\;\defeq\;\frac{\pi 2^\alpha}{\sin (\pi \alpha)},\quad 
\Lambda_\alpha^{\rm DY}(\beta)\,\defeq\, C_{\alpha}^{\rm DY}\beta^\alpha,
\]
and the local time $\{L^{\rm DY}_t\}$ satisfies 
\eqref{def:Lt} with $(L_t,C_\alpha^\star)$ replaced by $(L^{\rm DY}_t,C_\alpha^{\rm DY})$. See \cite[(2.2) on p.882 and the formula of $C_\alpha$ on p.883]{DY:Krein}. Identity \eqref{LT:scaling} can also be seen by using Proposition~\ref{prop:key}. \qed 
\end{rmk}

{\bf In the remaining of this paper, we only work with the local time $\{L_t\}$ satisfying (\ref{def:Lt})}. This normalization is a natural choice for our arguments. See Section~\ref{sec:DBG} for the applications. 

\subsection{First stochastic path integral representation}\label{sec:firstdr}
\begin{figure}
\begin{center}
\begin{tikzcd}[scale cd=.8, column sep=100pt,row sep=80pt]
\mbox{Radial part in $\{P^\beta_t\}$}
\arrow[drr, bend left, leftarrow, "\substack{\mbox{Lower-dimensional}\\ 
\mbox{approximations}\\
\mbox{(Theorem~\ref{thm:3})}}
" {anchor=north, xshift=-1.2cm}, "\mbox{\Large $\alpha
\searrow\, 0$}" {anchor=south, xshift=.5cm}]
\arrow[ddr, leftrightarrow, bend right, "\substack{\mbox{Feynman--Kac-type formula}\\ \mbox{(Theorem~\ref{thm:1})}}" description]
\arrow[dr, leftarrow, "\mbox{}" {anchor=north, xshift=-1.2cm}, "{\substack{\mbox{Approximations via}\\ \mbox{radial mollifications \cite{C:DBG0}}}}" description 
] & & \\
&\quad\quad  \BES(0)\hspace{-.2cm} \arrow[r, leftarrow, "\substack{\mbox{Convergences as a standard}\\ \mbox{property of Bessel processes}}" {anchor=north}, "\mbox{\Large $\alpha
\searrow\, 0$}" {anchor=south}] \arrow[d, leftrightarrow, dotted, "\substack{\mbox{No absolute continuity }\\ \mbox{between them}}" description]
&\quad  \BES(-\alpha) \arrow[d, leftrightarrow, "\substack{\mbox{Esscher--Girsanov}\\ \mbox{transformation (\cite{DY:Krein}}\\ \mbox{and Proposition~\ref{prop:key})}}" description] \\
&\quad \;\; \BES(0,\beta\da) \hspace{-.4cm}
&\quad \quad  \BES(-\alpha,\beta\da) \arrow[l, "\substack{\;\;\mbox{ 
Convergences as}\\ \mbox{\;\; diffusions \cite{DY:Krein} or as}
\\ \mbox{\;\; SDEs (Theorem~\ref{thm:2})}
}", "\mbox{\Large $\alpha
\searrow\, 0$}" {anchor=south}]
\end{tikzcd}
\end{center}
\caption{Relationships among the main radial objects in this paper}
\label{fig:BESab}
\vspace{-.2cm}
\end{figure}
Up to this point, the heuristic picture can be summarized as a chain of relationships among several objects: the radial part of $\{P^\beta_t\}$, $\BES(-\alpha)$, $\BES(-\alpha,\beta\da)$, and finally, $\BES(0,\beta\da )$. Recall \eqref{def:DBG} for $\{P^\beta_t\}$, and see Figure~\ref{fig:BESab} for a diagram including these relationships. Specifically, the first relationship we propose is that for all radial functions $f$, as $\alpha\searrow 0$,
\begin{align}\label{conv:BESab1}
P^\beta_tf(z_0)\simeq  \E^{(-\alpha)}_{\lvert z_0 \rvert}[\e^{\Lambda_\alpha L_t}f( \rho_t)],\quad \forall\; z_0\in \Bbb C\setminus\{0\}.
\end{align}
Here, the approximation ``$\simeq$'' is up to a change of multiplicative constants for space and time variables for the right-hand side, since the ``non-It\^{o}'' Laplacian $\Delta_z$, rather than $\Delta_z/2$, defines the relative motion in \eqref{def:L}. The other relationships are from \eqref{def:BESab}:
\begin{align}
\lim_{\alpha\searrow 0}\E^{(-\alpha)}_{\lvert z_0 \rvert}[\e^{\Lambda_\alpha L_t}f(\rho_t)]&=\lim_{\alpha\searrow 0}\E^{\albe}_{\lvert z_0 \rvert}\left[\frac{\e^{\beta t}\widehat{K}_\alpha(\sqrt{2\beta}\lvert z_0 \rvert)}{\widehat{K}_{\alpha}(\sqrt{2\beta}\rho_t)}f( \rho_t)\right]\notag\\
&=\E^{(0),\beta\da }_{\lvert z_0 \rvert}\left[\frac{\e^{\beta t} K_0(\sqrt{2\beta}\lvert z_0 \rvert)}{K_0(\sqrt{2\beta}\rho_t)}f( \rho_t)\right],\label{conv:BESab2}
\end{align}
where $\P^{\zbe}$ denotes the probability measure of $\BES(0,\beta\da )$. To verify these relationships, we cannot expect a proof of \eqref{conv:BESab1} by weak convergences as $\alpha\searrow 0$ since $0$ is polar under $\P^{(0)}$. This difficulty is similar to the one from dealing with the Brownian approximations via \eqref{FK:approx}. Verifying the last equality by the Portmanteau theorem is another issue mainly because $K_0$ decays exponentially to zero at infinity, but this issue is much milder than the ``wild singularity'' of local times that \eqref{conv:BESab1} is faced with. Note that $K_0$ also has a logarithmic singularity at $0$. See \eqref{asymp:K0} and \eqref{asymp:Kinfty} for the asymptotic representations of $K_0$ at 0 and at infinity. Accordingly, we consider the following outline:
 
\begin{outline}\label{outline:main}
The main results of this paper are developed in two different directions:
\begin{itemize}[leftmargin=3\labelsep]
\item A direct verification that $P_t^\beta f(z_0)$ for any radial function $f$ coincides with the limiting expectation in \eqref{conv:BESab2} up to suitable multiplicative constants for space and time variables, and an extension to $z_0=0$ and general, not necessarily radial, functions $f$. 
\item A second proof by two steps:

\begin{itemize}[leftmargin=3\labelsep]
\item {The validity of \eqref{conv:BESab2} by proving the convergence of $\BES(-\alpha,\beta\da)$ to $\BES(0,\beta\da )$ via the weak convergence of the SDEs.} 
\item {The validity of \eqref{conv:BESab1} and an extension to $z_0=0$ and general $f$, which takes the form of ``lower-dimensional approximations.'' \qed 
} 
\end{itemize}
\end{itemize}
\end{outline}

More specifically, the direct verification is to apply the original characterization from \cite{DY:Krein} of $\BES(0,\beta\da)$ by excursion theory. See Section~\ref{sec:EXCURSION} for the proof. 
 Note that the one-dimensional nature of this characterization is shared by the construction of $\BES(0,\beta\da )$ in \cite{DY:Krein} in the proofs of Theorems 2.1 and 2.2 \emph{ibid}. It specifies explicit formulas of the scale function and speed measure of $\BES(0,\beta\da )$. Therefore, applying the corresponding time change of a one-dimensional Brownian motion and scale transformation gives $\BES(0,\beta\da )$. 

The second proof will be summarized as Theorem~\ref{thm:4}. Along the way, we will unfold several new properties of 
 $\BES(-\alpha,\beta\da )$, and we will introduce a new ``lower-dimensional approximation'' of the relative motion operator $\ms L$.
See Sections~\ref{sec:SRMSDE} and \ref{sec:lda} for the overviews. This approximation method may be of independent interest for generalizations to other models with delta-function interactions. It not only provides a new construction of two-dimensional delta-function potentials by approximations of local times from lower-dimensional Bessel processes. It is also compatible with probabilistic purposes.

\begin{rmk}[Many-body delta-Bose gas]\label{rmk:DBG3+}
One of our earlier motivations of the second proof was to prepare the extension to a Feynman--Kac-type formula for the many-body delta-Bose gas \eqref{def:H}. We expected that the method proposed below would be feasible: 
\begin{quote}
\footnotesize
Roughly speaking, the extension begins by coupling the distance dynamics among all of the $N$ particles as a multi-dimensional SDE for $N\geq 3$. The driving Brownian motions are appropriately chosen such that in the solution of this SDE, say
$\{\rho_{\bi}(t);\bi=(i\prime,i),1\leq i<i\prime\leq N \}$ under $\P^{N,\al}$,
each $\rho_{\bi}$ is a version of $\BES(-\alpha)$. Moreover, these solutions converge to distances among $N$ independent planar Brownian motions as $\alpha\searrow 0$. 
Then the solution allows an extension of \eqref{def:BESab} as the following probability measure on $\bigvee_{\bi}\sigma(\rho_{\bi}(s);s\leq t)$:
\begin{align}\label{def:BESabN}
\d \P^{N,\albe}\,\defeq\,\prod_{1\leq i<i\prime\leq N}\frac{\widehat{K}_\alpha(\sqrt{2\beta}\rho_{\bi}(t))}{\widehat{K}_\alpha(\sqrt{2\beta}\rho_{\bi}(0))}\e^{\sum_{1\leq i<i\prime\leq N}[\Lambda_\alpha(\beta) L_\bi(t)-\beta t]+A_t}\d \P^{N,\al}.
\end{align}
Here, $L_\bi(t)$ denotes the local time of $\{\rho_{\bi}(t)\}$ at the point $0$, and $\{A_t\}$ is a correction term defined by a sum of quadratic covariations. More specifically, the formula of $\{A_t\}$ can be read off from the explicit stochastic logarithms of the following martingales:
\[
\frac{\widehat{K}_\alpha(\sqrt{2\beta}\rho_{\bi}(0))}{\widehat{K}_\alpha(\sqrt{2\beta}\rho_{\bi}(t))}\e^{\Lambda_\alpha(\beta) L_\bi(t)-\beta t}.
\]
See Proposition~\ref{prop:key} for this stochastic logarithm in the case of two particles. 
Among several things to be handled under this setting, the convergence of $\{\rho_{\bi}(t)\}$ under $\d \P^{N,\albe}$ will again be approached by considering the SDE dynamics; proving the analogue of \eqref{conv:BESab1} can use an extension of the diagrammatic expansion in \cite{C:DBG0} which constructs the many-body Schr\"odinger semigroups by mollifying the delta-function potentials. 
\end{quote}

After the submission of this paper, a Feynman--Kac-type formula for the many-body delta-Bose gas is obtained in \cite{C:SBG3+}. That formula uses a formulation different from the one implied by the approximation method proposed above, but the proof in \cite{C:SBG3+} still builds on some of the results and methods from the second proof in Outline~\ref{outline:main}. Presently, we do not have a definitive answer of whether the approximation method of using \eqref{def:BESabN} is also feasible.
 \qed 
\end{rmk}

\subsection{The stochastic relative motion}\label{sec:aug}
The complete choice of stochastic relative motion emerges from 
the extension to $z_0=0$ and general $f$ in both directions of Outline~\ref{outline:main}. 
This choice will be motivated in part by the fact that $\BES(0,\beta\da)$ relates to $\BES(0)$ via the following identity \cite[(2.4) on p.883]{DY:Krein}:
\begin{align}\label{T0-1}
\begin{split}
\P^{(0),\beta\da }_{x_0}(\{t<T_0\}\cap A)&=\E^{(0)}_{x_0}\left[\frac{\e^{-\beta t}K_0(\sqrt{2\beta}\rho_t)}{K_0(\sqrt{2\beta}x_0)};A\right],\;  \forall\; A\in \sigma(\rho_s;s\leq t),\;x_0>0,
\end{split}
\end{align}
where $T_0=T_0(\rho)$, $T_a(X)\defeq\inf\{t\geq 0;X_t=a\}$ for any $\{X_t\}$, and $\E[U;\!A]\,\defeq\,\E[U\1_A]$.

\begin{rmk}\label{rmk:beta}
The right-hand side of \eqref{T0-1} allows the following interpretation of the parameter $\beta$: for $0<t<\infty$ and an independent exponential $\mathbf e_\beta$ with mean $1/\beta$, 
\begin{align}\label{lim:interpretation}
\lim_{\vep\searrow 0}\P^{(0)}_{x_0}(A\cap \{t<T_\vep(\rho)\}|T_\vep(\rho)<\mathbf e_\beta)=\E^{(0)}_{x_0}\left[\frac{\e^{-\beta t}K_0(\sqrt{2\beta}\rho_t)}{K_0(\sqrt{2\beta}x_0)};A\right].
\end{align}
This approximation can be seen by  first noting that the left-hand side equals
\begin{align*}
\lim_{\vep\searrow 0}\frac{\E^{(0)}_{x_0}[\1_A\e^{-\beta T_{\vep}(\rho)};t<T_{\vep}(\rho)]}{ \E^{(0)}_{x_0}[\e^{-\beta T_{\vep}(\rho)}]}
=\lim_{\vep\searrow 0}\frac{\E^{(0)}_{x_0}[\1_A\e^{-\beta t}
\E^{(0)}_{\rho_{t}}[\e^{-\beta T_{\vep}(\rho)}];t<T_{\vep}(\rho)]}{\E^{(0)}_{x_0}[\e^{-\beta T_{\vep}(\rho)}]}.
\end{align*}
Then the right-hand side of \eqref{lim:interpretation} can be reached by 
 using the formula $\E^{(0)}_{x}[\e^{-q T_y(\rho)}]=K_0(\sqrt{2q}x)/K_0(\sqrt{2q}y)$ \cite[Eq.~(2.5)]{HM}, $0<y\leq x$, and the polarity of $0$ for $\BES(0)$.
See \cite[Section~3 on pp.299+]{PY:BESINF} for more on \eqref{lim:interpretation}. 

For general $A$, \eqref{lim:interpretation} without using $\{t<T_\vep(\rho)\}$ cannot hold since by using the probability density in \eqref{def:GIG} of $T_0(\rho)$, the left-hand side of \eqref{T0-1} depends on $t$ nontrivially, and hence, so does the right-hand side of \eqref{T0-1}. For $\alpha>0$, see \cite[Section~3.1]{DY:Krein} for a similar situation, but with a resolution.
\qed 
\end{rmk}

Now, we specify the stochastic relative motion $\{ \two Z_t\}$ mentioned in Section~\ref{sec:intro_overview}.

\begin{cond}\label{cond:sp}
We choose a complex-valued diffusion process $\{Z_t\}$ such that the following two conditions hold:
\begin{itemize}[leftmargin=3\labelsep]
\item [(1$\cc$)] The radial process $\{\lvert Z_t \rvert\}=\{\rho_t\}$ is a version of $\BES(0,\beta\da )$.  
\item [(2$\cc$)] For all $z_0\in \Bbb C\setminus\{0\}$, the law of $\{Z_t\}$ conditioned on $Z_0=z_0$ satisfies 
\begin{align}\label{Z:skewproduct}
Z_t=\lvert Z_t \rvert\exp\big\{\i \dot\gamma_{\int_0^t\d s/\lvert Z_s\rvert^2}\big\},\quad \forall\; t<T_0(Z),
\end{align}
where $\{\dot\gamma_t\}$ is a circular Brownian motion (i.e. a one-dimensional Brownian motion mod $2\pi$) independent of $\{\lvert Z_t \rvert\}$.\qed 
\end{itemize}
\end{cond}

In particular, the clock process  $\int_0^t\d s/\lvert Z_s\rvert^2$ of the angular part in \eqref{Z:skewproduct} takes the same form as the clock process in the skew-product representation of planar Brownian motion \cite[(2.11) Theorem on p.193]{RY}. Also, \eqref{T0-1} allows inverting $\{Z_t;t<T_0(Z)\}$ back to planar Brownian motion, so the above setting is \emph{heuristically} enough similar to the Brownian approximations discussed in Section~\ref{sec:intro} for $\{P^\beta_t\}$. Furthermore, we expect nothing else from $\{Z_t\}$ in the limit of $\alpha\searrow 0$ due to the functional analytic \emph{universality} suggested by \cite[Theorem~2.2]{AGHH:2D}. This is so because that theorem from \cite{AGHH:2D} constructs the Hamiltonian of the relative motion as \emph{all} the self-adjoint extensions of the Laplacian restricted to $\C_0^\infty(\Bbb C\setminus\{0\})$, the set of $\C^\infty$-functions with compact support in $\Bbb C\setminus\{0\}$. Here, the self-adjoint extensions from  \cite[Theorem~2.2]{AGHH:2D} range over exactly the Laplacian $\Delta_z$ and the operators $\Delta_\beta$ described before \eqref{def:DBG} for $\{P^\beta_t\}$. 

In this paper, we will work with $\{Z_t\}$ such that $\{\lvert Z_t \rvert\}$ is 
$\BES(-\alpha)$ and $\BES(-\alpha,\beta\da )$, $\alpha \in [0,1/2)$, by requiring Condition~\ref{cond:sp} and its analogues for the other cases. Here, $\{Z_t\}$ for $\{|Z_t|\}\sim \BES(0)$ is the two-dimensional Brownian motion.
Except this case, the \emph{weak existence and uniqueness of $\{Z_t\}$} for $t$ ranging over the \emph{entire} $ \R_+$ is a particular consequence of Erickson's continuous extensions of skew-product diffusions \cite[Theorem~1 on pp.75--76]{Erickson}. The main condition thus required is
\begin{align}\label{Zt:explosion}
\textstyle \int_0^{T_0(Z)} \d s/\lvert Z_s\rvert^2=\infty\quad\mbox{a.s.}
\end{align}
 Informally, \eqref{Zt:explosion} has the meaning of driving ``rapid spinning'' of $\{Z_t\}$ via \eqref{Z:skewproduct} as $t\nearrow T_0(Z)$. This leads to the use of the equilibrium distribution of the circular Brownian motion as $t\nearrow T_0(Z)$, and so, to the weak uniqueness of $\{Z_t\}$. See also It\^{o} and McKean~\cite[§7.16 on pp.274+]{IM:Diffusion} on the spinning of skew-product diffusions.

\begin{prop}\label{prop:esp}
For any $\alpha \in (0,1/2)$, there exists a $\Bbb C$-valued continuous strong Markov process $\{Z_t\}_{t\geq 0}$ such that $\{\lvert Z_t \rvert\}\sim \BES(-\alpha)$, and  Condition~\ref{cond:sp} with $\BES(0,\beta\da)$ replaced by $\BES(-\alpha)$ holds. The analogous existence result holds if the radial process is distributed as $\BES(-\alpha,\beta\da)$, for any $\alpha\in [0,1/2)$. In any of these cases, the process $\{Z_t\}$ is uniqueness in law under the corresponding conditions.
\end{prop}

The proof of this proposition will be given in Section~\ref{sec:esp}.

\begin{nota}\label{nota:Z}
We will continue to denote the probability measure for $\{Z_t\}$ by $\P^{(-\alpha)}$ or $\P^{\albe}$ if its radial part is distributed as $\BES(-\alpha)$ or $\BES(-\alpha,\beta\da)$. Also, we denote complex states with the letter $z$ and real states with the letters $x$, $y$, $a$ and $b$. For example, $\P^{(-\alpha)}_{z_0}$ means that $\{Z_t\}$ starts at $z_0$, and $\P^{(-\alpha)}_{x_0}$ only refers to $\BES(-\alpha)$ starting from $x_0$.\qed
\end{nota}

The next two propositions specify some of the basic properties of the stochastic relative motion. See Section~\ref{sec:singular} for the proofs.

\begin{prop}
\label{prop:GEN}
Given any $\beta\in (0,\infty)$, the infinitesimal generator of $\{\two Z_t\}$ under $\P^{\zbe}_{z/\two}$ is given by $\ms A^{\beta\da}_0$ in \eqref{def:gen}. Also, the following decomposition of the two-dimensional Laplacian $\Delta_z$ holds: for all $f\in \C^\infty$ and $z=x+\i y\neq 0$, 
\[
\Delta_z f(x,y)=K_0(\sqrt{\beta }\lvert z\rvert)\ms A^{\beta\da}_0 \left\{(\tilde{x},\tilde{y})\mapsto \frac{f(\tilde{x},\tilde{y})}{K_0(\sqrt{\beta}\lvert\tilde{z}\rvert)} \right\}(x,y)+\beta f(x,y)\quad(\tilde{z}=\tilde{x}+\i \tilde{y}).
\]
\end{prop}

\begin{prop}\label{prop:singular}
For $t_1\in (0,\infty)$, the law of $\{Z_t\}_{t\in [0,t_1]}$ has a nonzero singular part in the Lebesgue decomposition with respect to the law of planar Brownian motion restricted to $[0,t_1]$. Moreover, the law of $\{Z_t\}_{t\in [0,\infty)}$ is singular to the law of planar Brownian motion.
\end{prop}

Proposition~\ref{prop:singular} shows that by refining \eqref{T0-1},
the stochastic relative motion is non-Gaussian to the degree that it cannot become a Gaussian process after some Girsanov transformation, since such a Gaussian process could only be a planar Brownian motion. 

In the following, we will proceed to the first theorem of this paper and then resume the discussion of the stochastic relative motion in Section~\ref{sec:Markovgen} for the analytical properties of the drift coefficient of $\ms A^{\beta\da}_0$.

\subsection{The Feynman--Kac-type formula for the relative motion}\label{sec:FKformula}
We are ready to state the first main theorem of this paper, which proves a Feynman--Kac-type formula of $\{P^\beta_t\}$. For the statement, we use the following explicit form of the kernel $P^\beta_t(z,\tilde{z})$ from inverting the Laplace transforms in \eqref{def:DBG}: 
\begin{align}
\begin{split}
P_t^\beta (z,\tilde{z})
=P_{2t}(z,\tilde{z}) + \int_0^t\d sP_{2s}(z)\mathring{P}_{t-s}^\beta(\tilde{z}),\quad z,\tilde{z}\in \Bbb C,
\end{split}
\label{def:Pbeta2}
\end{align}
where $\mathring{P}^\beta_t$ are linear functionals on $\B_+(\Bbb C)$ with density
\begin{align}\label{def:ringP}
\mathring{P}^\beta_t(z)\,\defeq\, \int_0^{t}\d  \tau \s^\beta(\tau) P_{2(t-\tau)}(z),
\end{align}
and we set
\begin{align}\label{def:s}
 \mathfrak s^\beta(\tau)\,\defeq \, 4\pi\int_0^\infty \d u\frac{\beta^u \tau^{u-1}}{\Gamma(u)}.
\end{align}
Here, \eqref{def:Pbeta2} holds because of
the formula 
\begin{align}\label{Inv:sbeta}
\int_0^\infty \e^{-q \tau}\left(4\pi\int_0^\infty \frac{\beta^u\tau^{u-1}}{\Gamma(u)}\d u\right) \d \tau=\frac{4\pi}{\log (q/\beta)},\quad q\in (\beta,\infty)
\end{align}
\cite[Eq. (5.6) on p.176]{C:DBG3+}. 
More details of \eqref{def:Pbeta2} can be found in \cite[the proof of Proposition~5.1 on p.176]{C:DBG3+} and Remark~\ref{rmk:lap}. Note that another formula for $P^\beta_t(z,\tilde{z})$ by inverting the Laplace transform in \eqref{def:DBG} appears in \cite[Eq.~(3.11)]{ABD:Schrodinger}. 

We will give two different proofs of the following theorem, as in Outline~\ref{outline:main}. 

\begin{thm}\label{thm:1}
The semigroup $\{P^\beta_t\}$ defined by \eqref{def:Pbeta2} and the linear functionals $\{\mathring{P}^\beta_t\}$ with density in \eqref{def:ringP} admit the following representations: 
\begin{align}
P^\beta_t f(z_0)&=\E^{(0),\beta\da }_{z_0/\two}\left[\frac{\e^{\beta t}K_0(\sqrt{\beta}\lvert z_0 \rvert)}{K_0(\sqrt{\beta}\lvert \two Z_t\rvert)}f(\two Z_t)\right],\quad z_0\in \Bbb C\setminus\{0\} ,\label{id:2-1}\\
\mathring{P}^\beta_t f&=\E^{(0),\beta\da }_0\left[\frac{\e^{\beta t}\cdot 2\pi}{K_0(\sqrt{\beta}\lvert \two Z_t\rvert)}f(\two Z_t)\right]\label{id:2-2}
\end{align}
for all $f\in \B_+(\Bbb C)$.
Equivalently, the marginals of $\{Z_t\}$ under $\P_{z_0}^{(0),\beta\da }$ are given by
\begin{align}\label{BESab:density}
& \E^{(0),\beta\da}_{z_0}\left[f(Z_t)\right]\notag\\
& =
\begin{cases}
\displaystyle \frac{\e^{-\beta t}P_{2t}f_\beta(\two z_0)}{K_0(\sqrt{\beta}\lvert \two z_0 \rvert)}\\
\displaystyle  +\frac{ \e^{-\beta t}}{K_0(\sqrt{\beta}\lvert \two z_0 \rvert)}\int_0^t \d s P_{2s}(\two z_0)\int_0^{t-s}\d \tau \s^\beta(\tau)P_{2(t-s-\tau)}f_\beta(0),&z_0\neq 0,\\
\vspace{-.2cm}\\
\displaystyle \frac{\e^{-\beta t}}{2\pi}\int_0^{t}\d \tau \s^\beta(\tau)P_{2(t-\tau)}f_\beta(0),& z_0=0,
\end{cases}
\end{align}
where $f_\beta(z)\,\defeq \,f(z/\two)K_0(\sqrt{\beta}\lvert  z\rvert)$ for $z\in \Bbb C$.
\end{thm}

To contrast the formula in Theorem~\ref{thm:1} with the classical exponential form of the Feynman--Kac formula, let us introduce a corollary. For the statement, we work with a slightly different setting. This uses $\{Z_t\}$ whenever appropriate (only to avoid the cumbersome $\two$ here and there), and set up the following hypothesis that assumes an additive functional $\{A_t\}$ subject to a minimal measurability condition and allows a killing time $\zeta$: 

\begin{hypothesis}\label{hyp:FK1}
There exists $(\Bbb Q,A,\zeta)$ such that all of the following conditions hold:

\begin{itemize}[leftmargin=3\labelsep]
\item $\Bbb Q$ is a probability measure on the Borel $\sigma$-field of $C(\R_+,\Bbb C)$ under which $\{Z_t\}$ is understood as the coordinate process (i.e. $Z_t(\mathrm w)=\mathrm w_t$) and is a Markov process. We stress that the strong Markov property of $\{Z_t\}$ under $\Bbb Q$ is \emph{not} imposed. 
\item $(t,\mathrm w)\mapsto A_t(\mathrm w_r;r\leq t):\R_+\times C(\R_+,\Bbb C)\to (-\infty,\infty]$  is Borel measurable such that $t\mapsto A_t(\mathrm w_r;r\leq t)$ is c\`adl\`ag for all $\mathrm w$, and if $\theta_t(\mathrm w)_r\,\defeq\,\mathrm w_{t+r}$ denotes the shift operator,
\begin{align}\label{def:A}
A_{t+s}(\mathrm w_r;r\leq t+s)-A_t(\mathrm w_r;r\leq t)=A_s((\mathrm w\circ \theta_t)_r;r\leq s).
\end{align}
In other words, $A_\cdot$ is an additive functional.
\item $\zeta(\mathrm w)$ is a stopping time with respect to $\F_t\,\defeq\,\sigma(Z_s;s\leq t)$.
\item The following {\bf exponential form} of the Feynman--Kac formula holds:
\begin{align}\label{FK:char}
\begin{split}
&\E_{z_0}^\Bbb Q[\e^{A_t}\phi(Z_t);\zeta>t]=\E_{z_0}^{(0),\beta\da}\left[D^{z_0}_t\phi(Z_t)\right],\\
&\hspace{2cm} \forall\;z_0\in \Bbb C\setminus\{0\},\;t\in (0,\infty),\;\phi\in \B_+(\Bbb C).
\end{split}
\end{align}
Here in \eqref{FK:char}, we use the shorthand notation:
\begin{align}\label{def:AD}
A_t\,\defeq\,A_t(Z_s;s\leq t),\quad D^z_t\,\defeq\,\frac{\e^{\beta t} K_0(\sqrt{2\beta}\lvert z\rvert)}{K_0(\sqrt{2\beta}\lvert Z_t \rvert)},
\end{align}
and $\E[U;\!A]\,\defeq\,\E[U\1_A]$.\qed
\end{itemize}
\end{hypothesis}

For \eqref{FK:char}, the assumption that the additive functional $A_t$ is $(-\infty,\infty]$-valued gives $\e^{A_t}\neq 0$.
The following corollary refutes Hypothesis~\ref{hyp:FK1}. See Section~\ref{sec:FK} for the proof.

\begin{cor}
\label{cor:FK}
For all $\beta\in (0,\infty)$, no triplet $(\Bbb Q,A,\zeta)$ in Hypothesis~\ref{hyp:FK1} exists. 
\end{cor}

\subsection{Second look at the stochastic relative motion: regularity of the drift}\label{sec:Markovgen}
We now resume the discussion of the stochastic relative motion for the regularity of the drift. The ill-behaved analytical properties in the following proposition may clarify the necessity of the skew-product construction in Section~\ref{sec:firstdr} in comparison with general theories of SDEs and Markov processes. See Section~\ref{sec:singular} for the proof.

\begin{prop}\label{prop:DBM}
We have the following properties of $b_\beta$ in \eqref{def:bbeta}, where $\psi\in \C_{+}(\Bbb C)$ takes values in $[0,1]$ with $\psi(0)=1$, $p\in [1,\infty]$, $\delta\in (0,1)$, and $s\in \R$, and ${\rm Leb}$ is the Lebesgue measure. Recall the convention of dot products and differentiations stated below \eqref{def:bbeta}.  
\smallskip 
\begin{enumerate}[label={\rm ({\arabic*}$\cc$)}]
\item For some constant $C(\beta)\geq 1$,
\begin{gather}
 C(\beta)^{-1}/(|z|\log |z|^{-1})\leq |b_\beta(z)|\leq C(\beta)/(|z|\log |z|^{-1}),\quad \forall\;0<|z|\leq 1/2,\label{asymp0:drift}\\
  C(\beta)^{-1}/|z|\leq |K_0(\sqrt{\beta}|z|)b_\beta(z)|\leq C(\beta)/|z|,\quad \forall\;0<|z|\leq 1/2.\label{asymp0:drift1}
 \end{gather}
 Also, the divergence of $b_\beta$ satisfies:
\begin{align}\label{asymp:drift}
\begin{split}
\nabla_z\cdot b_\beta( z)&=\Delta_z\log K_0(\sqrt{\beta}|z|)^2=2\beta-2\beta\left(\frac{K_1}{K_0}\right)^2(\sqrt{\beta}|z|),\\ 
\nabla_z\cdot b_\beta( z)
&\sim \frac{-1}{(|z|\log |z|)^2},\quad |z|\to 0.
\end{split}
\end{align}
\smallskip 
\item $b_\beta\psi\in L^{p}$ if and only if $p\in[1,2]$, and $K_0(\sqrt{\beta}|\cdot|)b_\beta \psi\in L^p$ if and only if $p\in [1,2)$. \smallskip 
\item Suppose that $\supp(\psi)$ is a bounded set. Then
$b_\beta\psi\in L^{p,\infty}$ if and only if $p\in[1,2]$, where $L^{p,\infty}$ denotes the weak $L^p$ space with finite quasi-norm defined by
\[
\|f\|_{L^{p,\infty}}\defeq\, \sup_{\ell>0}\ell\cdot {\rm Leb}\{z:|f(z)|>\ell\}^{1/p}.
\]
\item 
$\Re (b_\beta \psi)\in B^{s}_{\infty,\infty}$ if and only if $s\leq -1$, where $B^s_{\infty,\infty}$ is the nonhomogeneous Besov space \cite[p.99]{BCD:FPDE}. The same statement with $\Re(b_\beta \psi)$ replaced by $\Im(b_\beta\psi)$ holds.  
\end{enumerate}
\end{prop}

Proposition~\ref{prop:DBM} (2$\cc$) extends to a Ladyzhenskaya--Prodi--Serrin (LPS)
type condition on the value of $d/p+2/q$ for $b(z,t)\in L^q([0,T];L^p_{\rm loc}(\R^d))$. Such a condition was originally from the literature of fluid dynamics and was introduced to study SDEs with singular drift in \cite{KrRo}. The following also considers the two-body dynamics of the two-body delta-Bose gas.

\begin{cor}
The coefficient $b_\beta(z,t)\,\defeq\,b_\beta(z)$ is {\bf critical} in the sense of the following LPS type condition:
\begin{align}\label{LPS2}
\left\{\frac{2}{p}+\frac{2}{q}; b_\beta(z,t)\in L^q([0,T];L^p_{\rm loc}(\Bbb C))\mbox{ for }\,p,q\in [1,\infty]\right\}=[1,4],
\end{align}
where $0<T<\infty$, and the minimal value $1$ in \eqref{LPS2} is attained only at $(p,q)=(2,\infty)$. Also,
for an independent two-dimensional standard Brownian motion $\{B_t\}$, the drift coefficient of $\{(Z_t+B_t)/\two,(Z_t-B_t)/\two)\}$ is {\bf supercritical} in the sense that
\begin{align}\label{LPS4}
\left\{\frac{4}{p}+\frac{2}{q};
\begin{bmatrix}
b_\beta(z'-z,t)/2\\
b_\beta(z-z',t)/2
\end{bmatrix}
\in L^q([0,T];L^p_{\rm loc}(\Bbb C^2))\mbox{ for }p,q\in [1,\infty]\right\}=[2,6],
\end{align}
and the minimal value $2$ in \eqref{LPS4} is attained only at $(p,q)=(2,\infty)$.
\end{cor}
\begin{proof}
By Proposition~\ref{prop:DBM} (2$\cc$), \eqref{LPS2} follows immediately. To see \eqref{LPS4}, we note that $(z',z)\mapsto (z'-z,z'+z)$ is a homeomorphism, and so, the $L^p_{\rm loc}$-integrability of $(z',z)\mapsto (b_\beta(z'-z),b_\beta(z-z'))$ is equivalent to the $L^p_{\rm loc}$-integrability of $z\mapsto b_\beta(z)$. Hence, \eqref{LPS4} is also implied by Proposition~\ref{prop:DBM} (2$\cc$).
\end{proof}

\subsection{Second look at the stochastic relative position}\label{sec:SRMSDE}
The first step of the second proof in Outline~\ref{outline:main} can be completed by proving some properties of the processes $\BES(-\alpha,\beta\da)$ as SDEs and the clock processes $\{\int_0^t\d s/\rho_s^{2}\}$. These properties are summarized in the next theorem. 

\begin{thm}\label{thm:2}
Fix $x_0\in\R_+$ and $\beta\in (0,\infty)$.
\begin{enumerate}[label={\rm ({\arabic*}$\cc$)}]
\item For all $\alpha\in [0,1/2)$, $\{\rho_t^2\}$ under $\P^{\albe}_{x_0}$ obeys the following SDE with $X_t=\rho_t^2$:
\begin{align}\label{SDE:DMY}
\begin{split}
X_t&=x_0^2+\int_0^t \mu^{\beta\da }_\alpha(X_s) \d s +2\int_0^t \lvert X_s\rvert^{1/2}\d \widetilde{W}_s,
\end{split}
\end{align}
where $\{\widetilde{W}_t\}$ is a one-dimensional standard Brownian motion, $\mu^{\beta\da}_\alpha\in \C(\R)$ is defined by 
\begin{align}\label{def:mu}
\mu^{\beta\da }_\alpha(x)\,\defeq \;2\left(1-\alpha-\sqrt{2\beta\lvert x\rvert}\frac{K_{1-\alpha}}{K_\alpha}(\sqrt{2\beta \lvert x\rvert})\right),\quad x\in \R,
\end{align}
with $0(K_{1-\alpha}/K_\alpha)(0)=0$ (Lemma~\ref{lem:lim0K0}),
and $K_\nu$ is the Macdonald function \eqref{def:K}. \smallskip

\item For all $\alpha\in [0,1/2)$, both nonnegativity of solutions and pathwise uniqueness in \eqref{SDE:DMY} hold.
Moreover, pathwise comparison of solutions in \eqref{SDE:DMY} subject to the same Brownian motion holds: if $\{X_\alpha(t)\}$ denotes the solution, then $\P(X_{\alpha_1}(t)\geq X_{\alpha_2}(t),\;\forall\;t\geq 0)=1$, whenever $0\leq \alpha_1\leq \alpha_2<1/2$ and $X_{\alpha_1}(0)\geq X_{\alpha_2}(0)$. \\

\item For all $\alpha\in [0,1/2)$, $\{\rho_t\}$ under $\P^{\albe}_{x_0}$ obeys the following SDE:
\begin{align}\label{SDE:DMYsqrt}
\rho_t=x_0+\int_0^t \left(\frac{1-2\alpha}{2\rho_s}-\sqrt{2\beta}\frac{K_{1-\alpha}}{K_\alpha}(\sqrt{2\beta}\rho_s)\right)\d s+\widetilde{W}_t,
\end{align}
where $\int_0^t \d s/\rho_s$ and $\int_0^t (K_{1-\alpha}/K_\alpha)(\sqrt{2\beta}\rho_s)\d s$ are finite for all $t$ almost surely. \smallskip

\item Suppose that $x_0>0$. Then for all $\alpha\in [0,1/2)$, $\{(\rho_t,\int_0^t \d s/\rho_s^2)\}$ under $\P^{\albe}_{x_0}$ is continuous as a process taking values in $\R_+\times [0,\infty]$ such that $\int_0^{T_0(\rho)}\d s/\rho_s^2=\infty$ and $T_0(\rho)<\infty$. Moreover, the family of probability distributions of $\{(\rho_t,\int_0^t \d s/\rho_s^2)\}$ under $\P^{\albe}_{x_0}$ is continuous at $\alpha=0$ with respect to the weak topology on the set of probability measures on $C(\R_+,\R_+\times [0,\infty])$.
\end{enumerate}
\end{thm}

The SDEs in \eqref{SDE:DMY} and \eqref{SDE:DMYsqrt} may be compared with the SDE of the square of $\BES(-\alpha)$, namely $\BES Q(-\alpha)$, and the SDE of $\BES(-\alpha)$. Here, recall that $\{\rho_t^2\}\sim\BES Q(-\alpha)$ and $\{\rho_t\}\sim \BES(-\alpha)$ obey the following two equations: 
\begin{align}
\d \rho_t^2&=2(1-\alpha) \d t+2\rho_t \d W_t,\label{def:BESQ}\\
\d\rho_t&=\frac{1-2\alpha}{2\rho_t}\d t+\d W_t,\label{def:BES}
\end{align}
for a one-dimensional standard Brownian motion $\{W_t\}$. Henceforth, $\BES Q(-\alpha,\beta\da)$ denotes the process $\{\rho_t^2\}$ under $\P^{\albe}$.
Note that the SDE \eqref{SDE:DMYsqrt} of $\BES (-\alpha,\beta\da)$ does not call for a local time to compensate for the boundary behavior at $0$. The reader may also compare \eqref{gen:BESab} with \eqref{SDE:DMYsqrt} for $\alpha=0$ to see that the latter is a ``trivial'' SDE realization.

We do not see the comparison of solutions and pathwise uniqueness in Theorem~\ref{thm:2} (2$\cc$) directly covered by the general theories of SDEs. The comparison requires the monotonicity property of $\alpha\mapsto \mu^{\beta\da}_{\alpha}(x)$, and one step of the proof of the pathwise uniqueness for $\alpha=0$ uses the monotonicity of $x\mapsto \mu^{\beta\da}_{\alpha}(x)$ (Lemma~\ref{lem:continuity}). But these coefficients involve the ratios $K_{1-\alpha}(x)/K_\alpha(x)$ that make the monotonicity properties \emph{analytically} not obvious, although {\sc Mathematica} plots give us some clues. See Figure~\ref{fig:drift}. We are not able to find these monotonicity properties in the literature of special functions either. As for the pathwise uniqueness, the drift coefficients $x\mapsto \mu^{\beta\da}_{\alpha}(x)$ fail to satisfy the concave Osgood condition of Yamada and Watanabe \cite[Theorem~1 on p.164]{YW:71} for pathwise uniqueness in SDEs. This fact is made precise in the following proposition, which may be viewed an extension of the ill-behaved analytic properties of $b_\beta$ in Proposition~\ref{prop:DBM}. See Section~\ref{sec:Osgood} for the proof.

\begin{figure}
\begin{center}
\includegraphics[width=12cm]{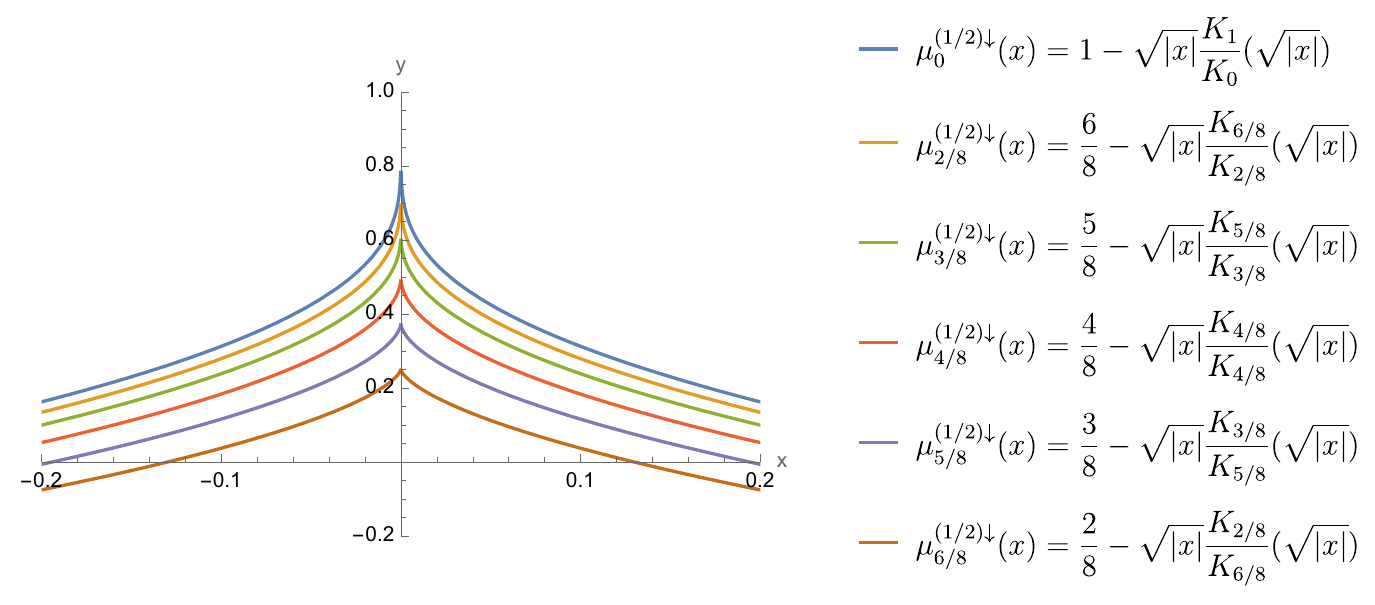}
\end{center}
\caption{A {\sc Mathematica} plot of the drift coefficients $\mu^{\beta\da}_\alpha$ in \eqref{SDE:DMY} for various values of $\alpha$ and $\beta=1/2$ --- the plot for $\mu^{(1/2)\da}_{0}(x)$ illustrates a very sharp cusp so that the value at $x=0$ should be understood as $1$}
\label{fig:drift}
\vspace{-.2cm}
\end{figure} 

\begin{prop}[Concave Osgood condition fails]\label{prop:osgood}
For all $\alpha\in [0,1/2)$, it holds that 
\begin{align}\label{Osgood}
|\mu^{\beta\da}_\alpha(y)-\mu^{\beta\da}_\alpha(x)|\leq \kappa_\alpha(2\beta |y-x|),\quad \forall\;x,y\in \R,
\end{align}
where, for $u\geq 0$,
\begin{align}\label{def:kappao}
\kappa_\alpha(u)\,\defeq\,
\begin{cases}
 C_{\ref{Osgood}}(\alpha)u^{\alpha}+ C_{\ref{Osgood}}(\alpha)u,&\alpha\in (0,1/2),\\
\displaystyle
-
 \frac{C_{\ref{Osgood}}}{\log (u\wedge \frac{\e^{-2}}{2})}+C_{\ref{Osgood}} u,&\alpha=0,
 \end{cases}
\end{align}
is an increasing, concave function defined on $\R_+$ such that $\kappa_\alpha(0)=0$, but 
\begin{align}\label{integral:kappa}
\textstyle \int_{0+}\kappa_\alpha(u)^{-1}\d u<\infty.
\end{align}
 Moreover, the orders of modulus of continuity in \eqref{Osgood} are sharp since
\begin{align}\label{Osgood1}
\begin{split}
&\lim_{x\searrow 0}|\mu^{\beta\da}_\alpha(x)-\mu^{\beta\da}_\alpha(0)|\cdot(2\beta x)^{-\alpha}= \frac{\Gamma(1+\alpha)\Gamma(1-\alpha)}{2^{2\alpha-2}\Gamma(\alpha)^2\alpha},\quad \alpha\in (0,1/2),\\
&\lim_{x\searrow 0}|\mu^{\beta\da}_0(x)-\mu^{\beta\da}_0(0)|\cdot |\log (2\beta x)|=4.
\end{split}
\end{align}
\end{prop}

The weak moduli of continuity of the coefficients $\mu^{\beta\da}_{\alpha}$ are due to the following functions:
\begin{align}\label{singular}
x\mapsto 2\sqrt{2\beta \lvert x\rvert}(K_{1-\alpha}/K_\alpha)(\sqrt{2\beta \lvert x\rvert}).
\end{align}
Another way to see the significance of these functions is to consider an attempt to convert $\BES Q(-\alpha,\beta\da)$ to $\BES Q(-\alpha)$. This attempt would aim at the removal of \eqref{singular} in $\mu^{\beta\da}_\alpha$, and so, by Girsanov's transformation, the use of the stochastic exponential of
\begin{align}\label{singular:INT}
-\sqrt{2\beta}\int_0^t \frac{K_{1-\alpha}}{K_\alpha}(\sqrt{2\beta}\rho_s)\d W_s.
\end{align}
See Proposition~\ref{prop:key} and \eqref{rep:Palbe} for the corresponding change of measures under $\P^{\al}$ when $\alpha\in (0,1/2)$. In this direction, note that for $\alpha\in [0,1/2)$, $x\mapsto -\sqrt{2\beta}(K_{1-\alpha}/K_\alpha)(\sqrt{2\beta}x)$ in the integrand of \eqref{singular:INT} blows up at $x=0$. For $\alpha=0$, this singularity is strong enough to forbid the global conversion of $\BES(0,\beta\da)$ to $\BES(0)$, as evidenced by \eqref{T0-1} and Proposition~\ref{prop:singular}.

The proof of Theorem~\ref{thm:2} is obtained in Section~\ref{sec:SDE}. The whole argument
depends heavily on special properties of the ratios of the Macdonald functions in $\mu^{\beta\da}_\alpha$ and the functions in \eqref{singular}, or more specifically, properties given by the various monotonicities and the uniform continuity at $\alpha= 0$. See Section~\ref{sec:RATIO}, especially
Propositions~\ref{prop:BESmon} and ~\ref{prop:Fxunif}.

\subsection{The lower-dimensional approximations}\label{sec:lda}
To carry out the second step of the second proof in Outline~\ref{outline:main}, we focus on the convergences of the Feynman--Kac semigroups defined by the right-hand side of \eqref{conv:BESab1} and their extensions $\E^{\al}_{z_0}[\e^{\Lambda_\alpha L_{t}}f(Z_{t})]$ as $\alpha\searrow 0$. The argument does not use $\BES(-\alpha,\beta\da)$ and essentially only operates at the expectation level.
The following proposition is a key step of the method (Theorem~\ref{thm:3}) and is taken from its proof.

\begin{prop}\label{prop:LTexplicit}
For all $\alpha\in (0,1/2)$, $0<\beta<q<\infty$, $f\in \B_+(\Bbb C)$ and $z_0\in \Bbb C$,
\begin{align}\label{resolvent:al}
\begin{split}
&\int_0^\infty \e^{-q t}\E^{\al}_{z_0}[\e^{\Lambda_\alpha L_{t}}f(Z_{t})]\d t \\
&\eqspace=U^{\al}_q f(z_0) +\frac{1}{\pi}\widehat{K}_{\alpha}(\sqrt{2q}\lvert z_0 \rvert)\cdot \frac{\pi\cdot  \frac{2^{1-\alpha}}{\Gamma(\alpha)}\cdot \beta^\alpha }{q^\alpha-\beta^\alpha}\cdot U^{\al}_q f(0),
\end{split}
\end{align}
where $\widehat{K}_\alpha$ is defined in \eqref{def:hK}, $\Lambda_\alpha=\Lambda_\alpha(\beta)$ is defined in \eqref{choice}, and
\begin{align}\label{def:resolvent}
U_q^{(-\alpha)} f(z)\,\defeq
\int_0^\infty \e^{-q t}\E^{\al}_{z}[f(Z_{t})]\d t.
\end{align}
Moreover,
$U^{\al}_q f(0)=\int_0^\infty \e^{-q t}\E_{0}^{\al}[\overline{f}(\rho_t)]\d t$, 
where $\overline{f}$ denotes the radialization of $f$:
\begin{align}\label{def:fbar}
\overline{f}(r)\,\defeq\, \frac{1}{2\pi}\int_{-\pi}^\pi f(r\e^{\i \theta})\d \theta.
\end{align}
See \eqref{density} for the transition densities of $\BES(-\alpha)$. 
\end{prop}
\begin{proof}
To get \eqref{resolvent:al}, combine \eqref{Lap:asymp}, \eqref{LT:norm} and \eqref{F0:exp1}. See \eqref{def:Falpha} for the definition of $F_\alpha$. Also, the alternative form of $U^{\al}_q f(0)$ by $\BES(-\alpha)$ restates Erickson's characterization \cite[Eq. (2.3) on p.75]{Erickson} specialized to the case of $\{Z_t\}$ under $\P^{\al}$.
\end{proof}

In particular, \eqref{resolvent:al} can be used to recover \eqref{def:DBG} by noting  
 that the special prefactor $4\pi/\log (q/\beta)$ arises explicitly from the following simple limit: 
\begin{align}\label{lim:log}
\frac{2\pi\cdot \frac{2^{1-\alpha}}{\Gamma(\alpha)}\cdot \beta^\alpha }{q^\alpha-\beta^\alpha}\xrightarrow[\alpha\searrow0]{}\frac{4\pi}{\log(q/\beta)}.
\end{align}
We remark that a calculation very similar to \eqref{lim:log} appears in \cite[pp.882--883]{DY:Krein}, but that calculation from \cite{DY:Krein} does not seem to be directly applicable for the present purpose.

\section{Excursions of the stochastic relative position}\label{sec:EXCURSION}
In this section, we give the first proof of Theorem~\ref{thm:1} by applying the original characterization of $\BES(0,\beta\da)$ in \cite{DY:Krein}, which uses excursion theory. The reader may wish to consult \cite[Chapter~IV]{Bertoin} and \cite{Blumenthal} for the excursion theory of general Markov processes. 

As before, let $\{\rho_t\}$ under $\P^{\zbe}$ denote a version of $\BES(0,\beta\da)$. Then the probability distribution of $\{\rho_t;t<T_0(\rho)\}$ can be characterized by \eqref{T0-1}, and we have
\begin{align}\label{T0-2}
\E^{(0),\beta\da }_{x_0}[\e^{-qT_0(\rho)}]=\frac{K_0(\sqrt{2(\beta+q)}x_0)}{K_0(\sqrt{2\beta}x_0)},\quad \forall\; q,x_0\in [0,\infty)
\end{align}
\cite[(2.9) on p.884]{DY:Krein}. [Alternatively, one may derive \eqref{T0-2} from \eqref{T0-1} by considering $\int_0^\infty \e^{-qt}\P^{\zbe}_{x_0}(T_0(\rho)>t)\d t$ and using the explicit probability densities of $\BES(0)$ in \eqref{density}.] Besides,
the probability distribution of $\{\rho_t;t>T_0(\rho)\}$ can be characterized by the local time at level $0$ and the associated excursion measure, denoted by $\{L_t\}$ and $\mathbf N^{\beta \da }_0$, respectively, which satisfy the following properties. We have
\begin{align}\label{Lap:lim}
\E^{(0),\beta\da }_0\left[\int_0^\infty \e^{-qt}\d L_t\right]=\frac{1}{\log (1+q/\beta)},\quad\forall\; q\in (0,\infty),
\end{align}
and, with $\zeta(\epsilon)$ denoting the lifetime of an excursion path $\epsilon=\{\epsilon_t\}\in C(\R_+,\R_+)$, 
\begin{enumerate}[label={\rm ({\roman*})}]
\item $\mathbf N_0^{\beta \da }(\zeta(\epsilon)\in \d v)=v^{-1}\e^{-\beta v}\d v$, $0<v<\infty$, and
\item under $\mathbf N_0^{\beta\da} (\,\cdot\mid \zeta(\epsilon)=v)$, $\{\epsilon_t;0\leq t\leq v\}$ is a two-dimensional Bessel bridge from $0$ to $0$ over the time interval $[0,v]$. 
\end{enumerate}
See \cite[Corollary~2.3 on p.884]{DY:Krein} for \eqref{Lap:lim}, and \cite[Theorem~2.1 on p.883]{DY:Krein} for (i) and (ii). Basic properties of Bessel bridges can be found in  \cite[Section~XI.3 on pp.463+]{RY}. To use (ii) in the following proof, note that the explicit probability densities of $\BES(0)$ in \eqref{density} give, for $t\in (0,v)$ and $y\in (0,\infty)$, 
\begin{align}
\mathbf N^{\beta\da }_0(\epsilon_t\in \d y\mid \zeta(\epsilon)=v)&=\lim_{y_0\searrow 0}\frac{p^{(0)}_{t}(0,y)p^{(0)}_{v-t}(y,y_0)}{p^{(0)}_v(0,y_0)}\d y\notag\\
&=\lim_{y_0\searrow 0}\frac{\frac{y}{t}\exp(-\frac{y^2}{2t})\frac{y_0}{v-t}\exp(-\frac{y^2+y_0^2}{2(v-t)})I_0(\frac{yy_0}{v-t})
}{\frac{y_0}{v}\exp(-\frac{y_0^2}{2v}) }\d y\notag\\
&=\frac{vy}{t(v-t)}\exp\left(-\frac{y^2}{2t}-\frac{y^2}{2(v-t)}\right)\d y,\label{N:density}
\end{align}
where the last equality holds since the modified Bessel function of the first kind $I_0(x)$ converges to $1$ as $x\searrow 0$ by \eqref{asymp:I}. 

\begin{proof}[First proof of Theorem~\ref{thm:1}]
To prove \eqref{id:2-1}, we first consider an equivalent of the right-hand side of \eqref{id:2-1} for the convenience of computations: for $z_0\in \Bbb C\setminus\{0\}$,
\begin{align}
&\eqspace\E_{z_0}^{(0),\beta\da}\left[\frac{\e^{\beta t}K_0(\sqrt{2\beta}\lvert z_0 \rvert)}{K_0(\sqrt{2\beta}\lvert Z_t \rvert)}f(Z_t)\right]\notag\\
&=\E_{z_0}^{(0),\beta\da}\left[\frac{\e^{\beta t}K_0(\sqrt{2\beta}\lvert z_0 \rvert)}{K_0(\sqrt{2\beta}\lvert Z_t \rvert)}f(Z_t);t<T_0(Z)\right]\notag\\
&\quad+\E^{(0),\beta\da}_{z_0}\left[\frac{\e^{\beta t}K_0(\sqrt{2\beta}\lvert z_0 \rvert)}{K_0(\sqrt{2\beta}\lvert Z_t \rvert)}f(Z_t);t\geq T_0(Z)\right]\notag\\
&=\E^{(0)}_{z_0}\left[f(Z_t)\right]+\E^{(0),\beta\da}_{z_0}\left[\frac{\e^{\beta t}K_0(\sqrt{2\beta}\lvert z_0 \rvert)}{K_0(\sqrt{2\beta}\lvert Z_t \rvert)}f(Z_t);t\geq T_0(Z)\right].\label{formula1-1}
\end{align} 
To justify the first term in \eqref{formula1-1}, recall that by \eqref{Z:skewproduct}, $Z_t$ for $t<T_0(Z)=T_0(\rho)$ takes the same form of skew-product representation as planar Brownian motion \cite[(2.11) Theorem on p.193]{RY}, and then we use \eqref{T0-1} and the independence $\{\lvert Z_t \rvert\}\ind \{\gamma_t\}$.  

To evaluate the last expectation in \eqref{formula1-1}, recall the radialization of $f$ defined by \eqref{def:fbar}. Then the Laplace transform of that expectation in \eqref{formula1-1} satisfies, for all $q\in (\beta,\infty)$,
\begin{align}
&\quad \int_0^\infty\e^{-q t} \E^{(0),\beta\da }_{z_0}\left[\frac{\e^{\beta t}K_0(\sqrt{2\beta}\lvert z_0 \rvert)}{K_0(\sqrt{2\beta}\lvert Z_t \rvert)}f(Z_t);t\geq T_0(Z)\right]\d t\notag\\
&=\E^{(0),\beta\da }_{z_0}[\e^{-(q-\beta) T_0(Z)}]\cdot \int_0^\infty \e^{-(q-\beta) t }\E^{(0),\beta\da }_0\left[\frac{K_0(\sqrt{2\beta}\lvert z_0 \rvert)}{
K_0(\sqrt{2\beta}\lvert Z_t \rvert)}\overline{f}(\lvert Z_t \rvert)
\right]\d t\label{formula1-1-1}
\end{align}
by the strong Markov property of $\{Z_t\}$ at time $T_0(Z)$ and Erickson's characterization of the resolvent of $\{Z_t\}$ starting from the origin \cite[Eq. (2.3) on p.75]{Erickson}. The last equality allows for the characterization of $\BES(0,\beta\da)$ recalled before the present proof. By \eqref{T0-2}, \eqref{Lap:lim}, and the compensation formula of excursion theory \cite[Eq. (7) on p.120]{Bertoin}, \eqref{formula1-1-1} gives
\begin{align}
&\quad \int_0^\infty\e^{-q t} \E^{(0),\beta\da }_{z_0}\left[\frac{\e^{\beta t}K_0(\sqrt{2\beta}\lvert z_0 \rvert)}{K_0(\sqrt{2\beta}\lvert Z_t \rvert)}f(Z_t);t\geq T_0(Z)\right]\d t\notag\\
&=\frac{K_0(\sqrt{2q}\lvert z_0 \rvert)}{K_0(\sqrt{2\beta}\lvert z_0 \rvert)}\cdot \frac{1}{\log [1+(q-\beta)/\beta]}\mathbf N_0^{\beta\da}\biggl(\int_0^{\zeta(\epsilon)} \e^{-(q-\beta) t}\frac{K_0(\sqrt{2\beta}\lvert z_0 \rvert)}{
K_0(\sqrt{2\beta}\epsilon_t)}\overline{f}(\epsilon_t)\d t\biggr)\label{formula1-1-0}\\
&=\frac{\pi}{\log(q/\beta)}\int_0^\infty \e^{-q t}P_{t}(z_0)\d t\mathbf N_0^{\beta\da}\biggl(\int_0^{\zeta(\epsilon)} \e^{-(q-\beta) t}\frac{\overline{f}(\epsilon_t)}{
K_0(\sqrt{2\beta}\epsilon_t)}\d t\biggr),\label{formula1-2}
\end{align}
where the last equality holds since \eqref{def:Pt} and \eqref{def:K} imply
\begin{align}\label{LT:norm0}
\int_0^\infty \e^{-q t}P_{t}(z)\d t=\frac{1}{\pi}K_0(\sqrt{2q}\lvert z\rvert).
\end{align}

Next, we compute for an explicit form of the $\mathbf N^{\beta\da }_0$-expectation in \eqref{formula1-2}. By
the description of $\mathbf N^{\beta\da}_0$ recalled below \eqref{Lap:lim} and by \eqref{N:density},
it holds that for all nonnegative $g$,
\begin{align*}
&\eqspace\mathbf N^{\beta\da }_0\biggl(\int_0^{\zeta(\epsilon)} \e^{-(q-\beta) t}g(\epsilon_t)\d t\biggr)\\
&=\int_0^\infty \frac{\e^{-\beta v}}{v}\int_0^v\e^{-(q-\beta) t}\int_0^\infty \frac{vy}{t(v-t)}\exp\left(-\frac{y^2}{2t}-\frac{y^2}{2(v-t)}\right)
g(y)\d y\d t\d v\\
&=\int_0^\infty \e^{-q t } \int_0^\infty  \frac{1}{t}\exp\left(-\frac{y^2}{2t}\right)yg(y) \int_t^\infty\e^{-\beta(v-t)} \frac{1}{(v-t)}\exp\left(-\frac{y^2}{2(v-t)}\right)\d v\d y\d t\\
&=\int_0^\infty \e^{-q t} \int_0^\infty \frac{1}{t}\exp\left(-\frac{y^2}{2t}\right)yg(y)\cdot 2K_0(\sqrt{2\beta}y)\d y\d t,
\end{align*}
where the last equality uses \eqref{def:K} for $\nu=0$ after a change of variable that replaces $\beta (v-t)$ with $v$. Hence, by the polar coordinates and the definition \eqref{def:fbar} of $\bar{f}$, we get
\begin{align}
\mathbf N^{\beta\da }_0\biggl(\int_0^{\zeta(\epsilon)} \e^{-(q-\beta) t}\frac{\overline{f}(\epsilon_t)}{
K_0(\sqrt{2\beta}\epsilon_t)}\d t\biggr)=2\int_0^\infty \e^{-q t}P_tf(0)\d t.\label{formula1-3}
\end{align}

Now, combining \eqref{formula1-1}, \eqref{formula1-2} and \eqref{formula1-3}, we get
\begin{align}
&\quad  \int_0^\infty \e^{-q t} \E^{(0),\beta\da }_{z_0}\left[\frac{\e^{\beta t}K_0(\sqrt{2\beta}\lvert z_0 \rvert)}{K_0(\sqrt{2\beta}\lvert Z_t \rvert)}f(Z_t)\right]\d t\notag\\
&=\int_0^\infty \e^{-q t}P_tf(z_0)\d t+\frac{2\pi}{\log (q/\beta)}\int_0^\infty \e^{-q t}P_t(z_0)\d t\int_0^\infty \e^{-q t}P_tf(0)\d t.\label{formula1-4}
\end{align}
To reconcile \eqref{formula1-4} and the Laplace transform of $\{P^\beta_t\}$ in \eqref{def:DBG}, note
\begin{align}
P_tf(z)=\E^{(0)}_{\two z}[f(Z_{2t}/\two)]=P_{2t}\{\tilde{z}\mapsto f(\tilde{z}/\two)\}(\two z),\quad
P_t(z)=2P_{2t}(\two z),\label{two}
\end{align}
where $P_t\{\tilde{z}\mapsto g(\tilde{z})\}(z)$ denotes $P_tg(z)$. Then \eqref{formula1-4} can be restated as follows:
\begin{align}
&\quad  \int_0^\infty \e^{-q t} \E^{(0),\beta\da }_{z_0/\two}\left[\frac{\e^{\beta t}K_0(\sqrt{\beta}\lvert z_0 \rvert)}{K_0(\sqrt{2\beta}\lvert Z_t \rvert)}f(\two Z_t)\right]\d t\label{formula1-5-1}\\
&=\int_0^\infty \e^{-q t}P_t\{\tilde{z}\mapsto f(\two \tilde{z})\}(z_0/\two)\d t\notag\\
&\quad\;+\frac{2\pi}{\log (q/\beta)}\int_0^\infty \e^{-q t}P_t(z_0/\two)\d t\int_0^\infty \e^{-q t}P_t\{\tilde{z}\mapsto f(\two \tilde{z})\}(0)\d t\notag\\
&=\int_0^\infty \e^{-q t}P_{2t}f(z_0)\d t+\frac{4\pi}{\log (q/\beta)}\int_0^\infty \e^{-q t}P_{2t}(z_0)\d t\int_0^\infty \e^{-q t}P_{2t}f(0)\d t\label{formula1-5}
\end{align}
by \eqref{two}. Up to this point, 
we have two functions of $t\in \R_+$, denoted by $F_1$ and $F_2$, that have the same Laplace transform given by the right-hand side of \eqref{formula1-5}. Specifically, $F_1$ refers to the expectation in the integral of \eqref{formula1-5-1}. 
The second function $F_2$ refers to the function of $t$ on the right-hand side of \eqref{def:Pbeta2} when we integrate it against $f(\tilde{z})\d\tilde{z}$ and set $z=z_0$. Note that $F_2$ has a Laplace transform given by the right-hand side of \eqref{formula1-5} by \eqref{Inv:sbeta}.
On the other hand, when $f$ is also bounded continuous, Lemma~\ref{lem:contexp} shows that  $F_1$ and $F_2$  are continuous and of exponential order, and so, $F_1$ and $F_2$ can be inverted from their Laplace transforms. We conclude \eqref{id:2-1} for all bounded continuous $f\geq 0$, hence, for all $f\in \B_+(\Bbb C)$.

\begin{rmk}\label{rmk:lap}
In the above inversion argument to complete the proof of \eqref{id:2-1}, we have chosen to work with an integrated form on the right-hand side of \eqref{def:Pbeta2}, rather than $P^\beta_t f(z_0)$. By this choice, we have also clarified the step of the proof of \cite[Proposition~5.1 on p.176]{C:DBG3+} on inverting the right-hand side of \eqref{def:DBG}
 to the right-hand side of \eqref{def:Pbeta2}. See \cite[Section 3.2 on pp.231+]{ABD:Schrodinger} on obtaining $P^\beta_t(z,\tilde{z})$ from inverting its Laplace transform. \qed 
\end{rmk}

The proof of \eqref{id:2-2} is similar. As in \eqref{formula1-1-0}, we obtain from
 \eqref{Lap:lim} and the compensation formula of excursion theory \cite[Eq. (7) on p.120]{Bertoin} that
\begin{align}
&\quad\;\int_0^\infty \e^{-q t} \E^{(0),\beta\da }_0\left[\frac{\e^{\beta t}\cdot 2\pi}{K_0(\sqrt{2\beta}\lvert Z_t \rvert)}f(\two Z_t)\right]\d t\notag\\
&=\frac{2\pi}{\log (q/\beta)}\mathbf N^{\beta\da }_0\biggl(\int_0^{\zeta(\epsilon)} \e^{-(q-\beta)t}\frac{\bar{f}(\two\epsilon_t)}{K_0(\sqrt{2\beta}\epsilon_t)}\d t\biggr)\notag\\
&=\frac{4\pi}{\log (q/\beta)}\int_0^\infty \e^{-q t}P_t\{\tilde{z}\mapsto f(\two \tilde{z})\}(0)\d t\notag\\
&=\frac{4\pi}{\log (q/\beta)}\int_0^\infty \e^{-q t}P_{2t}f(0)\d t
=\int_0^\infty \e^{-q t}\mathring{P}^\beta_tf\d t,\notag
\end{align}
where the second equality uses \eqref{formula1-3}, the third equality uses the first equality in \eqref{two},
and the last equality can be seen by comparing \eqref{def:DBG} and \eqref{def:Pbeta2}. The last equality proves \eqref{id:2-2}, again using Lemma~\ref{lem:contexp} to justify the Laplace inversions.

Finally, \eqref{BESab:density} with $f$ replaced by $\tilde{f}$ 
follows by using \eqref{id:2-1} with $z_0/\two$ replaced by $z_0$ and then choosing, with $H\in \{K_0(\sqrt{\beta}|\two z_0|),2\pi\}$, 
\begin{align*}
\tilde{f}(z)=\frac{\e^{\beta t}H}{K_0(\sqrt{\beta}|\two z|)}f(\two z)&\Longleftrightarrow f(z)=\frac{\e^{-\beta t}}{H}\tilde{f}(z/\two)K_0(\sqrt{\beta}|z|).
\end{align*}
The first proof of Theorem~\ref{thm:1} is complete.
\end{proof}

The following proof uses the SDE of $\BES Q(0,\beta\da)$, which will be derived as a preparation of the second proof of Theorem~\ref{thm:1}.

\begin{lem}
\label{lem:contexp}
Write $\widetilde{P}^\beta_tf(z_0)$ for the right-hand side of \eqref{def:Pbeta2} integrated against $f(\tilde{z})\d \tilde{z}$ with $z=z_0$.  Then for  bounded $f\in \C_+(\Bbb C)$,
$\widetilde{P}^\beta_tf(z_0)$ and the right-hand side of \eqref{id:2-1} as functions of $t\in \R_+$
can be bounded by $C(\beta,f,z_0)\e^{C'(\beta,f,z_0)t}$
and are continuous. The same properties, with $z_0=0$ for constants in the bound, are satisfied by both sides of \eqref{id:2-2}.
\end{lem}
\begin{proof}
We consider $\widetilde{P}^\beta_tf(z_0)$ and the left-hand side of \eqref{id:2-2} first. 
To see the required bounds of these functions, it suffices to note that  for $\s^\beta$ defined in \eqref{def:s},
$\int_0^t \s^\beta(\tau)\d \tau\leq C(\beta)\e^{C'(\beta)t}$ for all $t\geq 0$. This bound is implied by the second inequality below:
\begin{align}
 \int_0^t\s^\beta(\tau)\d \tau&=4\pi \int_0^t\int_0^\infty \frac{\beta^u\tau^{u-1}}{\Gamma(u)}\d u\d \tau=4\pi \int_0^\infty \frac{(\beta t)^u}{\Gamma(u+1)} \d u\label{sbeta1}\\
 &\leq 4\pi \int_0^{\lceil u_0\rceil } \frac{(\beta t)^u}{\Gamma(u+1)} \d u+ 4\pi \sum_{n=\lceil u_0\rceil+1}^\infty \frac{\max\{(\beta t)^n,1\}}{\Gamma(n)},\label{sbeta2}
\end{align}
where $u_0$ is the unique zero of $\Gamma'(\cdot)$ in $(0,\infty)$ since $\Gamma''(x)=\int_0^\infty t^{x-1}(\log t)^2\e^{-t}\d t>0$ for all $x>0$. Note that we use an integral comparison to get \eqref{sbeta2}. 

Next, we prove the continuity of $t\mapsto \widetilde{P}^\beta_tf(z_0)$ and the left-hand side of \eqref{id:2-2}. 
For $t\mapsto \widetilde{P}^\beta_tf(z_0)$, note that $t\mapsto P_{2t}f(z_0)$ is continuous on $\R_+$ by dominated convergence. Hence, it is enough to show the continuity of the left-hand side of \eqref{id:2-2} on $\R_+$. 
To this end, we consider, for $t'>t$,
\begin{align}
|\mathring{P}^\beta_{t'}f-\mathring{P}^\beta_{t}f|
&\leq C(f) \int_t^{t'}\s^\beta(t'-\tau)\d \tau +C(f)\int_{0}^t |\s^\beta(t'-\tau)-\s^\beta(t-\tau)|\d \tau\notag\\
&= C(f) \int_0^{t'-t}\s^\beta(\tau)\d \tau +C(f)\int_{0}^t |\s^\beta(t'-t+\tau)-\s^\beta(\tau)|\d \tau\label{Lapinv1}\\
\begin{split}
&\leq C(f)\int_0^\infty \frac{[\beta (t'-t)]^u}{\Gamma(u+1)} \d u+C(f)\int_0^1\frac{(\beta t)^u-(\beta t')^u+(\beta t'-\beta t)^u}{\Gamma(u+1)}\d u\\
&\quad\;+C(f)\int_1^\infty \frac{(\beta t')^u+(\beta t'-\beta t)^u-(\beta t)^u}{\Gamma(u+1)}\d u\label{Lapinv2}
\end{split}
\end{align}
by using the computation in \eqref{sbeta1}. More specifically,
the sum of the last two terms of \eqref{Lapinv2} bounds the last term of \eqref{Lapinv1} since the definition \eqref{def:s} of $\s^\beta$ implies
\begin{align*}
|\s^\beta(t'-t+\tau)-\s^\beta(\tau)|&\leq \int_0^1\frac{\beta^u[\tau^{u-1}-(t'-t+\tau)^{u-1}]}{\Gamma(u)}\d u\\
&\quad\;+\int_1^\infty\frac{\beta^u[(t'-t+\tau)^{u-1}-\tau^{u-1}]}{\Gamma(u)}\d u.
\end{align*}
By dominated convergence, \eqref{Lapinv2} implies $\mathring{P}^\beta_{t'}f-\mathring{P}^\beta_{t}f\to 0$ as $t'-t\searrow 0$ whenever $t'$ and $t$ are in a fixed compact set of $\R_+$. We have proved the required continuity of $t\mapsto  \mathring{P}^\beta_{t}f$. 

To handle the right-hand sides of \eqref{id:2-1} and \eqref{id:2-2}, our key tool is the fact that $\{|Z_t|\}=\{\rho_t\}$ under $\P^{\zbe}$ can be pathwise dominated by a version of $\BES(0)$ with the same initial condition $\rho_0$. This domination uses the SDE of $\{\rho_t^2\}$ in \eqref{SDE:DMY} and the comparison theorem of SDEs \cite[Theorem~VI.1.1 on pp.437--438]{IW:SDE}. 
Given this pathwise domination, the required growth properties of the right-hand sides of \eqref{id:2-1} and \eqref{id:2-2} follow immediately since $\E[\e^{q |B_t|}]\leq 4 \e^{q^2 t}$ for a two-dimensional standard Brownian motion $\{B_t\}$ with $B_0=0$ and all $q\in \R$ 
by the following bound for $Z\sim \mathcal N(0,1)$: $\E[\e^{a|Z|}]\leq 2\E[\e^{aZ}]=2\e^{a^2/2}$. In particular, the exponential moments of any fixed order of $|Z_t|$ are bounded on compacts in $t$, so by \eqref{asymp:Kinfty}, the family of $K_0(\sqrt{2\beta}|Z_t|)^{-1}$ for $t$ in a fixed compact set is uniformly integrable.

It remains to obtain the required continuity of the right-hand sides of \eqref{id:2-1} and \eqref{id:2-2}. Note that $K_0^{-1}(\cdot)$ is continuous on $\R_+$ by setting $K^{-1}_0(0)\,\defeq\,0$ and using \eqref{asymp:K0}. The required continuity of the right-hand sides of \eqref{id:2-1} and \eqref{id:2-2} thus follows from the path continuity of $\{Z_t\}$ since we can pass limits under the expectations by using a standard theorem under uniform integrability. Recall that we have the form of uniform integrability of 
$K_0(\sqrt{2\beta}|Z_t|)^{-1}$ from the end of the previous paragraph. The proof of Lemma~\ref{lem:contexp} is complete.
\end{proof}

\section{The stochastic relative position: the SDE and the clock for time change}\label{sec:SDE}
Our main goal in this section is to prove Theorem~\ref{thm:2}. In Sections~\ref{sec:thm2-1} and~\ref{sec:thm2-2}, we will first prove Theorem~\ref{thm:2} (1$\cc$) for $\alpha>0$ and Theorem~\ref{thm:2} (2$\cc$) for $\alpha,\alpha_1>0$, respectively. Section~\ref{sec:thm2-12} will give the extensions to $\alpha=\alpha_1=0$ by continuity since $\BES(0,\beta\da)$ will be approached as the $(\alpha\searrow0)$-distributional limit of $\BES(-\alpha,\beta\da)$ as in \cite{DY:Krein}.  Theorem~\ref{thm:2} (3$\cc$) and (4$\cc$) will be proven in Sections~\ref{sec:thm2-3} and~\ref{sec:thm2-4}.

In giving the proofs for Theorem~\ref{thm:2}, we will use several supporting properties. The proofs of such properties of $\mu^{\beta\da }_\alpha(x)$, defined in \eqref{def:mu}, are postponed to 
 Section~\ref{sec:RATIO}. There we will study the corresponding properties of the functions $R_\alpha(x)$ satisfying
\begin{align}\label{eq:mu}
\mu^{\beta\da }_\alpha(x)=2-2R_\alpha(\sqrt{2\beta \lvert x\rvert}).
\end{align}
See \eqref{def:Ralpha} for the precise definition of $R_\alpha$. The other supporting properties are about the Bessel processes. The proofs are postponed to Section~\ref{sec:BES}.

\subsection{Proof of Theorem~\ref{thm:2} (1$\cc$) for $\alpha\in (0,1/2)$}\label{sec:thm2-1}
The main step of the proof is to represent $\P^{\albe}_{x_0}$ as a probability measure defined by Girsanov's transformation of $\P^{\al}_{x_0}$. Recall the probability measure $\P^{\albe}_{x_0}$ defined in \eqref{def:BESab}. We stress that the assumption $\alpha>0$ is crucial for the following arguments. For example, Proposition~\ref{prop:mombdd} (2$\cc$) will be applied repeatedly for the case of $\eta=-2+4\alpha$. 

\begin{prop}\label{prop:key}
For all $\alpha\in (0,1/2)$, $\beta\in (0,\infty)$ and $x_0\in \R_+$, it holds that 
\begin{align}\label{eq:key}
\log \frac{\widehat{K}_\alpha(\sqrt{2\beta}\rho_t)}{\widehat{K}_\alpha(\sqrt{2\beta}x_0)}=-\Lambda_\alpha(\beta) L_t+\beta t+N_t -\frac{1}{2}\langle N,N\rangle_t\quad\mbox{under $\P^{\al}_{x_0}$,}
\end{align}
where $\Lambda_\alpha(\beta)$ is defined in \eqref{choice}, $\{L_t\}$ is the diffusion local time of $\{\rho_t\}$ satisfying \eqref{def:Lt}, and 
\begin{align}
\begin{split}\label{def:Nt}
N_t\,\defeq 
-(2\beta)^\alpha\int_0^t\frac{1}{\rho_s^{1-2\alpha} } \frac{\widehat{K}_{1-\alpha}}{\widehat{K}_\alpha}(\sqrt{2\beta}\rho_s) \d W_s
= -\sqrt{2\beta}\int_0^t \frac{K_{1-\alpha}}{K_\alpha}(\sqrt{2\beta}\rho_s) \d W_s
\end{split}
\end{align}
is a continuous $\P^{\al}_{x_0}$-martingale using $\widehat{K}_\nu$ defined in \eqref{def:hK}.
\end{prop}

\begin{rmk}\label{rmk:Ito}
It is not clear to us that the martingale $\{N_t\}$ can be derived from \eqref{def:BESab} by applying  It\^{o}'s formula to $x\mapsto \log\widehat{K}_\alpha(x)$ directly.
More specifically,  we consider the version of It\^{o}'s formula in \cite[6.22 Theorem on p.214]{KS:BM}. Since $\BES(-\alpha,\beta\da)$ can hit zero, an attempt to apply this version of It\^{o}'s formula will require an extension of $x\mapsto \log\widehat{K}_\alpha(x)$, $x\geq 0$, to 
a function on $\R$ given by a linear combination of convex functions. But  when $\alpha\in (0,1/2)$,
 this extension does not exist since every convex function defined on $\R$ allows finite one-sided derivatives at all points \cite[p.544]{RY}, but
\[
\frac{\mathrm d}{\mathrm d x}\log \widehat{K}_{\alpha}(x)=-\frac{K_{1-\alpha}}{K_\alpha}(x)
\to -\infty\quad\mbox{as }x\to 0+.
\]
See \eqref{d1logG} for the computation of the derivative and \eqref{asymp:K0} for the limit.
\qed 
\end{rmk}

To circumvent the difficulty discussed in Remark~\ref{rmk:Ito}, we will follow a method from
 \cite[Theorem~2.1]{DRVY} to prove Proposition~\ref{prop:key}. That method from \cite{DRVY} originally derives the Doob--Meyer decomposition of $\{\rho_t^{2\alpha}\}$ for $\BES(-\alpha)$, $0<\alpha<1 $, by using $(\vep+\rho_t^2)^\alpha$ as $\vep\searrow 0$ and a special approximation of the Markovian local times $\{L_t\}$. See Proposition~\ref{prop:LTapprox} for a relation between this approximation and the speed measure of $\BES(-\alpha)$. The use of Markovian local times is necessary since the local times at level $0$ from Tanaka's formula vanish; see \eqref{DLT:otf}.

\begin{cor}\label{cor:SDE>0}
Under $\P_{x_0}^{\al,\beta\da }$ for $\alpha\in (0,1/2)$, \eqref{SDE:DMY} holds with $X_t=\rho_t^2$ and
\begin{align}\label{def:tildeW}
\widetilde{W}_t=W_t+ \sqrt{2\beta} \int_0^t\frac{K_{1-\alpha}}{K_\alpha}(\sqrt{2\beta}\rho_s)\d s.
\end{align}
\end{cor}
\begin{proof}
By \eqref{eq:key}, the definition in \eqref{def:BESab} can be rewritten as
\begin{align}\label{rep:Palbe}
\d\P^{\albe}_{x_0}\rvert_{\sigma(\rho_s;s\leq t)}=\ms E(N)_t\d \P^{\al}_{x_0}\rvert_{\sigma(\rho_s;s\leq t)}.
\end{align}
Here, $\ms E(N)_t=\exp\{N_t-\langle N,N\rangle_t/2\}$ is the stochastic exponential of the martingale $\{N_t\}$ in \eqref{def:Nt}. Note that $\ms E(N)$ is a martingale under $\P^{\al}_{x_0}$ since Novikov's criterion \cite[(1.15) Proposition on p.332]{RY} is satisfied by using the first equality of \eqref{def:Nt}, \eqref{asymp:K0}, \eqref{asymp:Kinfty}, and Proposition~\ref{prop:mombdd} (2$\cc$)  with $\eta=-2+4\alpha$. The required property of Corollary~\ref{cor:SDE>0} thus follows from Girsanov's theorem \cite[(1.7) Theorem on p.329]{RY}, \eqref{def:BESQ}, and the second equality of \eqref{def:Nt}.
\end{proof}

We now begin the proof of Proposition~\ref{prop:key}. 
The first lemma obtains an approximate form of \eqref{eq:key} with the shorthand notation
\begin{align}\label{F:shorthand}
G_\nu(b)\,\defeq \,\widehat{K}_\nu(\sqrt{b}\,)=b^{\nu/2}K_\nu(\sqrt{b}).
\end{align}
\begin{lem}\label{lem:key}
Fix $\alpha\in (0,1/2)$, $\beta\in (0,\infty)$ and $x_0\in \R_+$. Under $\P^{(-\alpha)}_{x_0}$, it holds that 
\begin{align}\begin{split}\label{hK:Ito0}
\log \frac{\widehat{K}_\alpha(\sqrt{2\beta}(\vep+\rho_t^2)^{1/2})}{ \widehat{K}_\alpha(\sqrt{2\beta}(\vep+x_0^2)^{1/2})}
=
I^{1,\vep}_t+I^{2,\vep}_t+N^\vep_t-\frac{1}{2}\langle N^\vep,N^\vep\rangle_t,\quad \vep\in(0,\infty).
\end{split}
\end{align}
Here, with $G_\nu$ defined by \eqref{F:shorthand}, 
\begin{align}
N^\vep_t&\,\defeq  \int_0^t \left.\frac{-2\beta G_{1-\alpha}(b)}{b^{1-\alpha} G_\alpha(b)}\right\rvert_{b=2\beta(\vep+\rho_s^2)}\cdot \rho_s \d W_s,\label{def:N}\\
I^{1,\vep}_t&\,\defeq \left.\int_0^t\frac{-(2\beta)^\alpha G_{1-\alpha}(b) }{ G_\alpha(b)}\right\rvert_{b=2\beta(\vep+\rho_s^2)}\cdot \frac{(1-\alpha)\vep}{(\vep+\rho_s^2)^{2-\alpha}}
\d s,\label{def:I2}\\
I^{2,\vep}_t&\,\defeq \int_0^t \beta\cdot \frac{\rho_s^2}{\vep+ \rho_s^2}\d s.\label{def:I3}
\end{align}
\end{lem}

We will  pass $\vep\to 0$ and prove that $N^\vep_t$, $\la N^\vep,N^\vep\ra_t$, $I^{1,\vep}_t$, and $I^{2,\vep}_t$ converge to $N_t$, $\la N,N\ra_t$, $-\Lambda_\alpha(\beta)L_t$, and $\beta t$, respectively. 
See \eqref{conv:N}, \eqref{conv:I2}, and \eqref{conv:I3}.

\begin{proof}[Proof of Lemma~\ref{lem:key}]
To find the first two derivatives of $\log G_\alpha(b)$, our main tools are
\begin{align}\label{hK:der}
(\d/\d x)[x^\nu K_\nu(x)]=-x^\nu K_{\nu-1}(x)\quad\forall\;\nu\in \R
\end{align}
\cite[the third identity in (5.7.9) on p.110]{Lebedev} and the even parity of $\nu\mapsto K_{\nu}(x)$ \cite[(5.7.10) on p.110]{Lebedev}. Hence, by the rules $(\d /\d b)F(f(b))=F'(f(b))f'(b)$ and $(\d^2 /\d b^2)F(f(b))=F''(f(b))f'(b)^2+F'(f(b))f''(b)$ for $F(x)\equiv \log \widehat{K}_\alpha(x)$ and $f(b)\equiv\sqrt{b}$, we  obtain, for $b>0$,
\begin{align}
\frac{\d}{\d b}\log G_\alpha(b)&=\left.\left(\frac{-x^{2\alpha-1}\widehat{K}_{1-\alpha}(x)}{\widehat{K}_\alpha(x)}\right)\right|_{x=\sqrt{b}}\left(\frac{1}{2\sqrt{b}}\right)=-\frac{G_{1-\alpha}(b)}{2b^{1-\alpha}G_\alpha(b)},\label{d1logG}\\
\frac{\d^2}{\d b^2}\log G_\alpha(b)&=\left(\frac{-(2\alpha-1)x^{2\alpha-2}\widehat{K}_{1-\alpha}(x)+x^{2\alpha-1} x^{1-\alpha}K_{-\alpha}(x)}{\widehat{K}_\alpha(x)}\right.\notag\\
&\quad \left.+\frac{x^{2\alpha-1}\widehat{K}_{1-\alpha}(x)[-x^{2\alpha-1}\widehat{K}_{1-\alpha}(x)]
}{\widehat{K}_\alpha(x)^2}\right)\Bigg|_{x=\sqrt{b}}\left(\frac{1}{2\sqrt{b}}\right)^2\notag\\
&\quad+\left.\left(\frac{-x^{2\alpha-1}\widehat{K}_{1-\alpha}(x)}{\widehat{K}_\alpha(x)}\right)\right|_{x=\sqrt{b}}\left(\frac{1}{-4b^{3/2}}\right)\label{d2logG}\\
&= \frac{(1-\alpha)G_{1-\alpha}(b)}{2b^{2-\alpha}G_\alpha(b)}+\frac{1}{4b}
 -\frac{G_{1-\alpha}(b)^2}{4b^{2-2\alpha}G_\alpha(b)^2}.\notag
\end{align}
\mbox{}\\
Next, recall $\d \rho_t^2=2(1-\alpha)\d t+2\rho_t\d W_t$, as mentioned in \eqref{def:BESQ}. Hence,
with the shorthand notation $b_s\,\defeq\,2\beta(\vep+\rho_s^2)$, It\^{o}'s formula gives
\begin{align*}
\log \frac{\widehat{K}_\alpha(\sqrt{2\beta}(\vep+\rho^2_t)^{1/2})}{\widehat{K}_\alpha(\sqrt{2\beta}(\vep+x_0^2)^{1/2})}
&=2\beta\int_0^t -\frac{G_{1-\alpha}(b_s)}{2b_s^{1-\alpha}G_\alpha(b_s)}[2(1-\alpha)\d s+2\rho_s\d W_s]\\
&\quad +\frac{(2\beta)^2}{2}\int_0^t\frac{(1-\alpha)G_{1-\alpha}(b_s)}{2b_s^{2-\alpha}G_\alpha(b_s)}4\rho_s^2\d s\\
&\quad+\frac{(2\beta)^2}{2}\int_0^t\frac{1}{4b_s}4\rho_s^2\d s -\frac{(2\beta)^2}{2}\int_0^t\frac{G_{1-\alpha}(b_s)^2}{4b_s^{2-2\alpha}G_\alpha(b_s)^2}4\rho_s^2\d s,
\end{align*}
which leads to \eqref{hK:Ito0} after some elementary algebra.
\end{proof}

\begin{proof}[Proof of Proposition~\ref{prop:key}]
First, the second equality in \eqref{def:Nt} holds upon recalling \eqref{def:hK}. To see that $\{N_t\}$ is a martingale, first, note that by
 the asymptotic representations of $K_{1-\alpha}$ and $K_\alpha$ at zero and infinity from \eqref{asymp:K0} and \eqref{asymp:Kinfty} for $\alpha\in (0,1/2)$, 
 \begin{align}\label{Kratio:asymp}
\frac{\widehat{K}_{1-\alpha}(x)}{x^{1-2\alpha}\widehat{K}_\alpha(x)}=\frac{x^{1-\alpha}K_{1-\alpha}(x)}{x^{1-2\alpha}x^\alpha K_{\alpha}(x)}\leq 
\begin{cases}
C(\alpha),& \forall\;x\in [1,\infty);\\
C(\alpha)/x^{1-2\alpha},& \forall\;x\in (0,1].
\end{cases}
 \end{align}
Since $\E^{\al}_{x_0}[\int_0^t \d s/\rho_s^{2-4\alpha}]<\infty$ by Proposition~\ref{prop:mombdd} (2$\cc$) with $\eta=-2+4\alpha$, the required martingale property of $\{N_t\}$ follows by using the first equality of \eqref{def:Nt}. 
 
In the remaining of this proof, we focus on the justification of \eqref{eq:key} by identifying the limits of the terms on the right-hand side of \eqref{hK:Ito0}. We stress that \emph{all of these limits are considered under $\P^{\al}_{x_0}$}. The left-hand side of \eqref{hK:Ito0} obviously converges to the left-hand side of \eqref{eq:key} as $\vep\searrow 0$.  \smallskip 

\noindent \hypertarget{prop:key-1}{{\sc Step 1.}}\; We show that 
\begin{align}\label{conv:N}
N^\vep_t\xrightarrow[\vep \to 0]{\P}N_t,\quad \langle N^\vep,N^\vep\rangle_t\xrightarrow[\vep \to 0]{\rm a.s.}\langle N,N\rangle_t,
\end{align}
where the mode of convergence of the first limit refers to convergence in probability, and the second is almost-sure convergence.
By \eqref{F:shorthand} and \eqref{def:N}, we can write
\begin{align*}
 N^\vep_t&= -2\beta\int_0^t \left.\frac{\widehat{K}_{1-\alpha}(x)}{x^{2-2\alpha} \widehat{K}_\alpha(x)}\right\rvert_{x=\sqrt{2\beta}(\vep+\rho_s^2)^{1/2}}\cdot \rho_s \d W_s.
 \end{align*}
The first limit in \eqref{conv:N} then follows from \eqref{Kratio:asymp}, the property $\E^{\al}_{x_0}[\int_0^t \d s/\rho_s^{2-4\alpha}]<\infty$ mentioned below \eqref{Kratio:asymp}, 
 and the dominated convergence theorem for stochastic integrals. For this theorem, see \cite[(2.12) Theorem on p.142]{RY} and its proof. The second limit in \eqref{conv:N} follows similarly by the usual dominated convergence theorem. \smallskip 
  
\noindent \hypertarget{prop:key-2}{{\sc Step 2.}}\;
Next, we show that
\begin{align}\label{conv:I2}
I^{1,\vep}_t\xrightarrow[\vep\to 0]{\rm a.s.} -\Lambda_\alpha(\beta)L_t.
\end{align}

First, write the integrand in \eqref{def:I2} for $I^{1,\vep}_t$ as
\begin{align}\label{function:I1vep}
\left.-\frac{(2\beta)^\alpha G_{1-\alpha}(b) }{2C_\alpha G_\alpha(b)}\right\rvert_{b=2\beta(\vep+\rho_s^2)}\cdot \kappa_\vep(\rho_s),
\end{align}
where $C_\alpha$ is chosen according to $C'_\alpha$ via \eqref{def:Cpalpha}, with $C'_\alpha=C^\star_\alpha$ specified below \eqref{def:Lt},
and $\kappa_\vep$ is a special approximation to the identity defined by \eqref{def:kappa}. To pass the limit of
\eqref{function:I1vep}, we use the fact that as functions of $x$,
\begin{align}\label{lim:G}
\lim_{\vep\searrow 0}\frac{(2\beta)^\alpha G_{1-\alpha}(2\beta(\vep+x)) }{2C_\alpha G_\alpha(2\beta(\vep+x))}
=\frac{(2\beta)^\alpha G_{1-\alpha}(2\beta x) }{2C_\alpha G_\alpha(2\beta x)}\quad\mbox{uniformly on compacts in $\R_+$}.
\end{align}
This uniform convergence on compacts in \eqref{lim:G} holds because for all $\nu_0\in \{\alpha,1-\alpha\}$, 
\begin{align}\label{Gvep:conv}
 G_{\nu_0}(2\beta(\vep+x))\nearrow  G_{\nu_0}(2\beta x)\quad\mbox{ as $\vep\searrow 0$}
\end{align}
and this convergence is uniform on compacts in $\R_+$ by Dini's theorem. In more detail, the monotonicity in \eqref{Gvep:conv} is due to \eqref{hK:der}, and the applicability of Dini's theorem is due to the continuity of $x\mapsto G_{\nu_0}(2\beta x)$ on $\R_+$, where the continuity at $0$ is by \eqref{asymp:K0}.

We deduce from \eqref{lim:G} and \eqref{char2:LT} ($L^0_t\equiv L_t$ under $C'_\alpha=C^\star_\alpha$) that
\[
I^{1,\vep}_t\xrightarrow[\vep\to 0]{\rm a.s.}-\int_0^t \frac{(2\beta)^\alpha G_{1-\alpha}(2\beta \rho_s^2) }{2C_\alpha G_\alpha(2\beta \rho_s^2)} \d L_s =-\frac{(2\beta)^\alpha 2^{1-\alpha-1}\Gamma(1-\alpha)}{2C_\alpha 2^{\alpha-1}\Gamma(\alpha) }L_t=-\Lambda_\alpha (\beta) L_t,
\]
where the first equality uses \eqref{def:hK0} and the property that $\supp(\d L_s)\subseteq \{s;\rho_s=0\}$, and the second equality uses  the relationship between $C_\alpha$ and $C'_\alpha=C^\star_\alpha$ in \eqref{def:Cpalpha} and the definition \eqref{choice} of $\Lambda_\alpha(\beta)$. We have proved \eqref{conv:I2}.\smallskip 

\noindent {\sc Step~3.} \; The last limit we show is that by \eqref{def:I3} and dominated convergence, 
\begin{align}\label{conv:I3}
\quad I^{2,\vep}_t\xrightarrow[\vep\to 0]{\rm a.s.}  \beta\int_0^t \1_{(\rho_s>0)}\d s=\beta t.
\end{align}
This equality holds since the time spent by $\BES Q(-\alpha)$ in $0$ has a zero Lebesgue measure \cite[(1.5)~Proposition on p.442]{RY}. \medskip

\noindent {\sc Step~4.} \; Finally, applying \eqref{conv:N}, \eqref{conv:I2} and \eqref{conv:I3} to the right-hand side of \eqref{hK:Ito0} is enough to get the required identity in \eqref{eq:key}, since the left-hand side of \eqref{hK:Ito0} converges almost surely as $\vep\searrow 0$. The proof is complete.
\end{proof}

\subsection{Proof of Theorem~\ref{thm:2} (2$\cc$) for $\alpha,\alpha_1\in (0,1/2)$}\label{sec:thm2-2}
The following proposition restates the corresponding part of Theorem~\ref{thm:2} (2$\cc$). Note that by Theorem~\ref{thm:2} (1$\cc$), the weak existence of solutions of \eqref{SDE:DMY} holds. Also, the condition of $\alpha,\alpha_1>0$ continues to be crucial for the following arguments. 

\begin{prop}\label{prop:nonnegative}
Fix $\beta\in(0,\infty)$. 
\begin{enumerate}[label={\rm ({\arabic*}$\cc$)}]
\item Given any $\alpha\in (0,1/2)$ and $x_0\geq 0$, any solution to \eqref{SDE:DMY} is nonnegative.\smallskip

\item Given any $\alpha\in(0,1/2)$ and $x_0\geq 0$, there is pathwise uniqueness in \eqref{SDE:DMY}.\smallskip

\item Fix $\alpha_1,\alpha_2\in (0,1/2)$ with $\alpha_1\leq \alpha_2$, and $x_1,x_2\in \R_+$ with $x_1\geq x_2$. Any two solutions to \eqref{SDE:DMY} with $(x_0,\alpha)=(x_1,\alpha_1)$ and $(x_0,\alpha)=(x_2,\alpha_2)$, denoted by $X_{1}(t)$ and $X_{2}(t)$, respectively, 
satisfy $\P(X_{1}(t)\geq X_{2}(t),\,\forall\;t\in \R_+)=1$. 
\end{enumerate}
\end{prop}
\begin{proof}
(1$\cc$) \; We proceed with a comparison of SDEs. By \eqref{def:mu} and Corollary~\ref{cor:BESmon}. 
\begin{align}\label{dom:mu}
\mu^{\beta\da}_{\alpha}(x)\geq -2(\sqrt{2\beta |x|})^2+\frac{1}{2}>-4\beta |x|,\quad \forall\;x\in \R,\;\alpha\in [0,1/2).
\end{align}
Therefore, we compare \eqref{SDE:DMY} with the following SDE: 
\begin{align}\label{SDE:main1}
Y_t=x_0^2-\int_0^t 4\beta |Y_s|\d s+2\int_0^t \lvert Y_s\rvert^{1/2}\d \widetilde{W}_s.
\end{align}
Note that the coefficients of this SDE satisfy the following properties:
\begin{align}\label{cond:coeff}
\mbox{$y\mapsto -4\beta |y|$ is Lipschitz, and $\forall\;y_1,y_2\in \R$, $|2|y_1|^{1/2}-2|y_2|^{1/2}|\leq 2|y_1-y_2|^{1/2}$}.
\end{align}

Before using the SDE in \eqref{SDE:main1}, we show that it is strongly well-posed. First, the weak existence of solutions holds since the drift and noise coefficients are continuous and have at most linear growth \cite[Theorems~IV~2.3 and IV~2.4 on p.159 and p.163]{IW:SDE}. Also, the pathwise uniqueness in \eqref{SDE:main1} holds by \cite[2.13 Proposition on p.291]{KS:BM} and \eqref{cond:coeff}.
Then by \cite[3.23 Corollary on pp.310--311]{KS:BM}, the weak existence extends to the strong existence.  

To complete the proof of (1$\cc$), apply the comparison theorem of SDEs twice. (See  \cite[2.18 Proposition on p.293]{KS:BM}  for the version of the comparison theorem we use here.) Therefore, $X_t\geq Y_t$ for all $t$ with probability one by \eqref{dom:mu} and \eqref{cond:coeff}. Second, $Y_t\geq 0$ for all $t$ with probability one by \eqref{cond:coeff} since  \eqref{SDE:main1} with $x_0$ replaced by $0$ is solved by the zero process. We have proved (1$\cc$). \smallskip

\noindent (2$\cc$)\;  Fix $\alpha\in (0,1/2)$ and $x_0\geq 0$.
Let us begin with an observation. Given any two solutions $X,\widetilde{X}$ to \eqref{SDE:DMY}, the square-root noise coefficient implies that $X-\widetilde{X}$ has a zero semimartingale local time at level $0$ \cite[(3.4) Corollary on p.390]{RY}. Hence, by \cite[(3.2) Proposition on p.389]{RY}, the pathwise uniqueness in \eqref{SDE:DMY} is implied by the uniqueness in law in \eqref{SDE:DMY}.

It remains to prove the uniqueness in law in \eqref{SDE:DMY}. This is divided into two steps given below. Step~\hyperlink{prop:nonnegative-3-1}{1} will show the uniqueness in law, assuming that any solution $\{X_t\}$, which is necessarily nonnegative by (1$\cc$), satisfies the following property: 
\begin{align}\label{BESab:t<infty}
\textstyle \int_0^t\d s/X_s^{1-2\alpha}<\infty,\quad\forall\;t\in(0,\infty).
\end{align}
Step~\hyperlink{prop:nonnegative-3-2}{2} will complete the proof of (3$\cc$) by verifying \eqref{BESab:t<infty}. \smallskip

\noindent \hypertarget{prop:nonnegative-3-1}{{\sc Step 1.}}\; Our main goal of this step is to show that under the assumption of \eqref{BESab:t<infty},
\begin{align}
\E[G(X_{t\wedge \tau_n};t\leq T)]=\E^{\albe}_{x_0}[G(\rho^2_{t\wedge \tau_n'};t\leq T)]\label{Gir-inverse2}
\end{align}
for all $n\in \Bbb N$, $T\in(0,\infty)$ and bounded continuous $G\in \B_+(C([0,T],\R_+))$, where 
\begin{align}\label{def:tauntaun'}
\tau_n\textstyle\defeq\inf\{t\geq 0;\int_0^t\d s/X_s^{1-2\alpha}\geq n\},\quad 
\tau_n'\textstyle\defeq\inf\{t\geq 0;\int_0^t \d s/\rho_s^{2-4\alpha}\geq n\}.
\end{align}
To see why \eqref{Gir-inverse2} suffices for the required uniqueness in law, first,
note that by \eqref{BESab:t<infty}, $\P(\lim_n\tau_n\wedge T=T)=1$. Also, $\P^{\albe}_{x_0}(\lim_n\tau_n'\wedge T= T)=1$ since
$\P^{\al}_{x_0}(\lim_n\tau_n'=\infty)=1$ by Proposition~\ref{prop:mombdd} (2$\cc$) with $\eta=-2+4\alpha$,
and $\P^{\al},\P^{\albe}$ are locally equivalent to each other as probability measures by \eqref{def:BESab}. 
Hence, given the validity of \eqref{Gir-inverse2} for all $n$, we obtain upon passing $n\to\infty$ on both sides of \eqref{Gir-inverse2} that $\{X_t;0\leq t\leq T\}$ has the same law as $\{\rho^2_t;0\leq t\leq T\}$ under $\P^{\albe}_{x_0}$, as required.  

To prove \eqref{Gir-inverse2}, the idea is to postulate that $\{X_t\}$ arises from a change of measures in the same way that $\BES(-\alpha,\beta\da)$ arises from changing the measure of  $\BES(-\alpha)$ with the exponential martingale $\ms E(N)$. Recall \eqref{def:BESab} and Proposition~\ref{prop:key}. Accordingly, we consider 
\begin{align}\label{def:Ntaun}
\widetilde{N}_t\,\defeq\,\sqrt{2\beta}\int_0^t \frac{K_{1-\alpha}}{K_\alpha}(\sqrt{2\beta X_s})\d \widetilde{W}_s,
\end{align}
and  then define a probability measure $\Bbb Q_{n,T}$ by 
\begin{align}\label{def:QnT}
\d \Bbb Q_{n,T}\,\defeq\, \ms E(\widetilde{N})_{T\wedge \tau_n} \d \P.
\end{align}
Here, $\widetilde{N}$ is a well-defined continuous local martingale under $\P$ by \eqref{BESab:t<infty} and \eqref{asymp:K0}. Also,
$\Bbb Q_{n,T}$ is indeed a probability measure since by \eqref{asymp:K0} and \eqref{asymp:Kinfty}, $\langle\widetilde{N},\widetilde{N}\rangle_{\tau_n}\leq C(\alpha,\beta,n)$ so that Novikov's criterion \cite[(1.15) Proposition on p.332]{RY} is satisfied. 

We show that inverting the Girsanov transformation in \eqref{def:QnT} leads to \eqref{Gir-inverse2}. First, by \eqref{def:QnT}, we have the following for all $F\geq 0$:
\begin{align}
\E[F(X_{t\wedge \tau_n},\widetilde{W}_{t\wedge \tau_n};t\leq T)]&=\E^{\Bbb Q_{n,T}}\left[\frac{F(X_{t\wedge \tau_n},\widetilde{W}_{t\wedge \tau_n};t\leq T)}{\ms E(\widetilde{N})_{T\wedge \tau_n}}\right]\notag\\
&=\E^{\Bbb Q_{n,T}}[\ms E(N^{(n)})_{T\wedge \tau_n}F(X_{t\wedge \tau_n},\widetilde{W}_{t\wedge \tau_n};t\leq T)],\label{Gir-inverse}
\end{align}
where
\begin{align}
N^{(n)}_t&\,\defeq -\sqrt{2\beta}\int_0^{t} \frac{K_{1-\alpha}}{K_\alpha}(\sqrt{2\beta X_s})\d W^{(n)}_s,\notag\\
W^{(n)}_t&\,\defeq \,\widetilde{W}_t-\sqrt{2\beta}\int_0^{t}\frac{K_{1-\alpha}}{K_\alpha}(\sqrt{2\beta X_s})\d s.\label{def:Wnaux}
\end{align}
By Girsanov's theorem \cite[(1.7) Theorem on p.329]{RY}, $\{W^{(n)}_{t\wedge \tau_n};0\leq t\leq T\}$ under $\Bbb Q_{n,T}$ is a standard Brownian motion stopped at time $T\wedge \tau_n$, which justifies \eqref{Gir-inverse}. Also, by \eqref{def:BESQ} and \eqref{SDE:DMY},
$\{X_{t\wedge \tau_n};0\leq t\leq T\}$  under $\Bbb Q_{n,T}$ obeys the SDE of $\BES Q(-\alpha)$ stopped at time $T\wedge \tau_n$ and driven by $W_{\cdot\wedge T\wedge \tau_n}^{(n)}$. 
By the pathwise uniqueness in the SDE \eqref{def:BESQ} of $\BES Q(-\alpha)$,  we deduce that $X_{\cdot \wedge T\wedge \tau_n}$ is a.s. equal to a version of $\BES Q(-\alpha)$ stopped at $T\wedge \tau_n$. It follows from the definition \eqref{def:Nt} of $\{N_t\}$ and the definition of $\tau_n$ in \eqref{def:tauntaun'} that
\begin{align}
\E^{\Bbb Q_{n,T}}[\ms E(N^{(n)})_{T\wedge \tau_n}F(X_{t\wedge \tau_n},\widetilde{W}_{t\wedge \tau_n};t\leq T)]&=\E^{\al}_{x_0}[\ms E(N)_{T\wedge \tau_n'}F(\rho_{t\wedge \tau_n'}^2,\widetilde{W}_{t\wedge \tau_n'};t\leq T)]\notag\\
&=\E^{\albe}_{x_0}[F(\rho_{t\wedge \tau_n'}^2,\widetilde{W}_{t\wedge \tau_n'};t\leq T)],\label{Gir-inverse1}
\end{align}
Here, by \eqref{def:Wnaux}, $\widetilde{W}$ under $\P^{\al}$ relates to the driving noise $W$ of $\{\rho_t\}$ in \eqref{def:BES} via the equation in \eqref{def:tildeW}. Also, the last equality holds by \eqref{def:BESab} and Proposition~\ref{prop:key}. Combining \eqref{Gir-inverse} and \eqref{Gir-inverse1} proves \eqref{Gir-inverse2}.\smallskip

\noindent \hypertarget{prop:nonnegative-3-2}{{\sc Step~2.}}\; The idea for the verification of \eqref{BESab:t<infty}
 is to compare $\{X_t\}$ with $\BES Q(-(\alpha+\delta))$ for small enough $\delta>0$ whenever $\{X_t\}$ is nearly zero, and apply Proposition~\ref{prop:mombdd} (2$\cc$). 
 
 To carry this out, we first choose $\delta>0$ such that both of the following conditions hold:
\begin{gather}
\alpha+\delta\in (0,1/2),\label{perturbation:delta0}\\
1-2(\alpha+\delta)+(-2+4\alpha)>-1.\label{perturbation:delta}
\end{gather}
(There is no room to choose such $\delta$ if $\alpha=0$.)
Then Lemma~\ref{lem:lim0K0} allows $\gamma>0$ such that 
\begin{align}\label{mu:dom}
\sup_{x\in (0,2\gamma]}\sqrt{2\beta x} \frac{ K_{1-\alpha}}{K_{\alpha}}(\sqrt{2\beta x} )<\frac{\delta}{2}.
\end{align}

To proceed, fix $T\in (0,\infty)$, and choose subintervals of $[0,T]$ for the forthcoming comparison of $\{X_t\}$ and $\BES Q(-(\alpha+\delta))$ as $[S_n, T_n]$ for $n\geq 1$, where 
\begin{align}
S_1&\defeq \inf\{t\geq 0;X_t\leq \gamma\}\wedge T,\; 
&T_1&\defeq \inf\{t>S_1;X_t=2\gamma\}\wedge T,\label{Stopping:1}\\
S_{n+1}&\defeq \inf\{t>T_n;X_t=\gamma\}\wedge T,\;
&T_{n+1}&\defeq \inf\{t>S_{n+1};X_t=2\gamma\}\wedge T,\quad n\in \Bbb N.\label{Stopping:2}
\end{align}
That is, the stopping times in \eqref{Stopping:2} are defined inductively on $n\in \Bbb N$ starting with the stopping times in \eqref{Stopping:1}. ($S_1$ is defined with the weaker condition ``$X_t\leq \gamma$'' so that $S_1=0$ whenever $X_0\leq \gamma$.) Note that since $\{X_t\}$ is continuous, $T_n=T$ for all large $n$. This implies
\begin{gather}\label{ST:set}
\{t\in [0,T];X_t\leq \gamma\}\subset \bigcup_{n=1}^\infty [S_n,T_n].
\end{gather}
Additionally, we have the following property:
\begin{gather}
\mu^{\beta\da}_\alpha(X_t)=2-2R_\alpha(\sqrt{2\beta X_t})\geq 2(1-\alpha -\delta),\; \forall\;S_n\leq t\leq T_n,\mbox{ if }T_n-S_n>0,\label{ST:compare}
\end{gather}
by using \eqref{eq:mu}, \eqref{mu:dom}, and the fact that $X_t\leq 2\gamma$ for all $t\in [S_n,T_n]$ if $T_n-S_n>0$.

We now prove \eqref{BESab:t<infty}. By the strong well-posedness of the SDE of $\BES Q(-(\alpha+\delta))$, we can construct a version of $\BES Q(-(\alpha+\delta))$ starting from $X_{S_n}$ at time $S_n$ subject to the Brownian motion $\{\widetilde{W}_{t+S_n}-\widetilde{W}_t;t\geq 0\}$. Denote this version by $\{Y_n(t);t\in [S_n,\infty)\}$. By \eqref{ST:compare} and the comparison theorem of SDEs \cite[2.18~Proposition on p.293]{KS:BM}, $X_t\geq Y_n(t)$ for all $S_n\leq t\leq T_n$. It follows from this comparison and \eqref{ST:set} that
\begin{align}\label{Compar:sum}
\int_0^T\frac{\d r}{X_r^{1-2\alpha}}\leq \sum_{n=1}^\infty \int_{S_n}^{T_n}\frac{\d r}{Y_n(r)^{1-2\alpha}}+\int_{[0,T]\backslash \bigcup_{n=1}^\infty [S_n,T_n]}\frac{\d r}{\gamma^{1-2\alpha}}.
\end{align}
Note that the right-hand side is finite by the following two properties: (i) 
\[
\textstyle \int_{S_n}^T\d r/Y_n(r)^{1-2\alpha}<\infty
\]
by Proposition~\ref{prop:mombdd} (2$\cc$), \eqref{perturbation:delta0}, and \eqref{perturbation:delta}; (ii) the infinite series in \eqref{Compar:sum} is a only finite sum by the fact mentioned above \eqref{ST:set} that $T_n=T$ for all large $n$. We have proved \eqref{BESab:t<infty}.
The proof of the pathwise uniqueness in \eqref{SDE:DMY} for all $\alpha\in (0,1/2)$ and $x_0\geq 0$
 is complete. 
\smallskip  
 
 \noindent (3$\cc$)\; By \eqref{eq:mu} and Proposition~\ref{prop:BESmon} (3$\cc$), $\mu^{\beta\da}_\alpha(x)$ is strictly decreasing in $\alpha\in [0,1/2)$ for all fixed $x\in \R$. Hence, the required result follows immediately from another version of the comparison theorem of SDEs \cite[Theorem~VI.1.1 on pp.437--438]{IW:SDE}. Note that we use the pathwise uniqueness obtained in (2$\cc$) for this comparison of solutions when $\alpha_1=\alpha_2$.\mbox{\quad}
\end{proof}

\subsection{Proofs of Theorem~\ref{thm:2} (1$\cc$) for $\alpha=0$ and (2$\cc$) for $\alpha=\alpha_1=0$}\label{sec:thm2-12}
We first prove the following lemma, which sets up an approximation of \eqref{SDE:DMY} for $\alpha=0$.
Recall that the existence of strong solutions to \eqref{SDE:DMY} for any $\alpha\in(0,1/2)$ follows from Theorem~\ref{thm:2} (1$\cc$) and Proposition~\ref{prop:nonnegative} (2$\cc$). By Proposition~\ref{prop:nonnegative} (1$\cc$), these strong solutions are nonnegative.  

\begin{lem}\label{lem:continuity}
Fix $\beta\in (0,\infty)$ and $x_0\in \R_+$. For any $\alpha\in (0,1/2)$, 
 denote by $\{X_\alpha(t)\}$ the strong solution to \eqref{SDE:DMY} for a given one-dimensional standard Brownian motion $\{\widetilde{W}_t\}$. 
\begin{enumerate}[label={\rm ({\arabic*}$\cc$)}]
\item For all $T\in (0,\infty)$ and $0<\alpha<\alpha'\leq \alpha_0<1/2$, 
\begin{align}\label{approx:X}
\E\left[\sup_{0\leq t\leq T}\lvert X_{\alpha}(t)-X_{\alpha'}(t)\rvert^2\right]\leq C(\alpha_0,T)\alpha'.
\end{align}
\item For any $\{\alpha_n\}\subset (0,1/2)$ with $\alpha_n\searrow 0$, $\{X_{\alpha_n}(t)\}$ converges uniformly on compacts to a continuous process, say $\{X_0(t)\}$, as $n\to\infty$ with probability one. Also, \eqref{approx:X} extends to the case that $\alpha=0$ is allowed, and $\{X_0(t)\}$ is a weak solution to \eqref{SDE:DMY} for $\alpha=0$.  \smallskip 
\item Suppose that $\{\widetilde{W}_t\}$ is the driving Brownian motion of a given weak solution $\{\widetilde{X}_0(t)\}$ to \eqref{SDE:DMY} with $\alpha=0$.
 Then $\P(\widetilde{X}_0(t)=X_0(t),\;\forall\;t)=1$, where $\{X_0(t)\}$ is the limit of $\{X_{\alpha_n}(t)\}$ as $n\to\infty$ from (2$\cc$). 
\end{enumerate} 
\end{lem}
\begin{proof}
(1$\cc$)\; We begin by showing a few auxiliary properties. 
First, we have
\begin{align}
\begin{split}
0&\leq X_{\alpha}(t)-X_{\alpha'}(t)\\
&\leq \int_0^t \{\mu^{\beta\da}_{\alpha}[X_{\alpha}(s)]-\mu^{\beta\da }_{\alpha'}[X_{\alpha}(s)]\}\d s +2\int_0^t [X_{\alpha}(s)^{1/2}-X_{\alpha'}(s)^{1/2}]\d \widetilde{W}_s.\label{compare:X}
\end{split}
\end{align}
Here, the first inequality uses Proposition~\ref{prop:nonnegative} (3$\cc$). To get the second inequality, note that the equations of $\{X_\alpha(t)\}$ and $\{X_{\alpha'}(t)\}$ from \eqref{SDE:DMY} give 
\begin{align}
X_{\alpha}(t)-X_{\alpha'}(t)&=\int_0^t \{\mu_{\alpha}^{\beta\da}[X_{\alpha}(s)]-\mu^{\beta\da }_{\alpha'}[X_{\alpha}(s)]+\mu^{\beta\da }_{\alpha'}[X_{\alpha}(s)]-\mu_{\alpha'}^{\beta\da}[X_{\alpha'}(s)]\}\d s\notag\\
&\quad+2\int_0^t [X_{\alpha}(s)^{1/2}-X_{\alpha'}(s)^{1/2}]\d \widetilde{W}_s.\notag
\end{align}
The second inequality of \eqref{compare:X} follows by using the first inequality of \eqref{compare:X} and the fact that for $\alpha'\in [0,1/2)$, $x\mapsto \mu_{\alpha'}^{\beta\da}(x)$ is decreasing on $[0,\infty)$ by \eqref{eq:mu}
and Proposition~\ref{prop:BESmon} (4$\cc$).

Second, we show the following auxiliary properties:
\begin{gather}
\|\mu_0^{\beta\da}-\mu^{\beta\da}_{\alpha''}\|_\infty\less C(\alpha_0)\alpha'',\quad \forall\;0<\alpha''\leq \alpha_0<1/2,\label{unifbdd:mu}\\
\E\left[\left(\sup_{0\leq s\leq t}X_{\alpha''}(s)\right)^n\right]<\infty,\quad \forall\;n\in \Bbb N,\;t\in (0,\infty),\;\alpha''\in [0,1/2).\label{unifbdd:mom}
\end{gather}
To get \eqref{unifbdd:mu}, use \eqref{eq:mu}  and Proposition~\ref{prop:Fxunif}.
Also, \eqref{unifbdd:mom} holds because $\{X_{\alpha''}(t)\}$ can be pathwise dominated by a version of $\BES Q(-\alpha'')$. In more detail, this domination uses the strong existence of the SDE \eqref{def:BESQ} of $\BES Q(-\alpha'')$ \cite[3.23 Corollary on p.310--311]{KS:BM} and applies the comparison theorem of SDEs
\cite[2.18 Proposition on p.293]{KS:BM} to \eqref{def:BESQ} and \eqref{SDE:DMY}. Recall \eqref{def:mu} for the definition of $\mu^{\beta\da}_{\alpha''}$.

Third, we have
\begin{equation}\label{unifbdd:momd}
\E[\lvert X_{\alpha}(t)-X_{\alpha'}(t)\rvert ]=
\E[ X_{\alpha}(t)-X_{\alpha'}(t) ]\less C(\alpha_0)\alpha' t, \quad\forall\;0<t<\infty.
\end{equation}
Here, the first equality just follows from \eqref{compare:X}. To get the inequality of \eqref{unifbdd:momd}, note that by \eqref{unifbdd:mom}, the last term in \eqref{compare:X} is an $L^2$-martingale. By this property and \eqref{unifbdd:mu}, taking the expectations of both sides of the second inequality of \eqref{compare:X} yields the required result.

Finally, we obtain \eqref{approx:X} from the second inequality of  \eqref{compare:X} and combine several things: the inequality $(|x|+|y|)^2\less x^2+y^2$ for all $x,y\in \R$, \eqref{unifbdd:mu}, Doob's $L^2$-inequality, the inequality $\lvert \sqrt{|x|}-\sqrt{|y|}\rvert \leq \sqrt{\lvert x-y\rvert }$ for all $x,y\in \R$, and \eqref{unifbdd:momd}. \smallskip

\noindent {\rm (2$\cc$)}\; The required mode of convergence of $\{X_{\alpha_n}(t)\}$ follows from (1$\cc$) and the pathwise monotonicity of $\{X_\alpha(t)\}$ with respect to $\alpha$, as mentioned in the proof of (1$\cc$). Also, the extension of \eqref{approx:X} to $0\leq \alpha<\alpha'\leq \alpha_0<1/2$ follows from Fatou's lemma.

To show that $\{X_0(t)\}$ is a weak solution of \eqref{SDE:DMY} with $\alpha=0$, we pass the limit of both sides of \eqref{SDE:DMY} for $\alpha=\alpha_n\searrow 0$. To find the limiting equation, note that for all $t\in (0,\infty)$,
\begin{gather}
\left|\int_0^t\mu^{\beta\da}_{\alpha_n}[X_{\alpha_n}(s)]\d s-\int_0^t\mu^{\beta\da}_{0}[X_{0}(s)]\d s\right|\xrightarrow [n\to\infty]{\rm a.s.}0,\label{LIM:alpha1}\\
\left|2\int_0^t|X_{\alpha_n}(s)|^{1/2}\d\widetilde{W}_s-2\int_0^t|X_{0}(s)|^{1/2}\d\widetilde{W}_s\right|\xrightarrow [n\to\infty]{\Bbb P}0,\label{LIM:alpha2}
\end{gather}
in the sense of a.s. convergence and convergence in probability, respectively. Here, \eqref{LIM:alpha1} uses \eqref{unifbdd:mu}, the a.s. uniform convergence of $\{X_{\alpha_n}(t)\}$ to $\{X_0(t)\}$ on compacts from the previous paragraph, and the continuity of $\mu^{\beta\da}_0$. See Lemma~\ref{lem:lim0K0} for the continuity of $\mu^{\beta\da}_{0}$ at $0$. Also, \eqref{LIM:alpha2} uses the inequality $\lvert \sqrt{|x|}-\sqrt{|y|}\rvert \leq \sqrt{\lvert x-y\rvert }$ for all $x,y\in \R$ and the extension of \eqref{approx:X} to $0\leq \alpha<\alpha'\leq \alpha_0<1/2$ from the previous paragraph. Since $X_{\alpha_n}(t)\to X_0(t)$ a.s., \eqref{LIM:alpha1} and \eqref{LIM:alpha2} suffice to prove that  $\{X_0(t)\}$ solves \eqref{SDE:DMY} for $\alpha=0$.
\smallskip

\noindent {\rm (3$\cc$)} \; The proof follows upon noting that 
the proof of (1$\cc$) of the present lemma still applies in the case of $0=\alpha<\alpha'\leq \alpha_0<1/2$ with $X_0(t)$ replaced by $\widetilde{X}_0(t)$ so that \eqref{approx:X} extends to this range of $\alpha,\alpha'$. In more detail, this extension holds since the proof of Proposition~\ref{prop:nonnegative} (3$\cc$) can be extended to give $\widetilde{X}_0(t)\geq X_{\alpha_n}(t)$ for all $n$ and $t\geq 0$, now that the pathwise uniqueness in the SDE for $\{X_{\alpha_n}(t)\}$, $\alpha_n\in (0,1/2)$, holds by Proposition~\ref{prop:nonnegative} (2$\cc$).
\end{proof}

The following proposition completes the proof of Theorem~\ref{thm:2} (1$\cc$) and (2$\cc$).

\begin{prop}
Theorem~\ref{thm:2} (1$\cc$) for $\alpha=0$ and (2$\cc$) for $\alpha=\alpha_1=0$ are valid. 
\end{prop}
\begin{proof}
To obtain Theorem~\ref{thm:2} (1$\cc$) for $\alpha=0$, we first recall that $\BES (0,\beta\da)$ is the distributional limit of $\BES(-\alpha,\beta\da)$ as $\alpha\searrow 0$ \cite{DY:Krein}. Therefore, Theorem~\ref{thm:2} (1$\cc$) for $\alpha=0$ follows immediately from Lemma~\ref{lem:continuity} (2$\cc$) and the fact that $\BES Q(-\alpha,\beta\da)$ for $\alpha\in (0,1/2)$ satisfies \eqref{SDE:DMY}.

Regarding Theorem~\ref{thm:2} (2$\cc$) for $\alpha=\alpha_1=0$, first, note that the proof of Proposition~\ref{prop:nonnegative} (1$\cc$) also applies to $\alpha=0$. To obtain the pathwise uniqueness in \eqref{SDE:DMY} for $\alpha=0$, by the proof of Proposition~\ref{prop:nonnegative} (2$\cc$), again it is enough to show the uniqueness in law in \eqref{SDE:DMY} for $\alpha=0$. To get this property, we use Lemma~\ref{lem:continuity} (3$\cc$) and approximate the law of any solution $\{\widetilde{X}_0(t)\}$ to \eqref{SDE:DMY} for $\alpha=0$ by a sequence of laws of solutions to \eqref{SDE:DMY} for $\alpha=\alpha_n$. By the uniqueness in law in \eqref{SDE:DMY} for $\alpha\in (0,1/2)$, the law of $\{\widetilde{X}_0(t)\}$ therefore must be uniquely determined. 
Finally, regarding the comparison of solutions, note that the proof of Proposition~\ref{prop:nonnegative} (3$\cc$) still applies to $0=\alpha_1\leq \alpha_2<1/2$. The proof is complete.
\end{proof}

\subsection{Proof of Theorem~\ref{thm:2} (3$\cc$)}\label{sec:thm2-3}
For $\alpha\in (0,1/2)$, \eqref{SDE:DMYsqrt}  follows readily from Corollary~\ref{cor:SDE>0} and the SDE \eqref{def:BES} obeyed by $\BES(-\alpha)$. It remains to consider the case $\alpha=0$.

To prove \eqref{SDE:DMYsqrt} for $\alpha=0$, we use \eqref{SDE:DMYsqrt} for $\alpha_0=\alpha\in (0,1/2)$ and passing $\alpha_0\searrow 0$. Specifically, let $\{X_\alpha(t)\}$, $\alpha\in (0,1/2)$, be defined in Lemma~\ref{lem:continuity}, such that $\{\widetilde{W}_t\}$ is the driving Brownian motion of the SDE of $\{X_0(t)\}$ as in \eqref{SDE:DMY} for $X_0(t)\,\defeq\,\rho_t^2$, and then set
\begin{align}\label{rhoalpha:special}
\rho_{\alpha}(t)\,\defeq\, X_\alpha(t)^{1/2},\quad \alpha\in [0,1/2).
\end{align}
Hence, \eqref{SDE:DMYsqrt} for $\alpha=0$ follows from the case of $\alpha\in (0,1/2)$ already proven if for any sequence $\{\alpha_n\}\subset (0,1/2)$,
\begin{align}
\lim_{n\to\infty}\textstyle \int_0^t 1/\rho_{\alpha_n}(s)\d s&=\textstyle \int_0^t 1/\rho_0(s)\d s<\infty,\label{additive:nonexplosion1}\\
\lim_{n\to\infty}\textstyle \int_0^t(K_{1-\alpha_n}/K_{\alpha_n})(\sqrt{2\beta}\rho_{\alpha_n}(s))\d s&=\textstyle \int_0^t(K_{1}/K_{0})(\sqrt{2\beta}\rho_{0}(s))\d s<\infty. \label{additive:nonexplosion2}
\end{align}
Then note that \eqref{additive:nonexplosion1} can be justified by dominated convergence and (1)--(3), and \eqref{additive:nonexplosion2} can be justified by dominated convergence and (1)--(4), where (1)--(4) are observations made below: 
\begin{itemize}
\item [(1)] For all $\alpha\in (0,1/2)$, 
\begin{align}
\textstyle \int_0^t\d s/\rho_\alpha(s)<\infty,\quad \P^{\albe}_{x_0}\mbox{-a.s.}
\end{align}
This bound follows immediately from
\eqref{def:BESab} since $\int_0^t \d s/\rho_s<\infty$ a.s. under $\BES(-\alpha)$ by Lemma~\ref{lem:BESint} with $\eta=-1$.
 \item [(2)] $\rho_{\alpha_1}(t)\geq \rho_{\alpha_2}(t)$ a.s. for any given $0\leq \alpha_1\leq\alpha_2<1/2$ by Theorem~\ref{thm:2} (2$\cc$).
\item [(3)] $\rho_\alpha(\cdot)\to \rho_0(\cdot)$ uniformly on compacts a.s. as $\alpha\searrow 0$ by Lemma~\ref{lem:continuity} (3$\cc$).
\item [(4)] $(K_{1-\alpha}/K_\alpha)(x)\to (K_1/K_0)(x)$ uniformly on compacts in $x\in (0,\infty)$ by \eqref{def:K}, and
\[
(K_{1-\alpha}/K_\alpha)(x)\leq R_\alpha(x)/x\leq R_{1/4}(x)/x=1/(4x)+(K_{3/4}/K_{1/4})(x)\less 1+1/x
\]
for all $\alpha\in [0,1/4]$ and $x\in (0,\infty)$. 
Here, the second inequality
follows from Proposition~\ref{prop:BESmon} (3$\cc$) for $R_\alpha(\cdot)$ defined in \eqref{def:Ralpha}
and the last inequality uses \eqref{asymp:K0} and \eqref{asymp:Kinfty}.
\end{itemize} 
The proof of  Theorem~\ref{thm:2} (3$\cc$) is complete.
 
\subsection{Proof of Theorem~\ref{thm:2} (4$\cc$)}\label{sec:thm2-4}
The following lemma obtains a more detailed version of Theorem~\ref{thm:2} (4$\cc$). Recall that $\BES Q(-\alpha,\beta\da)$ is characterized by the SDE in \eqref{SDE:DMY} thanks to Theorem~\ref{thm:2} (1$\cc$) and (2$\cc$).

\begin{lem}\label{lem:explosion}
For all $\beta\in (0,\infty)$ and $x_0\in (0,\infty)$, the following holds:
\begin{enumerate}[label={\rm ({\arabic*}$\cc$)}]
\item For all $\alpha\in [0,1/2)$, $\int_0^{T_0(\rho)}\d r/\rho_r^2=\infty$ a.s. under $\P_{x_0}^{\albe }$. In particular, $t\mapsto \int_0^t \d r/\rho_r^2$ is a $[0,\infty]$-valued continuous function.\smallskip 
\item For $\alpha\in [0,1/2)$, let $\{X_{\alpha}(t)\}$ be as in Lemma~\ref{lem:continuity}, and let $\{\rho_{\alpha}(t)\}$ be defined by \eqref{rhoalpha:special}. Then for any $\{\alpha_n\}\subset(0,1/2)$ with $\alpha_n\searrow 0$, $T_0(\rho_{\alpha_n})\nearrow T_0(\rho_0)<\infty$ a.s.\smallskip 
\item In the same setting of (2$\cc$), 
\[
\{(\rho_{\alpha_n}(t),\textstyle \int_0^t \d r/\rho_{\alpha_n}(r)^2)\}\xrightarrow[n\to\infty]{\rm a.s.} \{(\rho_0(t),\int_0^t \d r/\rho_{0}(r)^2)\}\mbox{ in } C(\R_+,\R_+\times [0,\infty]),
\]
where $C(\R_+,\R_+\times [0,\infty])$ is equipped the topology of uniform convergence on compacts.  
\end{enumerate}
\end{lem}
\begin{proof}
We will use the fact that $T_0(\rho)<\infty$ a.s. under $\P^{\albe}_{x_0}$ for all $\alpha\in [0,1/2)$. This property in the case of $\alpha=0$ has been obtained in \cite[Theorem~2.1 on p.883]{DY:Krein}, and for $\alpha\in (0,1/2)$, it is immediate by the following identity \cite[(2) in (3.2) Examples on p.886]{DY:Krein} :
\begin{align}\label{preGIG}
\E^{\albe}_{x_1}[\e^{-qT_0(\rho)}]=\widehat{K}_\alpha(\sqrt{2(\beta+q)}x_1)/\widehat{K}_\alpha(\sqrt{2\beta}x_1),\quad \forall\;x_1,q\in \R_+.
\end{align}

\begin{rmk}
Identity \eqref{preGIG} can be derived from \eqref{eq:T0alpha} and the following identity explained at the end of this remark:
\begin{align}
\begin{split}\label{preGIG1}
\E^{\albe}_{x_0}[\e^{-q(T_0(\rho)\wedge t)}]
&=\E^{\al}_{x_0}\left[\frac{\widehat{K}_\alpha(\sqrt{2\beta}\rho_{t})}{\widehat{K}_\alpha(\sqrt{2\beta}x_0)}\e^{-(q+\beta) t}\1_{\{t\leq T_0(\rho)\}}\right]\\
&\quad +\E^{\al}_{x_0}\left[\frac{\e^{-(q+\beta) T_0(\rho)}\1_{\{t>T_0(\rho)\}}}{\E_{x_0}^{\al}[\e^{-\beta T_0(\rho)}]}\right],\quad \forall\;0<t<\infty.
\end{split}
\end{align}
Indeed, to get  \eqref{preGIG}, take $q=0$ to see that the first term on the right-hand side of \eqref{preGIG1} vanishes as $t\to\infty$. Then apply \eqref{eq:T0alpha} to the limit of the right-hand side of \eqref{preGIG1} as $t\to\infty$.

To see \eqref{preGIG1}, use \eqref{def:BESab} to convert $\E^{\albe}$ to $\E^{\al}$, and the last term of \eqref{preGIG1} follows by conditioning on $\sigma(\rho_s;s\leq T_0(\rho))\cap \{t>T_0(\rho)\}$ and using \eqref{def:BESab}, \eqref{def:hK0} and \eqref{eq:T0alpha}.
\qed
\end{rmk}

\noindent {\rm (1$\cc$)}\; The a.s. explosion $\int_0^{T_0(\rho)}\d r/\rho_r^2=\infty$ under $\P^{\albe}_{x_0}$ for all $\alpha\in [0,1/2)$ can be obtained by verifying an analytic criterion in \cite[Theorem~1 on pp.75--76]{Erickson}. For completeness, below, we include a proof based on the method of \cite[Section~6, on pp.10--11]{MU:Int}. This method uses the first Ray--Knight theorem
and a criterion equivalent to the following one with respect to a one-dimensional standard Brownian motion $\{W(y);y\geq 0\}$ starting from $0$:
\begin{itemize}
\item [(A)] $y\cdot f(y)\in L^1_{\rm loc}(0+)$ implies $\int_{0+}|f(y)|W(y)^2\d y<\infty$ a.s.
\item [(B)] $y\cdot f(y)\notin L^1_{\rm loc}(0+)$ implies $\int_{0+}|f(y)|W(y)^2\d y=\infty$ a.s.
\end{itemize}

We will prove the required property $\int_0^{T_0(\rho)}\d r/\rho_r^2=\infty$ for $\{\rho_t\}\sim\BES(-\alpha,\beta\da)$, $\alpha\in [0,1/2)$, via an alternative representation of $\int_0^{T_0(\rho)}\d r/\rho_r^2$. To obtain this representation, first, note that the scale function of $\{\rho_t\}$ can be chosen as
\begin{align}\label{def:scale}
\textstyle s(x)=s^{\beta\da }_\alpha(x)\,\defeq\,\int_0^x \d y/[y^{1-2\alpha}\widehat{K}_\alpha(\sqrt{2\beta}y)^2].
\end{align}
This formula can be deduced from the SDE \eqref{SDE:DMYsqrt}  of $\{\rho_t\}$ in Theorem~\ref{thm:2} (3$\cc$), a slight modification of  \cite[1$\cc$) of (3.20) Exercise on p.311]{RY}, and the observation that 
\begin{align}
\frac{1-2\alpha}{2x}-\sqrt{2\beta}\frac{K_{1-\alpha}}{K_\alpha}(\sqrt{2\beta}x)=\frac{1}{2}\frac{\d}{\d x}\log x^{1-2\alpha}+\frac{1}{2}\frac{\d}{\d x}\log \widehat{K}_\alpha(\sqrt{2\beta}x)^2.\label{drift:rule}
\end{align}
Here, with $\widehat{K}_\nu(x)$ from \eqref{def:hK}, \eqref{drift:rule} uses the rule $(\d/ \d x)\widehat{K}_\nu(x)=-x^\nu K_{\nu-1}(x)$
 \cite[the third identity in (5.7.9) on p.110]{Lebedev} and the even parity of $\nu\mapsto K_\nu(\cdot)$ \cite[(5.7.10) on p.110]{Lebedev}.

Now, to represent $\int_0^{T_0(\rho)}\d r/\rho_r^2$, 
we use the effect that $Y_{t}\,\defeq s(\rho_{t\wedge T_0(\rho)})$ is a continuous local martingale with $T_0(Y)=T_0(\rho)$. By \eqref{SDE:DMYsqrt} and the Dambis--Dubins--Schwarz theorem \cite[(1.6) Theorem on p.181]{RY}, there exists a one-dimensional standard Brownian motion $B$ such that $Y_t=B_{A_t}$, where $A_t\,\defeq\,\int_0^t[s'(\rho_r)]^2\d r$. Hence, by using the bi-continuous Bronwian local time $\{\ell^y_{t}\}$ of $B$ from Tanaka's formula,
\begin{align}
\int_0^{T_0(\rho)}\frac{\d r}{\rho_r^2}&=\int_0^{T_0(Y)}\frac{\d r}{[s^{-1}(Y_r)]^2}
=\int_0^{T_0(Y)}\frac{1}{[s^{-1}(B_{A_r})]^2}\cdot \frac{[s'(\rho_r)]^2}{[s'(s^{-1}(B_{A_r}))]^2}\d r\notag\\
&=\int_0^{T_0(B)}\frac{1}{[s^{-1}(B_v)]^2} \cdot\frac{1}{ [s'(s^{-1}(B_v))]^{2}}\d v\label{T0infty-00}\\
&=\int_0^\infty\frac{1}{[s^{-1}(y)]^2} \cdot \frac{1}{[s'(s^{-1}(y))]^{2}}\cdot\ell^y_{T_0(B)}\d y\label{T0infty-0}.
\end{align}
Here, the usual change of variables formula gives \eqref{T0infty-00}, and \eqref{T0infty-0} uses the occupation times formula. More specifically, the upper limit $T_0(B)$ of the integral in \eqref{T0infty-00} follows since $t\mapsto A_t$ is bijective up to $t=T_0(Y)$ and $A_{T_0(Y)}=T_0(B)$. This property   $A_{T_0(Y)}=T_0(B)$  is implied by the definition $Y_t=B_{A_t}$ and the fact that $T_0(Y)=T_0(\rho)<\infty$.

To complete the proof of (1$\cc$), we show the explosion of the integral in \eqref{T0infty-0}. 
By the first Ray--Knight theorem for Brownian motion \cite[(2.2) Theorem on p.455]{RY}, $\{\ell^{y}_{T_0(B)};0\leq y\leq x_0\}$ is a version of $\BES Q(0)$ starting from $0$ and restricted to the time interval $[0,x_0]$. Hence, $\{\ell^y_{T_0(B)};0\leq y\leq x_0\}$ has the same law as $\{W_1(y)^2+W_2(y)^2;0\leq y\leq x_0\}$, for independent one-dimensional standard Brownian motions $W_1(\cdot),W_2(\cdot)$ starting from $0$. This leads to the criterion by (A) and (B) stated before \eqref{def:scale} for the convergence or divergence of the integral in \eqref{T0infty-0}; see also \cite{Spitzer}.
To use that criterion by (A) and (B), note that, with \eqref{def:hK}, we have the following bounds for $0<x\leq 1/2$ by \eqref{def:scale}  and \eqref{asymp:K0}:
\begin{align}
s'(x)\less 1/(x\lvert \log x\rvert^2) &\quad \mbox{if $\alpha=0$},&& \quad s'(x)\less C(\alpha,\beta)/x^{1-2\alpha}&& \mbox{if  $\alpha\in (0,1/2)$},\label{s':asymp}\\
s(x)\more 1/\lvert \log x\rvert &\quad \mbox{if $\alpha=0$},&&\quad s(x)\more C(\alpha,\beta) x^{2\alpha}&& \mbox{if  $\alpha\in (0,1/2)$},\label{s:asymp}
\end{align}
where the bound in \eqref{s:asymp} for $\alpha=0$ uses the formula $\int 1/[y(\log y)^2]=-1/\log y+C$ for $0<y<1$. Then condition (B) holds since the corresponding integral $\int_{0+}y\cdot |f(y)|\d y$ is
\begin{align}
&\quad\;\int_{0+}y\cdot \frac{1}{[s^{-1}(y)]^2} \cdot \frac{1}{[s'(s^{-1}(y))]^{2}}\d y=\int_{0+}s(x)\cdot \frac{1}{x^2}\cdot \frac{1}{s'(x)}\d x\notag\\
&\more 
\begin{cases}
\displaystyle C(\beta) \int_{0+} \frac{1}{\lvert \log x\rvert }\cdot \frac{1}{x^2}\cdot x \lvert \log x\rvert ^2\d x=\infty&\mbox{if }\alpha=0,\\
\vspace{-.4cm}\\
\displaystyle C(\alpha,\beta)\int_{0+}x^{2\alpha}\cdot \frac{1}{x^2}\cdot x^{1-2\alpha} \d x=\infty&\mbox{if }\alpha\in (0,1/2).
\end{cases}\label{T0infty-2}
\end{align}
We have proved the divergence of the integral in \eqref{T0infty-0}, and hence, $\int_0^{T_0(\rho)}\d r/\rho_r^2=\infty$. \smallskip

\noindent {\rm (2$\cc$)}\;  The pathwise increasing monotonicity of $\{\rho_{\alpha_n}(t)\}$ in $n$ from Theorem~\ref{thm:2} (2$\cc$) yields that $T_0(\rho_{\alpha_n})$ is increasing in $n$ and bounded by $ T_0(\rho_0)$. This gives $\lim_n T_0(\rho_{\alpha_n})\leq T_0(\rho_0)$. To obtain the converse inequality, note that since $\rho_{\alpha_n}(\cdot)\nearrow \rho_{0}(\cdot)$ uniformly on compacts with probability one by Lemma~\ref{lem:continuity} (3$\cc$), we have $0=\rho_m(T_0(\rho_{\alpha_m}))\to \rho_0(\lim_{n}T_0(\rho_{\alpha_n}))$ as $m\to\infty$, and so, $\lim_{n}T_0(\rho_{\alpha_n})\geq T_0(\rho_0)$ must hold. We conclude that $\lim_{n}T_0(\rho_{\alpha_n})= T_0(\rho_0)$.\smallskip

\noindent {\rm (3$\cc$)}\; The topology of $\R\cup\{\pm\infty\}$ admits the metric of $(x,y)\mapsto |\tan^{-1}(x)-\tan^{-1}(y)|$.
Hence, by the continuity of $t\mapsto \int_0^t \d r/\rho_0(r)^2$ from (1$\cc$),
the required a.s. uniform convergence on compacts follows 
if, whenever $t_n\to t_0<\infty$ as $n\to \infty$, we have $\int_0^{t_n}\d r/\rho_{\alpha_n}(r)^2\to \int_0^{t_0}\d r/\rho_{0}(r)^2$. To verify this property, consider three possibilities of $t_0$: $t_0>T_0(\rho_0)$, $t_0<T_0(\rho_0)$ and $t_0=T_0(\rho_0)$. For the first case of $t_0>T_0(\rho_0)$, the required limit holds trivially since $\int_0^{t_0}\d r/\rho_{0}(r)^2=\infty$ and $\infty=\int_0^{t_n}\d r/\rho_{\alpha_n}(r)^2$ for all large $n$ by (1$\cc$) and (2$\cc$) proven above. Second, the case of $t_0<T_0(\rho_0)$ can be handled easily by (2$\cc$) and dominated convergence. Finally, for $t_0=T_0(\rho_0)$, it suffices to note that $\int_0^{t_n}\d r/\rho_{0}(r)^2\leq \int_0^{t_n}\d r/\rho_{\alpha_n}(r)^2$ and $\int_0^{t_0}\d r/\rho_{0}(r)^2=\infty$.  The proof of Lemma~\ref{lem:explosion} is complete.
\end{proof}

\subsection{Ratios of the Macdonald functions}\label{sec:RATIO}
In this subsection, we consider the drift coefficients $\mu^{\beta\da}_\alpha(x)$ of the SDEs in \eqref{SDE:DMY} for $\BES(-\alpha,\beta \da)$. The properties we need are the monotonicity in $\alpha$ and $x$ and the modulus of continuity at $\alpha=0$ for fixed $x$. So
we will prove such properties of the following functions in the more general context of $\alpha\in [0,\infty)$:
\begin{align}\label{def:Ralpha}
R_\alpha(x)\,\defeq \, 
\begin{cases}
\displaystyle \alpha+x\frac{K_{1-\alpha}(x)}{K_\alpha(x)},& x\in (0,\infty),\\
\alpha,&x=0,
\end{cases}
\end{align}
now that  in the particular case of $\alpha\in [0,1/2)$, the definition \eqref{def:mu} of $\mu^{\beta\da}_\alpha(x)$ entails \eqref{eq:mu}.
See Propositions~\ref{prop:BESmon} and~\ref{prop:Fxunif} for the required properties of $R_\alpha(x)$.
Note that by Lemma~\ref{lem:lim0K0} proven below, setting $R_\alpha(0)$ to be $\alpha$  in \eqref{def:Ralpha} ensures the continuity  of $R_\alpha(\cdot)$ at $x=0$.

 The technicality of the proofs of this subsection is in the scope of special functions. We will
 use the following basic properties of the Macdonald functions $K_\nu(x)$: for $\nu\in \R$,
\begin{align}
&K_\nu(x)=K_{-\nu}(x),\label{K:sym}\\
&K_{\nu-1}(x)+K_{\nu+1}(x)=-2K'_\nu(x),\label{K:rec1}\\
&K_{\nu-1}(x)-K_{\nu+1}(x)=-\frac{2\nu}{x}K_\nu(x),\label{K:rec2}\\
&K_\nu(x)\sim 
\begin{cases}
-\log x\hspace{1.89cm} \mbox{ as $x\searrow 0$, }&\nu=0,\\
2^{\lvert \nu\rvert-1}\Gamma(\lvert \nu\rvert)/x^{\lvert \nu\rvert}\quad \mbox{ as $x\searrow 0$}, &\nu\neq 0,
\end{cases}\label{asymp:K0}\\
&K_\nu(x)\sim \sqrt{\frac{\pi}{2 x}}\e^{-x}\quad \mbox{ as }x\to \infty.\label{asymp:Kinfty}
\end{align}
See \cite[(5.7.2) and (5.7.3) on pp.108--109]{Lebedev} for \eqref{K:sym},
\cite[the last two equations of (5.7.9) on p.110]{Lebedev} for \eqref{K:rec1} and \eqref{K:rec2}, and
\cite[(5.16.4)  and (5.16.5) on p.136]{Lebedev} for \eqref{asymp:K0} and \eqref{asymp:Kinfty}. 
Additionally, we will use the following alternative integral representations of the Macdonald functions $K_\nu(x)$:  for $x\in (0,\infty)$, 
\begin{align}
K_\nu(x)&=\int_0^\infty \e^{-x\cosh(t)}\cosh(\nu t)\d t,\quad \forall\;\nu\in \R,\label{K:1}\\
K_\nu(x)&=\sqrt{\frac{\pi}{2x}}\frac{\e^{-x}}{\Gamma(\nu+1/2)}\int_0^\infty \e^{-t}t^{\nu-1/2}\left(1+\frac{t}{2x}\right)^{\nu-1/2}\d t, \quad\forall\;\nu\in (-1/2,\infty),\label{K:3}\\
K_\nu(x)&=\frac{2^{\nu-1}}{x^\nu}\int_0^\infty \e^{-s}\int_{x^2/(4s)}^\infty \e^{-u}u^{\nu-1}\d u\d s,\quad\forall\;\nu\in \R.\label{Knu:int}
\end{align}
See \cite[(5.10.23) on p.119]{Lebedev} for \eqref{K:1}, and \cite[(14.131) on p.691]{AWH:MP} for \eqref{K:3}. To get \eqref{Knu:int}, we rewrite \eqref{def:K} as
\begin{align*}
K_{\nu}(x)&=\frac{x^\nu}{2^{\nu+1}}\int_0^\infty \e^{-s}\int_0^s \e^{-x^2/(4t)}t^{-\nu-1}\d t\d s,
\end{align*}
and then the substitution $u=x^2/(4t)$ for the right-hand side leads to \eqref{Knu:int}.

\begin{lem}\label{lem:lim0K0}
For all $\alpha\in[0,\infty)$, $ \lim_{x\searrow 0}xK_{1-\alpha}(x)/K_\alpha(x)=0$. 
Accordingly,  the value of $0K_{1-\alpha}(0)/K_\alpha(0)$ is understood as $0$. 
\end{lem}
\begin{proof}
 The limit holds since by \eqref{asymp:K0},
\begin{align*}
\lim_{x\searrow 0}x\frac{K_{1-\alpha}(x)}{K_\alpha(x)}&=\lim_{x\searrow 0}\left.\left(x\frac{2^{\lvert 1-\alpha\rvert-1}\Gamma(\lvert 1-\alpha\rvert)}{x^{\lvert 1-\alpha\rvert}}\right)\right/\left(\frac{2^{\alpha-1}\Gamma(\alpha)}{x^{\alpha}}\right)=0,\;\;\; \alpha\notin \{0,1\},\\
\lim_{x\searrow 0}x\frac{K_{1}(x)}{K_0(x)}&=\lim_{x\searrow 0}\left.\left(x\frac{2^{1-1}\Gamma(1)}{x^{1}}\right)\right/-\log x=0,\\
\lim_{x\searrow 0}x\frac{K_{0}(x)}{K_1(x)}&=\lim_{x\searrow 0}-x\log x\left/\left(\frac{2^{1-1}\Gamma(1)}{x^{1}}\right)\right.=0.
\end{align*}
In more detail, the limit in the first line hold since $x^{1-|1-\alpha|+\alpha}\to 0$ as $x\to 0$.
\end{proof}

We now begin the proof of the required monotonicity of $R_\alpha(x)$ in $\alpha$ and $x$ by the following proposition. Proposition~\ref{prop:BESmon} (3$\cc$) is crucial to the proof and will be used again later on when we deal with the modulus of continuity of $\alpha\mapsto R_\alpha(x)$ at $\alpha=0$.
 See \cite{IM:BES} for related results. 
 
\begin{prop}\label{prop:BESmon}
The following properties of $R_\alpha(x)$ defined by \eqref{def:Ralpha} hold:
\begin{enumerate}[label={\rm ({\arabic*}$\cc$)}]
\item For all $\alpha\in [0,\infty)$ and $x\in (0,\infty)$, 
\begin{align}\label{Ralpha:alt}
R_\alpha(x)=x\frac{K_{\alpha-1}(x)+K_{\alpha+1}(x)}{2K_\alpha(x)}.
\end{align} 
\item For all $\alpha\in [0,\infty)$ and $x,\vep\in (0,\infty)$,
\begin{align}
\begin{split}
R_{\alpha+\vep}(x)-R_{\alpha}(x) =2(2\alpha+\vep)\int_0^\infty \frac{K_{2\alpha+\vep}[2x\cosh (t)]}{K_{\alpha+\vep}(x)K_\alpha(x)}
\sinh(\vep t)\tanh(t)\d t.\label{diff:Ralpha}
\end{split}
\end{align}

\item For all $x\in [0,\infty)$, $\alpha\mapsto R_\alpha(x)$ is strictly increasing to infinity on $[0,\infty)$.\smallskip 
\item For all $\alpha\in [0,\infty)$, $x\mapsto R_\alpha(x)$ is strictly increasing  to infinity on $[0,\infty)$.  
\end{enumerate}
\end{prop}
\begin{proof}
\noindent {\rm (1$\cc$)}\;
Take the arithmetic averages of both sides of \eqref{K:rec1} and \eqref{K:rec2} with $\nu=\alpha$, and use \eqref{K:sym} to write $K_{\alpha-1}(x)$ as $K_{1-\alpha}(x)$. It follows that
\[
K_{1-\alpha}(x)=-K'_\alpha(x)-\frac{\alpha}{x}K_\alpha(x).
\]
Hence, by the definition \eqref{def:Ralpha} of $R_\alpha(x)$ for $x\in (0,\infty)$,
\begin{align*}
R_\alpha(x)=\frac{\alpha K_\alpha(x)-xK_\alpha'(x)-\alpha K_\alpha(x)}{K_\alpha(x)}
=\frac{x\cdot\frac{1}{2}[K_{\alpha-1}(x)+K_{\alpha+1}(x)]}{K_\alpha(x)},
\end{align*}
where the last equality holds by \eqref{K:rec1} with $\nu=\alpha$. We have proved \eqref{Ralpha:alt}. \medskip

\noindent {\rm (2$\cc$)}\; To obtain \eqref{diff:Ralpha}, we use a Nicholson-type formula given by
\begin{align}\label{Nik}
K_\mu(x)K_\nu(x)=2\int_0^\infty K_{\mu+\nu}[2x\cosh(t)]\cosh[(\mu-\nu)t]\d t,\quad \forall\;\mu,\nu\in \R,
\end{align}
which follows by applying \cite[Problem~9 on p.140]{Lebedev} to $K_\mu(x)K_{-\nu}(x)$, and then using \eqref{K:sym} and
a substitution of $t$ by $t/2$. By using \eqref{Ralpha:alt} and then \eqref{Nik} in the next two equalities, 
\begin{align}
&\quad \frac{K_{\alpha+\vep}(x)K_{\alpha}(x)}{x}[R_{\alpha+\vep}(x)-R_\alpha(x)]\notag\\
&=\frac{1}{2}\big[K_{\alpha}(x)K_{\alpha+\vep-1}(x)+K_{\alpha}(x)K_{\alpha+\vep+1}(x)-K_{\alpha+\vep}(x)K_{\alpha-1}(x)-K_{\alpha+\vep}(x)K_{\alpha+1}(x)\big]\notag\\
&=\int_0^\infty \big\{K_{2\alpha+\vep-1}[2x\cosh (t)]\cosh[(1-\vep)t]+ K_{2\alpha+\vep+1}[2x\cosh (t)]\cosh[(1+\vep)t]\notag\\
&\quad - K_{2\alpha+\vep-1}[2x\cosh (t)]\cosh[(1+\vep)t] - K_{2\alpha+\vep+1}[2x\cosh (t)]\cosh[(1-\vep)t]\big\}\d t\notag\\
&=\int_0^\infty \{K_{2\alpha+\vep-1}[2x\cosh(t)]-K_{2\alpha+\vep+1}[2x\cosh (t)]\}\{\cosh[(1-\vep)t]-\cosh[(1+\vep)t]\}\d t\notag\\
\begin{split}
&=\frac{2(2\alpha+\vep)}{x}\int_0^\infty K_{2\alpha+\vep}[2x\cosh (t)]\sinh(\vep t)\tanh(t)\d t.
\label{coeff:ineq1}
\end{split}
\end{align}
The last equality uses \eqref{K:rec2} with $\nu=2\alpha+\vep$ and the identity $\cosh(a+b)-\cosh(a-b)=2\sinh(a)\sinh(b)$ for $a=t$ and $b=-\vep t$. Multiplying both sides of \eqref{coeff:ineq1} by $x/[K_{\alpha+\vep}(x)K_{\alpha}(x)]$ proves \eqref{diff:Ralpha}. \smallskip 

\noindent {\rm (3$\cc$)}\; By \eqref{diff:Ralpha}, $R_{\alpha+\vep}(x)>R_\alpha(x)$ for $\vep,x>0$. 
Hence, $\alpha\mapsto R_\alpha(x)$ is strictly increasing on $\in [0,\infty)$ for any $x\in (0,\infty)$. This extends to $x=0$ by $R_\alpha(0)=\alpha$. The property $R_\alpha(x)\to\infty$ as $\alpha\to\infty$ is obvious. \smallskip

\noindent {\rm (4$\cc$)}\; The proof of the monotonicity is a similar application of \eqref{Nik}. First, write
\begin{align*}
&\quad K_\alpha(x)^2\frac{\d}{\d x}\left(x\frac{K_{1-\alpha}(x)}{K_\alpha(x)}\right)
=K_{1-\alpha}(x)K_\alpha(x)+xK_\alpha(x)^2\frac{\d}{\d x}\left(\frac{K_{1-\alpha}(x)}{K_\alpha(x)}\right)\\
&=K_{1-\alpha}(x)K_\alpha(x)+x[K_\alpha(x)K'_{1-\alpha}(x)-K_{1-\alpha}(x)K'_\alpha(x)]\\
&=K_{1-\alpha}(x)K_\alpha(x)-\frac{x}{2}\{K_\alpha(x)[K_{-\alpha}(x)+K_{2-\alpha}(x)]-K_{1-\alpha}(x)[K_{\alpha-1}(x)+K_{\alpha+1}(x)]\},
\end{align*}
where the last equality uses \eqref{K:rec1}. By \eqref{Nik}, the last equality gives
\begin{align*}
&\quad K_\alpha(x)^2\frac{\d}{\d x}\left(x\frac{K_{1-\alpha}(x)}{K_\alpha(x)}\right)\\
&=2\int_0^\infty K_{1}[2x \cosh(t)]\cosh[(1-2\alpha)t]\d t\\
&\quad -x\int_0^\infty \big\{K_{0}[2x\cosh (t)]\cosh[(-2\alpha)t]+ K_2[2x\cosh(t)]\cosh(2-2\alpha)t]\\
&\quad - K_0[2x\cosh (t)]\cosh[(2-2\alpha)t] - K_2[2x\cosh(t)]\cosh[(-2\alpha)t]\big\}\d t\\
&=2\int_0^\infty K_{1}[2x \cosh(t)]\cosh[(1-2\alpha)t]\d t\\
&\quad -x\int_0^\infty \{[K_{0}[2x\cosh (t)]-K_2[2x\cosh(t)]\}\{\cosh[(-2\alpha)t]-\cosh(2-2\alpha)t]\}\d t\\
&=2\int_0^\infty \{K_{1}[2x \cosh(t)]\cosh[(1-2\alpha)t]- K_{1}[2x\cosh(t)]\sinh[(1-2\alpha)t]\tanh(t)\}\d t,
\end{align*}
where the last equality is obtained by algebra similar to the algebra for getting \eqref{coeff:ineq1}. Note that the integrand of the last integral for any $t>0$ is strictly positive since $K_1(x')>0$ for $x'>0$ and the angle subtraction formula of the hyperbolic cosine gives
\begin{align*}
\cosh[(1-2\alpha)t]-\sinh[(1-2\alpha)t]\tanh(t)
=\frac{\cosh(2\alpha t)}{\cosh (t)}.
\end{align*}
We have proved that 
$\frac{\d}{\d x}(x\frac{K_{1-\alpha}(x)}{K_\alpha(x)})>0$,
and hence, the required strict monotonicity of $x\mapsto R_\alpha(x)$ for all $\alpha\in [0,\infty)$. Finally, $R_\alpha(x)\to\infty$ as $x\to\infty$ by \eqref{asymp:Kinfty}.
\end{proof}

\begin{cor}\label{cor:BESmon}
With the convention that $0K_{1-\alpha}(0)/K_\alpha(0)=0$,
\begin{align}\label{Kx:1}
1-\alpha-x\frac{K_{1-\alpha}(x)}{K_\alpha(x)}>   -x^2+\frac{1}{4},\quad \forall\; \alpha\in [0,1/2),\; x\in[ 0,\infty).
\end{align}
\end{cor}
\begin{proof}
By Proposition~\ref{prop:BESmon} (3$\cc$), $R_\alpha(x)<R_{1/2}(x)=1/2+x$ for all $\alpha\in [0,1/2)$ and $x\in [0,\infty)$.  We also have $1/2-x\geq -x^2+1/4$ for all $x\in [0,\infty)$.
\end{proof}

The next proposition bounds the modulus of continuity of $\alpha\mapsto R_\alpha(x)$ at $\alpha=0$.

\begin{prop}\label{prop:Fxunif}
Recall the function $R_\alpha$ defined in \eqref{def:Ralpha}.
For all $\vep_0\in (0,1/2)$,
\begin{align}\label{Fx:unif}
\sup_{x\in [0,\infty)}\lvert R_{\vep}(x)-R_0(x)\rvert\leq C(\vep_0) \vep,\quad \forall\; \vep\in (0,\vep_0].
\end{align}
\end{prop}
\begin{proof}
We consider \eqref{diff:Ralpha} for $\alpha=0$: 
\begin{align}\label{Fx:int}
R_\vep(x)-R_0(x)=2\vep\int_0^\infty \frac{K_{\vep}[2x\cosh (t)]}{K_\vep(x)K_0(x)}
\sinh(\vep t)\tanh(t)\d t.
\end{align}
The goal of this proof is to show that the integral on the right-hand side is uniformly bounded over all $x\in (0,\infty)$ and $\vep\in (0,\vep_0]$. This extends to $x=0$ since $R_\alpha(0)$ is defined by continuity. 
Below, we show some technical bounds (Steps~1 and~2) before handling the integral in \eqref{Fx:int} (Step~3).\smallskip 

\noindent {\sc Step~1.}
We show the following three bounds:
for all fixed $\vep\in [0,1/2]$ and $M\in(0,\infty)$, 
\begin{align}
&K_\vep(x)\less \frac{1}{\sqrt{x}}\e^{-x},\quad\forall\;
0<x<\infty, \label{Kratio:bdd2-1}\\
&K_\vep(x)\geq C(M) \frac{1}{\sqrt{x}}\e^{-x},\quad\forall\;
M\leq x<\infty, \label{Kratio:bdd2}\\
&\frac{K_\vep(z)}{K_\vep(x)}\leq \frac{x^\vep}{z^\vep},\quad \forall\; 0<x\leq z<\infty.\label{Kratio:bdd1}
\end{align}
The bounds \eqref{Kratio:bdd2-1} and \eqref{Kratio:bdd2}
are obtained from \eqref{K:3}, and \eqref{Kratio:bdd1} is from \eqref{Knu:int}.
In more detail, \eqref{Kratio:bdd2-1} holds since $[1+t/(2x)]^{\vep-1/2}\leq 1$ for all $x,t\in (0,\infty)$ and $\vep\in [0,1/2]$; \eqref{Kratio:bdd2} holds since $[1+t/(2x)]^{\vep-1/2}\geq C(M)(1\vee t)^{\vep-1/2}$ for all $x\in [M,\infty)$, $t\in (0,\infty)$ and $\vep\in [0,1/2]$.
{\bf In the remaining of this proof, we fix a universal constant $M>2$ when using (\ref{Kratio:bdd2}).} \smallskip

\noindent {\sc Step~2.}
We claim that
\begin{align}\label{def:Y}
Y(x)\,\defeq \max\{t\in \R_+;2x\cosh(t)\leq M\}\vee 0,
\end{align}
with the convention that $\max\varnothing=-\infty$, satisfies 
\begin{align}
 Y(x)&\leq (\log M-\log x)\vee 0,\quad \forall\;x\in (0,\infty),\label{Y(x):bdd-1}\\
Y(x)&\geq (-\log x)\vee 0,\quad \forall\;x\in (0,M].\label{Y(x):bdd-2}
\end{align}
The use of $Y(x)$ is that we will bound the integral on the right-hand side of \eqref{Fx:int} according to $t\in [0,Y(x)]$ and $t\in [Y(x),\infty)$ separately.

To see \eqref{Y(x):bdd-1} and \eqref{Y(x):bdd-2}, it suffices to consider $x\in (0,M]$. For $M/2\geq x>0$, recall ${\rm arccosh}(z)=\log (z+\sqrt{z^2-1})$ for all $z\geq 1$, and so,
\begin{align}\label{M:asymp}
Y(x)= \log  \big\{[M/(2x)]+\sqrt{[M/(2x)]^2-1}\,\big\}.
\end{align} 
Then by \eqref{M:asymp} and the choice of $M>2$, the inequalities in \eqref{Y(x):bdd-1} and \eqref{Y(x):bdd-2} follow. For $M\geq x>M/2$, the inequalities in \eqref{Y(x):bdd-1} and \eqref{Y(x):bdd-2} hold trivially since $Y(x)=0$ and $-\log x<0$ again by the choice of $M>2$. We have proved both \eqref{Y(x):bdd-1} and \eqref{Y(x):bdd-2}.\\

\noindent {\sc Step~3.} We bound the integral in \eqref{Fx:int} in this step.
First, restrict the domain of integration to $[0, Y(x)]$. The associated integral for $x\in (0,\infty)$ and $\vep\in (0,\vep_0]$ satisfies the first inequality below by \eqref{Kratio:bdd1}:
\begin{align}
&\quad \int_0^{Y(x)} \frac{K_{\vep}[2x\cosh (t)]}{K_\vep(x)K_0(x)}
\sinh(\vep t)\tanh(t)\d t
\less \int_0^{Y(x)}\frac{\sinh(\vep t)\tanh(t)}{[\cosh(t)]^{\vep}K_0(x)} \d t\less 1.\label{MK:asymp1}
\end{align}
To see the last inequality, note that by \eqref{def:Y}, the last integral is nonzero only if $x\in (0,M]$, in which case the integral can be $\less$-bounded by $Y(x)/K_0(x)$ thanks to the following bounds:
\begin{align}\label{bdd:hyper}
\cosh (t)\geq \e^t/2, \quad \sinh (\vep t)\leq \e^{\vep t}/2, \quad |\tanh(t)|\leq 1,\quad \forall\;\vep>0,\;0<t<\infty.
\end{align}
Also, for $x\in [1,M]$, $Y(x)/K_0(x)\less 1$ since 
$1/K_0(x)\leq 1/K_0(M)$ by 
 \eqref{K:1} and $Y(x)\less 1$ by \eqref{Y(x):bdd-1}; for $x\in (0,1]$, $Y(x)/K_0(x)\less 1$ by \eqref{asymp:K0} for $\nu=0$ and \eqref{Y(x):bdd-1}. 

Next, use the domain of integration $[Y(x),\infty)$.
For $x\in (0,M]$, the associated integral is
\begin{align}
&\quad \int_{Y(x)}^\infty  \frac{K_\vep[2x\cosh(t)]}{K_\vep(x)K_0(x)}\sinh(\vep t)\tanh(t)\d t\notag\\
&\less \int_{Y(x)}^\infty \frac{1}{K_\vep(x)K_0(x)} \frac{1}{\sqrt{2x\cosh t}}\e^{-2x\cosh (t)}\sinh(\vep t)\tanh(t)\d t\label{MK:asymp2--1}\\
&\leq \frac{x^\vep}{M^\vep K_\vep(M)K_0(x)}\cdot \frac{1}{x^{1/2}} \int_{Y(x)}^\infty  \frac{\sinh(\vep t)}{\sqrt{\cosh t}}\tanh(t) \d t\label{MK:asymp2-0}\\
&\leq \frac{x^\vep}{M^\vep K_\vep(M)K_0(x)}\cdot \frac{1}{ x^{1/2}} \int_{Y(x)}^\infty \e^{-(1/2-\vep)t}\d t\label{MK:asymp2-1}\\
&\leq \frac{ x^\vep}{M^\vep K_0(M)K_0(M)}
\cdot \frac{1}{x^{1/2}}\cdot \frac{1}{1/2-\vep_0}\e^{-(1/2-\vep)[(-\log x)\vee 0]},\quad\forall\; \vep\in (0,\vep_0].\label{MK:asymp2-000}
\end{align}
Here, \eqref{MK:asymp2--1} applies \eqref{Kratio:bdd2-1};  
\eqref{MK:asymp2-0} applies \eqref{Kratio:bdd1}; \eqref{MK:asymp2-1} applies \eqref{bdd:hyper}; \eqref{MK:asymp2-000} applies the property that by \eqref{K:1}, $\vep\mapsto K_\vep(x)$ is increasing and $x\mapsto K_0(x)$ is decreasing, in addition to the bound \eqref{Y(x):bdd-2}. 
The right-hand side of \eqref{MK:asymp2-000} simplifies to $C(\vep_0)/[M^\vep K_0(M)K_0(M)]$ if we further require $x\leq 1$. Hence, the right-hand side of \eqref{MK:asymp2-000} is uniformly bounded over $\vep\in (0,\vep_0]$ and $x\in (0,M]$.
For $x\in [M,\infty)$, we have
\begin{align}
&\quad \int_{Y(x)}^\infty  \frac{K_\vep[2x\cosh(t)]}{K_\vep(x)K_0(x)}\sinh(\vep t)\tanh(t)\d t\notag\\
&\less C(M) \e^{2x}\sqrt{x}\int_{Y(x)}^\infty  \sqrt{\frac{x}{2x\cosh t}}\e^{-2x\cosh (t)}\sinh(\vep t)\tanh(t)\d t\label{MK:asymp2}\\
&\leq C(M)\e^{2x}\sqrt{x}\int_{0}^\infty \e^{-2x\cosh(t)}\e^{-(1/2-\vep)t}\d t\label{MK:asymp3-0}\\
&\less C(M)\e^{2x}\sqrt{x}K_{1/2}(2x)\label{MK:asymp3}.
\end{align}
Here, \eqref{MK:asymp2} uses \eqref{Kratio:bdd2-1} and \eqref{Kratio:bdd2};  
\eqref{MK:asymp3-0} uses \eqref{bdd:hyper}; \eqref{MK:asymp3} uses the bound $\e^{-(1/2-\vep)t}\leq 2\cosh (t/2)$, for
$\vep\leq \vep_0<1/2$ and $0<t<\infty$, and \eqref{K:1} for $\nu=1/2$. Note that the right-hand side of \eqref{MK:asymp3} is independent of $\vep\in (0,\vep_0]$ and remains bounded as $x\to\infty$ by \eqref{asymp:Kinfty}. 

In summary, for $\vep\in (0,\vep_0]$, \eqref{MK:asymp1} holds for $x\in (0,\infty)$, and the right-hand sides of \eqref{MK:asymp2-000} and \eqref{MK:asymp3} for $x\in (0,M]$ and $x\in [M,\infty)$, respectively,
are bounded by $C(\vep_0)$. So the integral in \eqref{Fx:int} for $\vep\in (0,\vep_0]$ and $x\in (0,\infty)$ is bounded by $C(\vep_0)$. This suffices for \eqref{Fx:unif}. The proof of Proposition~\ref{prop:Fxunif} is complete.
\end{proof}

\section{The lower-dimensional approximations}\label{sec:DBG}
Our goal in this section is to give the second proof of Theorem~\ref{thm:1}. To this end, the main step is to prove Theorem~\ref{thm:3} stated below. 
It formalizes \eqref{conv:BESab1} and gives extensions to $z_0=0$ and general $f$. The proof is based on Theorem~\ref{thm:2}.
After the statement of Theorem~\ref{thm:3}, Theorem~\ref{thm:1} will be obtained in a more general form as Theorem~\ref{thm:4}. 

To state Theorem~\ref{thm:3}, first, recall that with respect to $\P^{(-\alpha)}_{z_0}$ for $\alpha\in (0,1/2)$, we set $Z_0=z_0$ (Notation~\ref{nota:Z}), the local time $\{L_t\}$ at the point $0$ is chosen to satisfy the normalization in \eqref{def:Lt}, and the $q$-resolvents are denoted by $U_{q}^{\al} f(z)$ in \eqref{def:resolvent}. Now we introduce the following notation which is understood pointwise in $z$:
\begin{align}\label{def:U0}
U_q^{(0)} f(z)\;\defeq\, \lim_{\alpha\searrow 0}U_q^{(-\alpha)} f(z)\quad\mbox{whenever the limit exists in $[-\infty,\infty]$.}
\end{align}
It will be shown that this limit in \eqref{def:U0} exists in $\R$ for (real-valued) $f\in \C_b(\Bbb C)$ and recovers the $q$-resolvent of two-dimensional Brownian motion as the notation suggests. See \eqref{rep:U0}.

\begin{thm}\label{thm:3}
Let $\beta\in (0,\infty)$ and $z_0\in \Bbb C\setminus\{0\}$. The following holds with respect to $\Lambda_\alpha=\Lambda_\alpha(\beta)$ and $\{L_t\}$ chosen in \eqref{choice} and \eqref{def:Lt}.
\begin{enumerate}[label={\rm ({\arabic*}$\cc$)}]
\item For all $q\in (\beta,\infty)$ and $f\in \B_+(\Bbb C)$ such that $U_{q}^{(0)}f(z_0)$ and $U_{q}^{(0)}f(0)$ exist in $\R_+$,
\begin{align}
\begin{split}\label{main:1-1-1}
& \lim_{\alpha\searrow 0}\int_0^\infty \e^{-q t}\E^{\al}_{z_0}[\e^{\Lambda_\alpha L_{t}}f(Z_{t})]\d t\\
&\hspace{1cm} =U_{q}^{(0)}f(z_0)+\left(\int_0^\infty \e^{-q t}P_{t}(z_0)\d t\right) \left(\frac{2\pi }{\log (q/\beta)}\right)U_{q}^{(0)} f(0).
\end{split}
\end{align}
Also, for all $q\in (\beta,\infty)$ and $f\in \B_+(\Bbb C)$ such that $U_{q}^{(0)}f(0)$ exists in $\R_+$,  
\begin{align}
\begin{split}\label{main:1-1-2}
&\lim_{\alpha\searrow 0}\int_0^\infty \e^{-q t}\Lambda_\alpha\E^{\al}_0[\e^{\Lambda_\alpha L_{t}}f(Z_{t})]\d t =\frac{2\pi }{ \log (q/\beta)}U^{(0)}_q f(0).
\end{split}
\end{align}
\mbox{}

\item For all $f\in \C_b(\Bbb C)$,
\begin{align}
\begin{split}
\label{main:1-2-1}
&\lim_{\alpha\searrow 0}\E^{\al}_{z_0}[\e^{\Lambda_\alpha L_{t}}f(Z_{t})]\\
&\hspace{1cm} =P_tf(z_0) +\int_0^t P_s(z_0)\int_0^{t-s}\left(2\pi \int_0^\infty\frac{\beta^u \tau^{u-1}}{\Gamma(u)} \d u\right)  P_{t-s-\tau}f(0)\d \tau\d s,
\end{split}\\
\begin{split}
&\lim_{\alpha\searrow 0}\Lambda_\alpha\E^{\al}_0[\e^{\Lambda_\alpha L_{t}}f(Z_{t})]=
\int_0^{t}\left(2\pi \int_0^\infty \frac{\beta^u \tau^{u-1}}{\Gamma(u)}\d u\right)  P_{t-\tau}f(0)\d \tau.\label{main:1-2-2}
\end{split}
\end{align}
\end{enumerate}
\end{thm}
\mbox{}

To use this theorem, note that the right-hand sides of \eqref{main:1-2-1} and \eqref{main:1-2-2} recover the right-hand sides of \eqref{def:Pbeta2} and \eqref{def:ringP} if we replace $t$ in  \eqref{def:Pbeta2} and \eqref{def:ringP} by $t/2$. Also, in \eqref{main:1-1-2} and \eqref{main:1-2-2}, $Z_0=0$ and $\Lambda_\alpha$  enters multiplicatively in the limiting scheme. This is different from \eqref{main:1-1-1} and \eqref{main:1-2-1}. We will discuss the method of proof of Theorem~\ref{thm:3} after proving Theorem~\ref{thm:4}.

\begin{thm}\label{thm:4}
Fix $\beta\in (0,\infty)$, $z_0\in \Bbb C$ and $t\in (0,\infty)$. Recall $\widehat{K}_\alpha$ defined in \eqref{def:hK}.

\begin{enumerate}[label={\rm ({\arabic*}$\cc$)}]
\item For $\Phi\in \{F(\rho_s;s\leq t);F\in \C_b(C([0,t],\R_+))\}\cup \{f(Z_t);f\in \C_b(\Bbb C)\}$, 
\begin{align}\label{id:goal}
\begin{split}
\E^{(0),\beta\da}_{z_0}\Bigg[\frac{\e^{\beta t}\Phi}{K_0(\sqrt{2\beta}\lvert Z_t \rvert)}\Bigg]
&=\lim_{\alpha\searrow 0+}\E^{\albe}_{z_0}\Bigg[\frac{\e^{\beta t} \Phi}{\widehat{K}_\alpha (\sqrt{2\beta}\lvert Z_t \rvert)}\Bigg]\\
&=
\lim_{\alpha\searrow 0}\frac{\E^{\al}_{z_0}[\e^{\Lambda_\alpha L_t}\Phi]}{\widehat{K}_\alpha(\sqrt{2\beta}\lvert z_0 \rvert)},
\end{split}
\end{align}
\mbox{}

\item For $f\in \C_b(\Bbb C)$, \eqref{id:2-1} is implied by 
\eqref{main:1-2-1} and \eqref{id:goal} with $\Phi=f(\two Z_t)$, and 
 \eqref{id:2-2} is implied by \eqref{main:1-2-2} and \eqref{id:goal} with $\Phi=f(\two Z_t)$.
\end{enumerate}
\end{thm}

\begin{lem}\label{lem:Kunif}
As $\alpha\searrow0$, $\widehat{K}_\alpha(\cdot)^{-1}\to K_0(\cdot)^{-1}$ uniformly on compacts in $\R_+$, with $K_0(0)^{-1}\,\defeq\, 0$.
\end{lem}
\begin{proof}
We first show that  for all $0<x_0<x_1<\infty$, $\{\widehat{K}_\alpha(\cdot)^{-1};\alpha\in (0,1/2]\}$ is a family of Lipschitz continuous functions on $[x_0,x_1]$ such that the Lipschitz constants are uniformly bounded.
To see this property, note that for all $\alpha\in (0,1/2]$ and $0<x_0\leq x< y\leq x_1$, the mean value theorem implies that the second equality below holds for some $x'\in (x,y)$: 
\begin{align}
\lvert \widehat{K}_\alpha(x)^{-1}-\widehat{K}_\alpha(y)^{-1}\rvert&= \widehat{K}_\alpha(x)^{-1}\widehat{K}_\alpha(y)^{-1}\lvert \widehat{K}_\alpha(x)-\widehat{K}_\alpha(y)\rvert\notag\\
&= \widehat{K}_\alpha(x)^{-1}\widehat{K}_\alpha(y)^{-1} \lvert \widehat{K}_\alpha'(x')\rvert\cdot \lvert x-y\rvert \leq C(x_0,x_1)\lvert x-y\rvert . \label{hKunif1}
\end{align}
Also, the last inequality follows by using \eqref{hK:der} and \eqref{K:1}.  By the last inequality, we have proved the required form of Lipschitz continuity of $\{\widehat{K}_\alpha(\cdot)^{-1};\alpha\in (0,1/2]\}$ on $[x_0,x_1]$.  

To get the required uniform convergence, note that $\widehat{K}_\alpha(x)^{-1}\to K_0(x)^{-1} $ pointwise on $\R_+$ by \eqref{K:1} for $x\in (0,\infty)$ and by \eqref{def:hK0} and $K_0(0)^{-1}=0$ for $x=0$. Given the uniform Lipschitz continuity proven above, the convergence is also uniform on compacts in $(0,\infty)$ by the Arzel\`a--Ascoli theorem. We need
extensions to compact intervals containing $0$. 

We show $\widehat{K}_\alpha(\cdot)^{-1}\to K_0(\cdot)^{-1}$ uniformly on $[0,x_1]$
for any $x_1\in (0,\infty)$. First, given $\vep>0$, $K_0(0)^{-1}=0$ and the continuity of $K_0(\cdot)^{-1}$
allow some $x_0\in (0,x_1)$ such that $K_0(x_0)^{-1}<\vep/3$. Then, the uniform convergence proven above allows some $\alpha_0\in (0,1/2]$ such that 
\begin{align}\label{hKunif2}
\lvert \widehat{K}_\alpha(x)^{-1}-K_0(x)^{-1}\rvert<\vep,\quad\forall\;\alpha\in (0,\alpha_0],\;x\in [x_0,x_1].
\end{align}
Also, since $\widehat{K}_\alpha(\cdot)^{-1}$ and $K_0(\cdot)^{-1}$ are increasing by
\eqref{hK:der} and \eqref{K:1},
\begin{align}
\forall\;\alpha\in (0,\alpha_0],\;x\in [0,x_0),&\quad\;\lvert\widehat{K}_\alpha(x)^{-1}-K_0(x)^{-1}\rvert \leq \widehat{K}_\alpha(x_0)^{-1}+K_0(x_0)^{-1}\notag\\
&\leq \lvert\widehat{K}_\alpha(x_0)^{-1}-K_0(x_0)^{-1}\rvert+2K_0(x_0)^{-1}
\leq \vep,\label{hKunif3}
\end{align}
where the last inequality uses the choice of $x_0$ and $\alpha_0$.
Combining \eqref{hKunif2} and \eqref{hKunif3} gives 
\[
\lvert\widehat{K}_\alpha(x)^{-1}-K_0(x)^{-1}\rvert\leq \vep,\quad \forall\;\alpha\in (0,\alpha_0],\;x\in [0,x_1],
\]
which is the required uniform convergence on $[0,x_1]$. The proof is complete.
\end{proof}

Recall that by \eqref{Z:skewproduct}, $\{\dot{\gamma}_t\}$ denotes the circular Brownian motion of the angular part of $\{Z_t\}$ when $Z_0\neq 0$. For the remaining proofs, we will involve the equilibrium distribution of $\{\dot{\gamma}_t\}$, which is uniform on $[-\pi,\pi)$. 

\begin{rmk}
The rate of convergence to equilibrium of $\{\dot{\gamma}_t\}$ is simple. Specifically, let $\theta\mapsto \E_{\gamma_0}^\gamma[\delta_\theta(\dot{\gamma}_t)]$ denote the probability density of $\dot{\gamma}_t$ with respect to the Lebesgue measure on $[-\pi,\pi)$, where $\delta_\theta$ is Dirac's delta function at $\theta$. As an implication of \cite[(1.2) on p.135]{KK}, 
\begin{align}
\E_{\gamma_0}^\gamma[\delta_\theta(\dot{\gamma}_t)]=\frac{1}{2\pi}\sum_{n=-\infty}^\infty \exp\left(-\frac{n^2t}{2}-\i n\gamma_0\right),\quad \forall\;\gamma_0,\theta\in [-\pi,\pi),\;t\in (0,\infty)
\end{align}
Therefore,
$\sup_{\gamma_0,\theta}\left\lvert\E_{\gamma_0}^\gamma[\delta_\theta(\dot{\gamma}_t)]-1/(2\pi)\right\rvert\leq (1/\pi)\sum_{n=1}^\infty \e^{-n^2t/2}$.
\qed 
\end{rmk}

\begin{proof}[Proof of Theorem~\ref{thm:4}]
(1$\cc$) \; For $\Phi=F(\rho_s;s\leq t)$, we only need to show the first equality in \eqref{id:goal}, since the second equality follows immediately from the definition of $\P^{\albe}$ in \eqref{def:BESab}. By Theorem~\ref{thm:2} (4$\cc$) and Lemma~\ref{lem:Kunif}, $\Phi\widehat{K}_\alpha(\sqrt{2\beta}\rho_t)^{-1}$
under $\P^{\albe}_{x_0}$ converges in distribution to $\Phi K_0(\sqrt{2\beta}\rho_t)^{-1}$ under $\P^{\zbe}_{x_0}$ as $\alpha\searrow 0$, for fixed $x_0\in \R_+$. Hence, it remains to verify the uniform integrability of $\Phi\widehat{K}_\alpha (\sqrt{2\beta}\rho_t)^{-1}$ under $\P^{\albe}_{x_0}$ for $\alpha\in (0,1/4]$, or just show that  for all $\eta\in (1,\infty)$, 
\begin{align}\label{ui:eta}
\sup\left\{\E^{\albe}_{x_0}[\widehat{K}_\alpha (\sqrt{2\beta}\rho_t)^{-\eta}];\alpha\in (0,1/4]\right\}<\infty.
\end{align}

To prove \eqref{ui:eta}, note that $x\mapsto \widehat{K}_\nu(x)$ is decreasing by \eqref{def:hK} and \eqref{hK:der}. Also, by \eqref{def:hK} and \eqref{K:1}, $\nu\mapsto \widehat{K}_\nu(x)$ is increasing in $ [0,\infty)$ for $x\geq 1$. Hence, \eqref{asymp:Kinfty} for $\nu=0$ is enough to get
\begin{align}\label{hK:bddfinal}
\widehat{K}_\alpha (\sqrt{2\beta}y)^{-\eta}\less\e^{C(\beta)\eta y},\quad \forall\;\alpha\in (0,1/4],\;y\in \R_+,\;\eta\in (1,\infty).
\end{align}
Next, for any $\alpha\in (0,1/4]$, $\{\rho_t^2\}$ under $\P^{(-\alpha),\beta\da}_{x_0}$ can be pathwise dominated by a version of $\BES Q(0)$ starting from $x_0^2$; this domination holds by the strong existence of $\BES Q(0)$, 
Theorem~\ref{thm:2} (2$\cc$) for $\alpha=0$, and \cite[2.18 Proposition on p.293]{KS:BM}.
Also, $\E_{z_0}^B[\e^{q\lvert B_t\rvert}]<\infty$ for a planar Brownian motion $\{B_t\}$ and all $q\in (0,\infty)$ and $z_0\in \Bbb C$. Hence, \eqref{hK:bddfinal} implies \eqref{ui:eta}.

For $\Phi=f(Z_t)$ and $z_0\neq 0$, proving the first equality in \eqref{id:goal} needs to incorporate additional randomness in the angular part of $\{Z_t\}$. Recall \eqref{Z:skewproduct}, and let $\dot{\gamma}_\infty$ be uniformly distributed over $[-\pi,\pi)$ and independent of $\{(\lvert Z_t \rvert,\dot{\gamma}_t);0\leq t<\infty\}$. Then for all $\alpha\in [0,1/2)$,
\begin{align}
&\eqspace \E_{z_0}^{\albe}\Bigg[\left.\E^{\albe}_0\Bigg[\frac{f(Z_{t-s})}{\widehat{K}_\alpha(\sqrt{2\beta}\lvert Z_{t-s}\rvert )}\Bigg]\right\rvert_{s=T_0(Z)};t\geq T_0(Z)\Bigg]\notag\\
&=\E_{z_0}^{\albe}\Bigg[\left.\E^{\albe}_0\Bigg[\frac{f(\lvert Z_{t-s}\rvert \exp\{\i\dot{\gamma}_\infty\})}{\widehat{K}_\alpha(\sqrt{2\beta}\lvert Z_{t-s}\rvert ) }\Bigg]\right\rvert_{s=T_0(Z)};t\geq T_0(Z)\Bigg],\label{radialization}
\end{align}
which is implied by Erickson's characterization of the resolvents of $\{Z_t\}$~
\cite[Eq. (2.3) on p.75]{Erickson}. 
The foregoing equality, \eqref{Z:skewproduct}, and Lemma~\ref{lem:explosion} (1$\cc$) imply that 
\begin{align}
\E^{\albe}_{z_0}\Bigg[\frac{f(Z_t)}{\widehat{K}_\alpha(\sqrt{2\beta}\lvert Z_t \rvert)
}\Bigg]&=\E_{z_0}^{\albe}\Bigg[\frac{f(\lvert Z_t \rvert\exp\{\i \gamma_{\int_0^t \d s/\lvert Z_s\rvert^2}\})}{\widehat{K}_\alpha(\sqrt{2\beta}\lvert Z_t \rvert)}\Bigg]\notag\\
&=\E^{\albe}_{z_0}\left[\frac{\Psi_f(\lvert Z_t \rvert,\textstyle{\int_0^t \d s/\lvert Z_s\rvert^2})}{\widehat{K}_\alpha(\sqrt{2\beta}\lvert Z_t \rvert)}\right].\label{Psif:0}
\end{align}
Here, to get the last equality, we integrate out $\{\dot{\gamma}_t;0\leq t\leq \infty\}$ and set
\begin{align}\label{def:Psif}
\Psi_f(r, t)\,\defeq\, \E[f(r\exp\{\i\dot{\gamma}_t\})]:\R_+\times [0,\infty]\to \R.
\end{align}
Note that $\Psi_f$
is bounded continuous since $f\in \C_b(\Bbb C)$ is uniformly continuous on compacts, and as recalled above, $\dot{\gamma}_t$ converges in distribution to $\dot{\gamma}_\infty$ as $t\to\infty$. Then, the limit in \eqref{id:goal} for $\Phi=f(Z_t)$ and $z_0\neq 0$ holds by applying \eqref{ui:eta} and Theorem~\ref{thm:2} (4$\cc$) to \eqref{Psif:0}. 

For $\Phi=f(Z_t)$ and $z_0=0$, the proof is simpler since  by Erickson's characterization again,
\[
\E^{\albe}_{0}\Bigg[\frac{f(Z_t)}{\widehat{K}_\alpha(\sqrt{2\beta}\lvert Z_t \rvert)
}\Bigg]=\E^{\albe}_{0}\Bigg[\frac{\Psi_f(|Z_t|,\infty)}{\widehat{K}_\alpha(\sqrt{2\beta}\lvert Z_t \rvert)
}\Bigg].
\]
The right-hand side allows one more application of \eqref{id:goal} for $\Phi=F(\rho_s;s\leq t)$ proven above. We have completed the proof of (1$\cc$).
\medskip 

\noindent (2$\cc$)\; For $z_0\neq0$, \eqref{id:2-1} follows upon applying \eqref{id:goal} and \eqref{main:1-2-1} in the same order:
\begin{align*}
&\quad\;\E_{z_0/\two}^{\zbe}\left[\frac{\e^{\beta t}K_0(\sqrt{\beta}|z_0|)}{K_0(\sqrt{\beta}|\two Z_t|)}f(\two Z_t)\right]
=\lim_{\alpha\searrow 0}\E^{\al}_{z_0/\two}[\e^{\Lambda_\alpha L_t}f(\two Z_t)]\\
&=P_t\{\tilde{z}\mapsto f(\two \tilde{z})\}(z_0/\two)\\
&\quad\;+\int_0^t P_s(z_0/\two)\int_0^{t-s}\left(2\pi \int_0^\infty\frac{\beta^u\tau^{u-1}}{\Gamma(u)} \d u\right)P_{t-s-\tau}\{\tilde{z}\mapsto f(\two \tilde{z})\}(0)\d \tau\d s\\
&=P^\beta_tf(z_0),
\end{align*}
where the last equality holds by \eqref{def:Pbeta2} and \eqref{two}. (We have used as well  in the first equality the fact that $\widehat{K}_\alpha\to K_0$ pointwise.)
Also, to see \eqref{id:2-2}, we apply \eqref{id:goal} to get 
\begin{align*}
&\quad\; \E^{(0),\beta\da}_{0}\Bigg[\frac{\e^{\beta t}\cdot 2\pi}{K_0(\sqrt{\beta}\lvert\two Z_t \rvert)}f(\two Z_t)\Bigg]
=\lim_{\alpha\searrow 0}\frac{2\pi \E^{\al}_0[\e^{\Lambda_\alpha L_t}f(\two Z_t)]}{\widehat{K}_\alpha(0)}\\
&=\lim_{\alpha\searrow 0} 2
 \beta^{-\alpha}\Lambda_\alpha\E_0^{\al}[\e^{\Lambda_\alpha L_t}f(\two Z_t)]\\
&=2
\int_0^{t}\left(2\pi \int_0^\infty \frac{\beta^u \tau^{u-1}}{\Gamma(u)}\d u\right)  P_{t-\tau}\{\tilde{z}\mapsto f(\two \tilde{z})\}(0)\d \tau=\mathring{P}^\beta_tf,
\end{align*}
where the second equality uses \eqref{choice}, the third equality uses \eqref{main:1-2-2}, and the last equality uses \eqref{def:ringP} and \eqref{two}. The proof is complete.
\end{proof}

The proof of Theorem~\ref{thm:3} is in part similar to the method of using \eqref{def:approxH} and \eqref{def:Lambdaphi}. Both methods are at the expectation level if one views the latter in terms of Brownian motion by \eqref{FK:approx}. Another similarity is that 
we will use as essential tools the following expansions in the lower-dimensional approximations: under $\P^{\al}$ for $\alpha>0$,
\begin{align}\label{Exp0:LT}
\e^{\Lambda_\alpha L_{t}}=1+\Lambda_\alpha\int_0^{t}\e^{\Lambda_\alpha (L_{t}-L_s)} \d L_s=1+\Lambda_\alpha\int_0^{t}\e^{\Lambda_\alpha L_s} \d L_s.
\end{align}
These are comparable to some expansions in \cite[Lemma~5.3]{C:DBG3+} for using \eqref{def:approxH} and \eqref{def:Lambdaphi}. See Proposition~\ref{prop:idg} and the proof, especially
\eqref{bexp:2} and \eqref{bexp:1}, for more on \eqref{Exp0:LT}. 

On the other hand, compared to the method of using \eqref{def:approxH} and \eqref{def:Lambdaphi}, the method in the proof of Theorem~\ref{thm:3} enjoys the very different technical advantage of various exact calculations. These are due to the use of the local times of the lower-dimensional Bessel processes. Therefore, the proof of Theorem~\ref{thm:3} shows that even at the approximate level of $\alpha>0$,
the Laplace transforms in \eqref{main:1-1-1} and \eqref{main:1-1-2}
admit exact formulas expressible in the $q$-resolvents $U^{\al}_q f$, as stated in Proposition~\ref{prop:LTexplicit}. Moreover, we will show that it is enough to derive the exact formulas, stated in \eqref{F0:exp1}, of the Laplace transforms of the following functions:
\begin{align}\label{def:Falpha}
F_\alpha(t)\,\defeq\, \Lambda_\alpha\E^{\al}_0[\e^{\Lambda_\alpha L_t}],\quad \alpha\in (0,1/2).
\end{align}

Here are the main points of the remaining proof of Theorem~\ref{thm:3}: For (1$\cc$), we will use convergences of $\{Z_t\}$ under $\P^{\al}$ to planar Brownian motion as $\alpha\searrow 0$ by distributions (Lemma~\ref{lem:PBM}) and by a special form  \eqref{LT:norm} of using
the constant $C^\star_\alpha$ defined in \eqref{choice}; the special term $\log(q/\beta)$ in \eqref{main:1-1-1} and \eqref{main:1-1-2} will arise  in \eqref{F0:exp1} from the limit of the Laplace transforms of $F_\alpha(t)$. For (2$\cc$), the proof is to apply the Arzel\`a--Ascoli theorem appropriately.
 
We now begin the proof of Theorem~\ref{thm:3}. The first step is to establish convergences of $\{Z_t\}$ under $\P^{\al}$ to planar Brownian motion as $\alpha\searrow 0$. 

\begin{lem}\label{lem:PBM}
Let $x_0\in \R_+$.
For all $\alpha\in [0,1/2]$, let $\{\rho_{\alpha}(t)\}$ denote a version of $\BES(-\alpha)$ such that it is subject to initial condition $x_0$ and
a fixed driving Brownian motion as in \eqref{def:BES}. Fix $\{\alpha_n\}_{n\in \Bbb N}\subset (0,\frac{1}{2})$ such that $\alpha_n\searrow 0$, and write $\alpha_\infty\defeq\,0$. Then the following holds: 
\begin{enumerate}[label={\rm ({\arabic*}$\cc$)}]
\item With probability one, 
\[
\left(\rho_{\alpha_n}(t),\int_0^t\frac{\d s}{\rho_{\alpha_n}(s)^2}\right)\xrightarrow[n\to\infty]{}\left(\rho_{0}(t),\int_0^t \frac{\d s}{\rho_0(s)^2}\right)\mbox{ with }\rho_{\alpha_n}(t)\nearrow \rho_0(t),\quad\forall\;t\geq 0.
\]

\item Suppose $x_0\neq 0$. Let $z_0\in \Bbb C\setminus\{0\}$ with $|z_0|=x_0$. Let
$\{\dot{\gamma}_t\}$ be a circular Brownian motion (i.e. a one-dimensional standard Brownian motion mod $2\pi$) such that $z_0=x_0\e^{\i \dot{\gamma}_0}$ and $\{\dot{\gamma}_t\}\ind \{\rho_{\alpha_n}(t)\}_{n\in\Bbb N\cup\{\infty\}}$ (as processes). For all $n\in \Bbb N\cup\{\infty\}$, define 
\begin{align}
Z_{\alpha_n}(t)\,\defeq\,\rho_{\alpha_n}(t)\exp\{\i\dot{\gamma}_{\int_0^t \d s/\rho_{\alpha_n}(s)^2}\},\quad t<T_0(\rho_{\alpha_n}).
\end{align}
Then $\{Z_0(t)\}$ is a two-dimensional standard Brownian motion, and
with probability one, $Z_{\alpha_n}(t)$ converges to $Z_{0}(t)$ as $n\to\infty$ for all $t\geq 0$. 
\end{enumerate}
Thus, for all $f\in \C_b(\Bbb C)$ and $z_0\in \Bbb C$, the processes $\{Z_t\}$ under $\P_{z_0}^{\al}$ (Notation~\ref{nota:Z}) satisfy 
\begin{align}\label{U0:PBM}
\lim_{\alpha\searrow 0}\E^{\al}_{z_0}[f(Z_t)]= P_tf(z_0),
\end{align}
where $\{P_t\}$ is the probability semigroup of two-dimensional standard Brownian motion.
In particular,  for all $q\in (0,\infty)$, $f\in \C_b(\Bbb C)$ and $z\in \Bbb C$, the limit defined by \eqref{def:U0} exists in $\R$ with
\begin{align}\label{rep:U0}
U^{(0)}_q f(z)=\int_0^\infty \e^{-q t}P_{t}f(z)\d t.
\end{align}
\end{lem}
\begin{proof}
(1$\cc$) \; With probability one, $\rho_{\alpha_n}(t)$ increases to $\rho_0(t)$ for all $t\geq 0$ by applying the comparison theorem of SDEs \cite[2.18 Proposition on p.293]{KS:BM} to the squared Bessel processes $\{\rho_{\alpha_n}(t)^2\}$ and $\{\rho_0(t)^2\}$. This convergence extends to the almost-sure convergence of $\int_0^t \d s/\rho_{\alpha_n}(s)^2$ if $x_0\neq0$, since $T_0(\rho_{\alpha_n})\nearrow T_0(\rho_0)=\infty$ and dominated convergence applies when $t<T_0(\rho_0)$. If $x_0=0$, then $\int_0^t\d s/\rho_{\alpha_n}(s)^2\geq \int_0^t \d s/\rho_{0}(s)^2=\infty$ by the law of the iterated logarithm of two-dimensional Brownian motion \cite[(1.21) Exercise on p.60]{RY}. In more detail, this law entails $\rho_0(s,\omega)^2\less s\log(\log (1/s))$ for all small $s$, and we have $\int_{0+}\d s/[s\log(\log 1/s)]=\int^\infty \d v/\log v=\infty$ ($\log 1/s=v$). The proof of (1$\cc$) is complete.\medskip 

\noindent {\rm (2$\cc$)}\; The skew-product representation of planar Brownian motion \cite[(2.11) Theorem on p.193]{RY} shows that $\{Z_0(t)\}$ is one version. The required convergence of $\{Z_{\alpha_n}(t)\}$ follows immediately from (1$\cc$). \medskip  

Finally, to get \eqref{U0:PBM}, use (2$\cc$) for $z_0\neq 0$ and the argument in \eqref{Psif:0} for $z_0=0$.  
\end{proof}

The next step is to generalize Kac's moment formula. See also \cite[(3.15) on p.215]{DRVY}. 

\begin{prop}\label{prop:idg}
Fix $\alpha\in (0,\frac{1}{2})$, $t\in (0,\infty)$ and $z_0\in \Bbb C$.
\begin{enumerate}[label={\rm ({\arabic*}$\cc$)}]
\item For all nonnegative Borel measurable functions $G_s(\mathrm w)$ and $F_s(\mathrm w)$, $(s,\mathrm w)\in \R_+\times C(\R_+,\Bbb C)$, it holds that  
\begin{align}\label{eq:Kac}
\begin{split}
&\quad\; \E^{\al}_{z_0}\left[\int_0^tG_s(Z_{s\wedge r};r\geq 0) F_s(Z_{(r\vee s)\wedge t};r\geq 0) \d L_s\right]\\
&=\E^{\al}_{z_0}\left[\int_0^{t}G_s(Z_{s\wedge  r};r\geq 0)\E_0^{\al}[ F_s(Z_{r\wedge (t-s)};r\geq 0)] \d L_ s\right].
\end{split}
\end{align}
\mbox{}

\item For all  $f\in \B_+(\Bbb C)$ and $g\in \B_+(\R_+)$, it holds that, with $G(\ell)\,\defeq\,\int_0^\ell g(s)\d s$, 
\begin{align}
\E^{\al}_{z_0}[G(L_t)f(Z_{t})]&=
\E^{\al}_{z_0}\left[\int_0^t\E^{\al}_0[g(L_{t-s})f(Z_{t-s})] \d L_s\right]\label{eq:FKeq1}\\
&=\E^{\al}_{z_0}\left[\int_0^tg(L_{s})\E^{\al}_0[f(Z_{t-s})] \d L_s\right].\label{eq:FKeq2}
\end{align}
\end{enumerate}
\end{prop}
\begin{proof}
(1$\cc$)\; Let $\{\tau_\ell\}$ denote the inverse local time associated with $\{L_t\}$. It follows from the change of variables formula for Stieltjes integrals \cite[(4.9) Proposition on p.8]{RY} that 
\begin{align*}
&\quad\E^{\al}_{z_0}\left[\int_0^tG_s(Z_{s\wedge r};r\geq 0) F_s(Z_{(r\vee s)\wedge t};r\geq 0) \d L_s\right]\\
&=\int_0^{\infty}\E^{\al}_{z_0}\left[\1_{(\tau_\ell\leq t)}G_{\tau_\ell}(Z_{\tau_\ell\wedge r};r\geq 0) F_{\tau_\ell}(Z_{(r\vee \tau_\ell)\wedge t};r\geq 0)\right]\d \ell\\
&=\int_0^{\infty}\E^{\al}_{z_0}\left[\1_{(\tau_\ell\leq t)}G_{\tau_\ell}(Z_{\tau_\ell\wedge r};r\geq 0)\E_0^{\al}[ F_u(Z_{r\wedge (t-u)};r\geq 0)]\big|_{u=\tau_\ell}\right] \d \ell\\
&=\E^{\al}_{z_0}\left[\int_0^{t}G_s(Z_{s\wedge  r};r\geq 0)\E_0^{\al}[ F_s(Z_{r\wedge (t-s)};r\geq 0)] \d L_s\right].
\end{align*}
Here, the second equality follows from the strong Markov property of $\{Z_t\}$ at time $\tau_\ell$ and the property that $Z_{\tau_\ell}=\rho_{\tau_\ell}=0$ on $\{\tau_\ell<\infty\}$. The last equality proves \eqref{eq:Kac}. \medskip

\noindent (2$\cc$)\; By the chain rule of Stieltjes integrals \cite[(4.6) Proposition on p.6]{RY}, it holds that 
\begin{align}
G(L_t)&=G(L_t-L_s)\rvert^0_{s=t}=\int_0^t g(L_t-L_s)\d L_s,\label{bexp:2}\\
G(L_t)&=G(L_s)\rvert_{s=0}^t=\int_0^t g(L_s)\d L_s.\label{bexp:1}
\end{align}
Hence, by holding $t$ fixed, we obtain \eqref{eq:FKeq1} and \eqref{eq:FKeq2} upon applying \eqref{eq:Kac} with the following two choices:
\begin{itemize}
\item $G_s(Z_{s\wedge r};r\geq 0)=1$, $F_s(Z_{(s\vee r)\wedge t};r\geq 0)=g(L_t-L_s)f(Z_t)$, and
\item $G_s(Z_{s\wedge r};r\geq 0)=g(L_s)$, $F_s(Z_{(s\vee r)\wedge t};r\geq 0)=f(Z_{t})$,
\end{itemize}
respectively.
The proof is complete.
\end{proof}

\begin{proof}[Proof of Theorem~\ref{thm:3}]
Throughout this proof, $\alpha$ is understood to be in $(0,1/2)$, and
we suppress superscripts `$\al$' whenever there is no risk of confusion. \medskip

\noindent (1$\cc$)\;
We first prove the following expansion: for all $z_0\in \Bbb C$,
\begin{align}
&\quad \int_0^\infty \e^{-q t}\E_{z_0}[\e^{\Lambda_\alpha L_{t}}f(Z_{t})]\d t\notag\\
\begin{split}\label{Lap:asymp}
&=U_{q} f(z_0)+\left(\Lambda_\alpha\E_{z_0}\left[\int_0^\infty \e^{-q t}\d L_{t}\right]\right)U_q f(0)\\
&\quad +\left(\E_{z_0}\left[\int_0^\infty \e^{-q t}\d L_{t}\right]\right)\left(q\int_0^\infty \e^{-q t}\left(\Lambda_\alpha \E_0[\e^{\Lambda_\alpha L_{t}}]-\Lambda_\alpha\right)\d t\right)U_q f(0).
\end{split}
\end{align}
To see \eqref{Lap:asymp}, first, apply \eqref{eq:FKeq1} and then \eqref{eq:FKeq2}, both for $g(\ell)=\e^{\Lambda_\alpha\ell}\Lambda_\alpha$:
\begin{align}
\E_{z_0}[\e^{\Lambda_\alpha L_{t}}f(Z_{t})]&=\E_{z_0}[f(Z_{t})]+\E_{z_0}\left[\int_0^{t}\Lambda_\alpha \E_0[\e^{\Lambda_\alpha L_{t-s}}f(Z_{t-s})] \d L_s\right]\notag
\\
\begin{split}\label{Exp1:LT}
&=\E_{z_0}[f(Z_{t})]+\E_{z_0}\left[\int_0^{t}\Lambda_\alpha \E_0[f(Z_{t-s})] \d L_s\right]\\
&\quad +\E_{z_0}\left[\int_0^{t}  \E_0\left[\int_0^{t-s}  \Lambda_\alpha^2\e^{\Lambda_\alpha L_r}\E_0[f(Z_{t-s-r})] \d L_r\right]\d L_s\right].
\end{split}
\end{align}
Taking the Laplace transform of both sides of \eqref{Exp1:LT} yields, for all $q\in (0,\infty)$,
\begin{align}
&\quad \int_0^\infty \e^{-q t}\E_{z_0}[\e^{\Lambda_\alpha L_{t}}f(Z_{t})]\d t\notag\\
\begin{split}
&=U_{q} f(z_0)+\left(\Lambda_\alpha\E_{z_0}\left[\int_0^\infty \e^{-q t}\d L_{t}\right]\right)U_{q} f(0)\\
&\quad +\left(\E_{z_0}\left[\int_0^\infty \e^{-q t}\d L_{t}\right]\right)\left(q\int_0^\infty \e^{-q t}\E_0\left[\int_0^{t}\Lambda_\alpha^2\e^{\Lambda_\alpha L_r}\d L_r\right]\d t\right)U_{q} f(0).\label{Exp2:LT}
\end{split}
\end{align}
The last summand in \eqref{Exp2:LT} arises since, by using $\e^{-q t}=\int_t^\infty q \e^{-q r}\d r$ and Fubini's theorem,
\[
\E_0\left[\int_0^{\infty}\e^{-q t}\Lambda_\alpha^2\e^{\Lambda_\alpha L_t}\d L_t\right]=
q\int_0^\infty \e^{-q t}\E_0\left[\int_0^{t}\Lambda_\alpha^2\e^{\Lambda_\alpha L_r}\d L_r\right]\d t.
\]
The required identity in \eqref{Lap:asymp} then follows by using \eqref{Exp2:LT} and an implication of the last equality in \eqref{Exp0:LT} that
\begin{align}\label{LT:forwardexp}
 \E_{z_0}[\e^{\Lambda_\alpha L_{t}}]=1+\E_{z_0}\left[\int_0^t \Lambda_\alpha \e^{\Lambda_\alpha L_{r}}\d L_r\right].
\end{align}

As a consequence of \eqref{Lap:asymp}, proving \eqref{main:1-1-1} and \eqref{main:1-1-2} amounts to finding the asymptotic representations of the following three terms as $\alpha\searrow 0$: 
\begin{align}\label{list:3}
\begin{split}
 \E_{z_0}^{\al}\left[\int_0^\infty \e^{-q t}\d L_{t}\right],\; \Lambda_\alpha \E^{\al}_0\left[\int_0^\infty \e^{-q t}\d L_{t}\right],\; \ms LF_\alpha(q)\,\defeq\,\int_0^\infty \e^{-q t}F_\alpha(t)\d t,
 \end{split}
\end{align}
for $z_0\neq 0$, where $F_\alpha$ is defined in \eqref{def:Falpha}. We stress that \emph{the first term in \eqref{list:3} will be considered only for $z_0\neq 0$, whereas the second term there comes with an additional multiplicative scaling factor $\Lambda_\alpha$.}  In more detail, to prove \eqref{main:1-1-2}, we multiply both sides of \eqref{Lap:asymp} for $z_0=0$ by $\Lambda_\alpha$ so that the second term in \eqref{list:3} comes into play. Also, for the last term in \eqref{list:3}, 
our aim is to show that the asymptotic representation coincides with \eqref{main:1-1-2} with $f\equiv 1$. 

Below we show in Steps~\hyperlink{thm:3-1-1}{1}--\hyperlink{thm:3-1-3}{3} the asymptotic representations of the three terms in \eqref{list:3}. Step~\hyperlink{thm:3-1-4}{4} explains how \eqref{main:1-1-1} and \eqref{main:1-1-2} can be obtained accordingly. \medskip

\noindent \hypertarget{thm:3-1-1}{{\sc Step 1.}}\;
For the first term in \eqref{list:3} with $z_0\neq 0$, we show that 
\begin{align}\label{Pt:LT}
\lim_{\alpha\searrow 0+}\E^{\al}_{z_0}\left[\int_0^\infty \e^{-q t} \d L_t\right]=\int_0^\infty \e^{-q t}P_{t}(z_0)\d t.
\end{align}

With $C'_\alpha=C_\alpha^\star$ from \eqref{choice}, the case of \eqref{eq:LTLT} for $x_0=|z_0|>0$ gives
\begin{align}
\E^{\al}_{z_0}\left[\int_0^\infty \e^{-q t} \d L_t\right]&=\frac{1}{\pi \cdot \frac{2^{1-\alpha}}{\Gamma(\alpha)}\cdot q^\alpha} \frac{2^{1-\alpha}}{\Gamma(\alpha)}\widehat{K}_{\alpha}(\sqrt{2q}\lvert z_0 \rvert)
=\frac{1}{\pi q^\alpha}\widehat{K}_{\alpha}(\sqrt{2q}\lvert z_0 \rvert).\label{LT:norm}
\end{align}
Hence,
by \eqref{def:hK}, \eqref{def:K} and dominated convergence, $\widehat{K}_\alpha(\sqrt{2q}\lvert z_0 \rvert)\to K_0(\sqrt{2q}\lvert z_0 \rvert)$.  
Comparing \eqref{LT:norm0} with the limit of the right-hand side of \eqref{LT:norm} as $\alpha\searrow 0$ proves \eqref{Pt:LT}. \medskip 

\noindent \hypertarget{thm:3-1-2}{{\sc Step 2.}}\; The asymptotic representation of the second term in \eqref{list:3} can be derived similarly by using the definition of $\Lambda_\alpha$ in \eqref{choice} and the other case of $x_0=0$ in \eqref{eq:LTLT}. We have
\begin{align}
\Lambda_\alpha\E^{\al}_0\left[\int_0^\infty \e^{-q t} \d L_t\right]&=\frac{C_\alpha^\star\beta^\alpha}{C_\alpha^\star q^\alpha} =\frac{\beta^\alpha}{q^\alpha}\lra 1\quad\mbox{ as }\alpha\searrow 0.\label{Pt:LT0}
\end{align}

\noindent \hypertarget{thm:3-1-3}{{\sc Step 3.}}\; The proof of \eqref{main:1-1-2} for $f\equiv 1$ uses another expansion. By \eqref{eq:FKeq1} with $G(\ell)=\Lambda_\alpha \e^{\Lambda_\alpha \ell}$ and $f\equiv 1$, we get
\begin{align}
F_\alpha(t)&=\Lambda_\alpha  +\Lambda_\alpha \E_0\left[\int_0^t F_\alpha(t-s)\d L_s\right],\label{F0:exp}
\end{align}
so that 
\begin{align}\label{eq:LFalpha}
\ms LF_\alpha(q)&=\frac{\Lambda_\alpha}{q}  +\Lambda_\alpha\E_0\left[\int_0^\infty \e^{-q t}\d L_t\right]\ms L F_\alpha(q)
=\frac{\Lambda_\alpha}{q} +\frac{\Lambda_\alpha}{C_\alpha^\star q^\alpha}\ms L F_\alpha(q),
\end{align}
where the second equality uses the second line in \eqref{eq:LTLT}. 

To solve for $\ms L F_\alpha(q)$ from the last equality, we need the a-priori condition that 
\begin{align}\label{Falpha:finite}
\ms L F_\alpha(q)<\infty,\quad\forall\;q\in (0,\infty). 
\end{align}
Consider \eqref{char1:LT}. Then by H\"older's inequality, it is enough to show that 
\[
\E_0[\e^{q \rho_t^{2\alpha}}]<\infty\quad  \&\quad \E_0[\e^{q \int_0^t \rho_s^{2\alpha-1}\d W_s}]<\infty,\quad\forall\;q\in (0,\infty).
\]
Here, since $\{\rho_t\}$ can be pathwise dominated by a version of $\BES(0)$ and $2\alpha<1$ by assumption, we have $\E_0[\e^{q \rho_t^{2\alpha}}]<\infty$. To bound the other expectation in the foregoing display, apply the Cauchy--Schwarz inequality to get
\[
\E_0[\e^{q \int_0^t \rho_s^{2\alpha-1}\d W_s}]\leq \E_0[\e^{2q \int_0^t \rho_s^{2\alpha-1}\d W_s-2q^2\int_0^t \rho_s^{4\alpha-2}\d s}]^{1/2}\E_0[\e^{2q^2\int_0^t \rho_s^{4\alpha-2}\d s}]^{1/2}.
\]
Here, the last expectation is finite by Proposition~\ref{prop:mombdd} (2$\cc$) with $\eta=4\alpha-2$ since $\alpha>0$,  and the first expectation on the right-hand side is finite because it is the expectation of the exponential local martingale associated with $2q \int_0^t \rho_s^{2\alpha-1}\d W_s$. We have proved \eqref{Falpha:finite}.

We now conclude this step by using \eqref{eq:LFalpha} and \eqref{choice} in the same order: for all $q\in (\beta,\infty)$, 
\begin{align}
\ms LF_\alpha(q)&=\frac{\Lambda_\alpha }{q[1-\Lambda_\alpha/(C_\alpha^\star q^\alpha)]}
=\frac{\pi \cdot \frac{2^{1-\alpha}}{\Gamma(\alpha)}\cdot \beta^\alpha }{q[1-(\beta/q)^\alpha]}\lra \frac{2\pi}{q \log (q/\beta)}\quad\mbox{as }\alpha\searrow 0,\label{F0:exp1}
\end{align}
where the limit follows since $\Gamma(\alpha)\alpha=\Gamma(\alpha+1)\to \Gamma(1)=1$ and $(\d /\d \alpha)(x^\alpha)\big\rvert_{\alpha=0}=\log x$. This proves \eqref{main:1-1-2} for $f\equiv 1$. \medskip 

\noindent \hypertarget{thm:3-1-4}{{\sc Step 4.}}\; In summary, \eqref{main:1-1-1} follows upon applying \eqref{Pt:LT} and \eqref{F0:exp1} to \eqref{Lap:asymp}. Note that here we use in particular the assumption that $U_q^{(0)}f(0)$, as a limit defined by \eqref{def:U0}, exists in $\R$. Therefore, for \eqref{Lap:asymp}, the second summand on the right-hand side vanishes in the limit, and the third summand converges to the last term in \eqref{main:1-1-1}. The proof of  \eqref{main:1-1-2} is similar if we multiply both sides of \eqref{Exp2:LT} by $\Lambda_\alpha$ and then apply \eqref{Pt:LT0} and \eqref{F0:exp1}. The assumption on $U_q^{(0)}f(0)$ is applied similarly in this case. The proof of Theorem~\ref{thm:3} (1$\cc$) is complete.\medskip

\noindent (2$\cc$) \; Fix $z_0\in \Bbb C\setminus\{0\}$. For any nonnegative $f\in \C_b(\Bbb C)$, we consider
\begin{align}\label{def:phipsi}
\phi^{f}_\alpha(t)\,\defeq\, \E^{\al}_{z_0}[\e^{\Lambda_\alpha L_{t}}f(Z_{t})],\quad 
\psi^f_\alpha(t)\,\defeq\, \Lambda_\alpha\E^{\al}_0[\e^{\Lambda_\alpha L_{t}}f(Z_{t})],
\end{align} 
along with the following properties to be verified:
\begin{itemize}
\item [(\ref{def:phipsi}-1)] Every sequence $\{\phi^{f}_{\alpha_n}\}$
 with $1/4\geq \alpha_n\searrow 0$ has a subsequence that converges pointwise to a function continuous on $\R_+$.  
\item [(\ref{def:phipsi}-2)] Every sequence $\{\psi^{f}_{\alpha_n}\}$  with $1/4\geq\alpha_n\searrow 0$ has a subsequence that converges pointwise to a function continuous on $\R_+$. 
\end{itemize}
Accordingly, we will prove (2$\cc$) by showing the following properties in four steps:
\begin{itemize}
\item  [-- Step~1:] For any $f$, $\{\phi^{f}_\alpha\}_{\alpha\in (0,1/4]}$ and $\{\psi^f_\alpha\}_{\alpha\in (0,1/4]}$ are uniformly of exponential order $q$ for any $q\in (\beta,\infty)$.
\item [-- Step~2:] For any $f$ satisfying (\ref{def:phipsi}-1), the required convergence \eqref{main:1-2-1} of Theorem~\ref{thm:3} (2$\cc$) holds. Also, \eqref{main:1-2-2} holds
 for any $f$ satisfying (\ref{def:phipsi}-2).
\item [-- Step~3:] For any $f$,  (\ref{def:phipsi}-1) holds.
\item [-- Step~4:] For any $f$,  (\ref{def:phipsi}-2) holds.
\end{itemize}

\noindent \hypertarget{thm:3-2-1}{{\sc Step 1.}}\; We first show that $\{\phi^{f}_\alpha\}_{\alpha\in (0,1/4]}$ satisfies the following bound:
\begin{align}\label{phi:expbdd}
\sup_{\alpha\in (0,1/4]}\sup_{T\in [0,\infty)}\e^{-q T}\phi^{f}_\alpha(T)<\infty,\quad\forall\;q\in (\beta,\infty).
\end{align}
This property holds since
for all $\alpha\in (0,1/4]$, $T\in[0,\infty)$ and  $q\in (\beta,\infty)$, 
\begin{align}\label{Falpha:bdd}
\e^{-q T}\E^{\al}_{z_0}[\e^{\Lambda_\alpha\Lambda_T}]=\int_T^\infty q\e^{-q t}\d t\E^{\al}_{z_0}[\e^{\Lambda_\alpha\Lambda_T}]\leq \int_0^\infty q\e^{-q t}\E^{\al}_{z_0}[\e^{\Lambda_\alpha\Lambda_t}]\d t,
\end{align}
and the limit superior of the rightmost side as $\alpha\searrow 0$ is finite by \eqref{main:1-1-1} for $f\equiv\1$.

For \eqref{phi:expbdd} with $\phi_\alpha^{f}$ replaced by $\psi_\alpha^f$, the proof follows similarly if we multiply both sides of \eqref{Falpha:bdd} by $\Lambda_\alpha$, replace $z_0$ with $0$, and use \eqref{main:1-1-2}.\medskip 

\noindent  \hypertarget{thm:3-2-4}{{\sc Step 2.}}\;
We first explain why (\ref{def:phipsi}-1) implies \eqref{main:1-2-1}. It suffices to show that whenever $\phi_0^{f}$ is a continuous function given by the pointwise limit of some sequence $\{\phi^{f}_{\alpha_n}\}$ with $\alpha_n\searrow 0$,
\begin{align}\label{phiPhi:id}
\phi^{f}_0(t)=\Phi^{f}(t), \quad \forall\;t\in \R_+,
\end{align}
where $\Phi^{f}(t)$ denotes the function on the right-hand side of \eqref{main:1-2-1}. Note that \eqref{phiPhi:id} is enough to obtain \eqref{main:1-2-1} since the Laplace transform of the right-hand side of \eqref{main:1-2-1} (with respect to $\int_0^\infty \e^{-q t}(\cdot)\d t$) equals the right-hand side of \eqref{main:1-1-1} by  \eqref{Inv:sbeta}. Moreover, thanks to the continuity of $\phi^{f}_0$ and $\Phi^{f}$ on $\R_+$,
the proof of \eqref{phiPhi:id} amounts to showing that for fixed $q_0\in (\beta,\infty)$,
\begin{align}\label{Falpha:Lap1}
\int_0^T\e^{-q_0t}\phi^{f}_0(t)\d t=\int_0^T\e^{-q_0t}\Phi^{f}(t)\d t,\quad \forall\; T\in (0,\infty).
\end{align}
Note that $\int_0^T\e^{-q_0t}\phi^{f}_0(t)\d t<\infty$ by \eqref{phi:expbdd}.
To see \eqref{Falpha:Lap1}, let $q\in (0,\infty)$, and consider
 \begin{align}
&\quad \int_0^\infty \e^{-q T}\int_0^T\e^{-q_0 t}\phi^{f}_0(t)\d t\d T\notag\\
&=q^{-1}\int_0^\infty\e^{-(q+q_0) t}\phi^{f}_0(t)\d t
=\lim_{n\to\infty}q^{-1}\int_0^\infty\e^{-(q+q_0)t}\phi^{f}_{\alpha_n}(t)\d t\notag\\
&=q^{-1}\int_0^\infty \e^{-(q+q_0)t}\Phi^{f}(t)\d t
=\int_0^\infty \e^{-q t}\int_0^T \e^{-q _0t }\Phi^{f}(t)\d t \d T,\notag
\end{align}
where the second equality holds by \eqref{phi:expbdd} and dominated convergence, and the third equality uses \eqref{main:1-1-1}. By normalization, a standard result of weak convergences of probability measures implies \eqref{Falpha:Lap1}. See \cite[Lemma~4.18]{C:DBG3+} for more details. We have proved \eqref{main:1-2-1}.

The required implication that (\ref{def:phipsi}-2) gives \eqref{main:1-2-2} can be deduced by almost the same argument. Consequently, if $\Psi^f(t)$ denotes the function on the right-hand side of \eqref{main:1-2-2}, then
\begin{align}\label{def:BPsi}
\psi^f_0(t)=\Psi^f(t),\quad\forall\; t\in \R_+,
\end{align}
whenever $\psi^f_0$ is a continuous function given by the pointwise limit of some sequence $\{\psi^f_{\alpha_n}\}$ with $\alpha_n\searrow 0$. This proves \eqref{main:1-2-2} since the Laplace transform of $\Psi^f$ equals the right-hand side of \eqref{main:1-1-2} by \eqref{Inv:sbeta}. 
 \medskip

\noindent  \hypertarget{thm:3-2-2}{{\sc Step~3.}} \; To verify (\ref{def:phipsi}-1) for any nonnegative $f\in \C_b(\Bbb C)$, apply \eqref{eq:FKeq1} and \eqref{eq:LTmain1} to get the next two equalities, where $C_\alpha$ on the right-hand side of the second obeys \eqref{def:Cpalpha} for $C'_\alpha=C_\alpha^\star$:
\begin{align}
 \E_{z_0}[\e^{\Lambda_\alpha L_{t}}f(Z_{t})]
 &= \E_{z_0}[f(Z_{t})]+\E_{z_0}\left[\Lambda_\alpha\int_0^t  \E_0[\e^{L_{t-s}}f(Z_{t-s})]\d L_s\right]\notag\\
 &=\E_{z_0}[f(Z_{t})]+\frac{C_\alpha 2^\alpha}{\Gamma(1-\alpha)}\int_0^t  \Lambda_\alpha\E_0[\e^{\Lambda_\alpha L_{t-s}}f(Z_{t-s})]s^{\alpha-1}\exp\left(-\frac{\lvert z_0 \rvert^2}{2s}\right)\d s\notag\\
  &=\E_{z_0}[f(Z_{t})]+\frac{1}{\pi\cdot 2^{1-\alpha}}\int_0^t  \psi^{f}_\alpha(t-s)s^{\alpha-1}\exp\left(-\frac{\lvert z_0 \rvert^2}{2s}\right)\d s,\label{der:functional}
\end{align}
where the last equality follows from the choice of $C_\alpha^\star$ in \eqref{choice} and uses the definition of $\psi_\alpha^{f}$. 

To obtain (\ref{def:phipsi}-1) from \eqref{der:functional}, note that
by \eqref{U0:PBM} of Lemma~\ref{lem:PBM}, $t\mapsto \E_{z_0}^{\al}[f(Z_t)]$ converges pointwise to the continuous function $t\mapsto P_tf(z_0)$ as $\alpha\searrow 0$. Hence, it remains to show that the integral terms in \eqref{der:functional}, as functions of $t$ and indexed by $\alpha\in (0,1/4]$, are relatively sequentially compact in $C(\R_+,\R_+)$ with respect to the usual topology, that is, the topology of uniform convergences on compacts. Moreover, by the Arzel\`a--Ascoli theorem and Step~1, it suffices to show that these terms are equicontinuous on compacts. To this end, note that the Leibniz integral rule gives  
\begin{align*}
&\eqspace\frac{\d}{\d t}\int_0^t \psi_\alpha^{f}(s)(t-s)^{\alpha-1}\exp\left(-\frac{\lvert z_0 \rvert^2}{2(t-s)}\right)\d s\\
&=\psi_\alpha^{f}(t)\lim_{s\nearrow t}(t-s)^{\alpha-1}\exp\left(-\frac{\lvert z_0 \rvert^2}{2(t-s)}\right) +\int_0^t \psi_\alpha^{f}(s)\frac{\d}{\d t}(t-s)^{\alpha-1}\exp\left(-\frac{\lvert z_0 \rvert^2}{2(t-s)}\right)\d s\\
&=\int_0^t \psi_\alpha^{f}(s)\frac{\d}{\d t}(t-s)^{\alpha-1}\exp\left(-\frac{\lvert z_0 \rvert^2}{2(t-s)}\right)\d s,
\end{align*}
where the last equality uses the assumption $z_0\neq 0$. The foregoing integral terms for $\alpha$ ranging over $(0,1/4]$ are uniformly bounded on compacts in $\R_+$ since $\sup_{\alpha\in (0,1/4]}\psi_\alpha^{f}(T)<\infty$ for all $0<T<\infty$ by Step~1 and we have
\begin{align*}
&\eqspace \frac{\d}{\d t}(t-s)^{\alpha-1}\exp\left(-\frac{\lvert z_0 \rvert^2}{2(t-s)}\right)\\
&=(\alpha-1)(t-s)^{\alpha-2}\exp\left(-\frac{\lvert z_0 \rvert^2}{2(t-s)}\right)+(t-s)^{\alpha-1}\exp\left(-\frac{\lvert z_0 \rvert^2}{2(t-s)}\right)\frac{\lvert z_0 \rvert^2}{2(t-s)^2}.
\end{align*}
Hence, by the mean value theorem, the integral terms in \eqref{der:functional} for $\alpha$ ranging over $(0,1/4]$ are relatively sequentially compact in $C(\R_+,\R_+)$, as required.  \medskip

\noindent  \hypertarget{thm:3-2-3}{{\sc Step 4.}}\;
To verify (\ref{def:phipsi}-2) for any nonnegative $f\in \C_b(\Bbb C)$, note that by \eqref{eq:FKeq2},
\begin{align}\label{psi:dec}
 \Lambda_\alpha\E_0[\e^{\Lambda_\alpha L_{t}}f(Z_{t})]
&=\Lambda_\alpha\E_0[f(Z_{t})] +\E_0\left[\Lambda_\alpha\int_0^{t}\Lambda_\alpha \e^{\Lambda_\alpha L_s}\E_0[f(Z_{t-s})]\d L_s\right].
\end{align}
Hence, similar to the situation of \eqref{der:functional},
we need to show that
\begin{align}\label{psi:afunction}
t\mapsto \E_0\left[\Lambda_\alpha\int_0^{t}\Lambda_\alpha \e^{\Lambda_\alpha L_s}\E_0[f(Z_{t-s})]\d L_s\right],\quad \alpha\in (0,1/4],
\end{align}
are relatively sequentially compact in $C(\R_+,\R_+)$. So, we consider, for $0\leq t_1\leq t_2\leq T$, 
\begin{align}
&\quad\;\E_0\left[\Lambda_\alpha\int_0^{t_2}\Lambda_\alpha \e^{\Lambda_\alpha L_s}\E_0[f(Z_{t_2-s})]\d L_s-\Lambda_\alpha\int_0^{t_1}\Lambda_\alpha \e^{\Lambda_\alpha L_s}\E_0[f(Z_{t_1-s})]\d L_s\right]\notag\\
\begin{split}\label{psiPsi:0}
&=\E_0\left[\Lambda_\alpha\int_{t_1}^{t_2}\Lambda_\alpha \e^{\Lambda_\alpha L_s}\E_0[f(Z_{t_2-s})]\d L_s\right]\\
&\quad\;+\E_0\left[\Lambda_\alpha\int_0^{t_1}\Lambda_\alpha \e^{\Lambda_\alpha L_s}\{\E_0[f(Z_{t_2-s})]-\E_0[f(Z_{t_1-s})]\}\d L_s\right].
\end{split}
\end{align}

For the first term on the right-hand side of \eqref{psiPsi:0}, we can show that 
\begin{align}\label{psiPsi:1}
\begin{split}
&\forall\;\vep>0\;\forall\;T>0\;\exists\;\delta>0\mbox{ such that } \forall\;0\leq t_1\leq t_2\leq T,\;t_2-t_1\leq \delta,\\
&\left|\E_0\left[\Lambda_\alpha\int_{t_1}^{t_2}\Lambda_\alpha \e^{\Lambda_\alpha L_s}\E_0[f(Z_{t_2-s})]\d L_s\right]\right|\leq \vep.
\end{split}
\end{align}
To see this property, first, note that
\begin{align}
\left|\E_0\left[\Lambda_\alpha\int_{t_1}^{t_2}\Lambda_\alpha \e^{\Lambda_\alpha L_s}\E_0[f(Z_{t_2-s})]\d L_s\right]\right|&\leq \|f\|_\infty\E_0\left[\Lambda_\alpha\int_{t_1}^{t_2}\Lambda_\alpha \e^{\Lambda_\alpha L_s}\d L_s\right]\notag\\
&=\|f\|_\infty \{\Lambda_\alpha\E_0[\e^{\Lambda_\alpha L_{t_2}}]-\Lambda_\alpha\E_0[\e^{\Lambda_\alpha L_{t_1}}]\}\notag\\
&=\|f\|_\infty [F_\alpha(t_2)-F_\alpha(t_1)],
\label{correct1}
\end{align}
where we use the definition \eqref{def:Falpha} of $F_\alpha $ in \eqref{correct1}. Recall the notation $\Psi^f$ in \eqref{def:BPsi}, and note 
that $\Psi^{\1}(t)=\int_0^{t}(2\pi \int_0^\infty \beta^u \tau^{u-1}\Gamma(u)^{-1}\d u)\d\tau $ by definition. 
Then to bound the right-hand side of \eqref{correct1}, note that by \eqref{main:1-1-2} for $f\equiv \1$, Step~1, and \cite[Lemma~4.18]{C:DBG3+}, $F_\alpha(t)$ converges to $ \Psi^{\1}(t) $ pointwise in $(0,\infty)$, and hence, in $\R_+$ since $\Psi^{\1}(t)=0$ and $F_\alpha(0)=\Lambda_\alpha$ converges to $ 0$.
Moreover, since $\Psi^\1$ is continuous and each $F_\alpha$ is increasing, the convergence of $F_\alpha$ to $\Psi^\1$ is uniform on compacts in $\R_+$. By the Arzel\`a--Ascoli theorem, the family $\{F_\alpha\}_{\alpha\in (0,1/4]}$ is equicontinuous on compacts. Applying this property to \eqref{correct1} proves \eqref{psiPsi:1}.

For the second term on the right-hand side of \eqref{psiPsi:0}, we can show the following property:
\begin{align}
\begin{split}\label{psiPsi:2}
&\forall\;\vep>0\;\forall\;T>0\;\exists\;\delta>0\mbox{ such that } \forall\;0\leq t_1\leq t_2\leq T,\;t_2-t_1\leq \delta,\\
&\left|\E_0\left[\Lambda_\alpha\int_0^{t_1}\Lambda_\alpha \e^{\Lambda_\alpha L_s}\{\E_0[f(Z_{t_2-s})]-\E_0[f(Z_{t_1-s})]\}\d L_s\right]\right|\leq \vep.
\end{split}
\end{align}
To see \eqref{psiPsi:2}, we consider
\begin{align}\label{correct:2}
&\quad\;\left|\E_0\left[\Lambda_\alpha\int_0^{t_1}\Lambda_\alpha \e^{\Lambda_\alpha L_s}\{\E_0[f(Z_{t_2-s})]-\E_0[f(Z_{t_1-s})]\}\d L_s\right]\right|\notag\\
&\leq \sup_{\stackrel{\scriptstyle 0\leq s_1<s_2\leq T}{ s_2-s_1\leq t_2-t_1}}|\E_0[f(Z_{s_2})]-\E_0[f(Z_{s_1})]|\cdot [F_\alpha(t_1)-F_\alpha(0)],\quad \forall\;0\leq t_1\leq t_2\leq T.
\end{align}
For the supremum in the last inequality, we will show at the end of this proof that
\begin{align}\label{psiPsi:3}
\lim_{\delta\searrow 0}\sup_{\alpha\in(0,1/4]}\sup_{\stackrel{\scriptstyle 0\leq s_1<s_2\leq T}{ s_2-s_1\leq \delta}}|\E_0^{\al}[f(Z_{s_2})]-\E_0^{\al}[f(Z_{s_1})]|=0.
\end{align}
Also, by (\ref{def:phipsi}-1) for the case of $\psi^{\1}_{\alpha}=F_\alpha$, $\sup_{\alpha\in (0,1/4]}F_\alpha(T)<\infty$ for all $0<T<\infty$. These are enough to get \eqref{psiPsi:2}. To sum up, we obtain  (\ref{def:phipsi}-2) upon applying \eqref{psiPsi:1} and \eqref{psiPsi:2} to \eqref{psiPsi:0} and recalling \eqref{psi:dec}. 

It remains to prove \eqref{psiPsi:3}. Note that as in the proof of \eqref{radialization},
$\E_0[f(Z_{s})]=\E_0[\overline{f}(\rho_{s})]=\E_0[\overline{f}(\sqrt{s}\rho_1)]$ by the Brownian scaling property of Bessel processes, where $\overline{f}$ is defined by \eqref{def:fbar}. Then we consider $s_1,s_2$ near zero and $s_1,s_2$ away from zero separately:
\begin{itemize}
\item For $s_1,s_2$ near zero, note that for all $M>0$,
\begin{align*}
&\quad\;|\E^{\al}_0[f(Z_s)]-f(0)|\\
&\leq \E^{\al}_0[|\overline{f}(\sqrt{s}\rho_1)-f(0)|;\rho_1\leq M]+2\|\overline{f}\|_\infty \P^{\al}_0(\rho_1> M)\\
&=\int_0^M |\overline{f}(\sqrt{s}y)-f(0)|\frac{2^\alpha}{\Gamma(1-\alpha)}y^{1-2\alpha}\exp\left(-\frac{y^2}{2}\right)\d y+2\|\overline{f}\|_\infty \P^{(0)}_0(\rho_1> M).
\end{align*}
In the last equality, the first term follows from \eqref{density} for $x=0$, and the second term follows from the comparison theorem of SDEs as in the proof of Lemma~\ref{lem:PBM} (1$\cc$). Therefore, 
\begin{align}\label{bullet1}
\begin{split}
&\quad\forall\;\vep>0\;\exists\;0<\eta<T\mbox{ such that }\\
&\sup_{\alpha\in (0,1/4]}\sup_{0\leq s_1,s_2\leq \eta}|\E_0^{\al}[f(Z_{s_2})]-\E_0^{\al}[f(Z_{s_1})]|\leq \vep.
\end{split}
\end{align}
\item For $s_1,s_2$ away from zero, note that 
$s\mapsto \sup\{(\d/\d s)\E_0^{\al}[\overline{f}(\rho_s)];\alpha\in (0,1/4]\}$ is bounded on compacts in $(0,\infty)$ by \eqref{density} for $x=0$. So, by the mean value theorem,
\begin{align}
\begin{split}\label{bullet2}
&\quad\forall\;\vep>0\;\forall\;0<\eta<T\;\exists\;\delta<\eta/2\quad\mbox{such that }\\
&\sup_{\alpha\in(0,1/4]}\sup_{\stackrel{\scriptstyle 0\leq s_1<s_2\leq T}{\stackrel{\scriptstyle s_2-s_1\leq \delta}{\eta\leq s_2 }}}|\E_0^{\al}[f(Z_{s_2})]-\E_0^{\al}[f(Z_{s_1})]|\leq \vep.
\end{split}
\end{align}
\end{itemize}
Combining \eqref{bullet1} and \eqref{bullet2} proves \eqref{psiPsi:3}.\medskip
 
The proof of Theorem~\ref{thm:3} (2$\cc$) is complete.
\end{proof}

\section{The lower-dimensional Bessel processes}\label{sec:BES}
This section considers Bessel processes mainly of dimensions lower than two (i.e. $\BES(-\alpha)$ for $-\alpha<0$). We
collect the proofs of some related results applied in the earlier sections.

\subsection{Some moment bounds}\label{sec:BESmom}
Recall that $\widehat{K}_\alpha(\cdot)$ is defined in \eqref{def:hK}. 

\begin{prop}\label{prop:mombdd}
\begin{enumerate}[label={\rm ({\arabic*}$\cc$)}]
\item Given any $\alpha_0\in (0,1/2)$, $M\in(0,\infty)$, $\beta\in (0,\infty)$ and $\eta\leq 0$, 
\begin{align}\label{Lap:bdd}
\sup_{\alpha\in (0,\alpha_0]}\sup_{x_0\in [0, M]}\E^{\al}_{x_0}[\widehat{K}_\alpha(\sqrt{2\beta}\rho_t)^{\eta}]\leq C(\alpha_0,M,\beta,\eta)\e^{C(\alpha_0,M,\beta,\eta)t},\;\forall\;t\in \R_+.
\end{align}

\item Given $\alpha\in (0,1/2)$, $t\in (0,\infty)$ and $f\in \B_+(\R_+)$,
\begin{align}\label{apbdd:1}
\begin{split}
&\sup_{0\leq s\leq t}\sup_{x_0\in\R_+}\E^{\al}_{x_0}[f(\rho_s)]<\infty\\
&\quad \Longrightarrow \sup_{x_0\in \R_+}\E^{\al}_{x_0}\left[\exp \left(q \int_0^t\rho_r^{\eta}\d r\right)f(\rho_t)\right]<\infty
\end{split}
\end{align}
for all $\eta\leq 0$ with $1-2\alpha+\eta>-1$ and all $q\in \R$.
\end{enumerate}
\end{prop}

The proof of this proposition requires the following bounds on negative moments. We include $\alpha=0$ here for future reference. 

\begin{lem}\label{lem:BESint}
Given any $\alpha\in [0,1)$ and $\eta\leq 0$ with $1-2\alpha+\eta>-1$, 
\begin{gather}
\sup_{x_0\in \R_+}\E^{\al}_{x_0}[ \rho_t^{\eta}]\leq t^{\eta/2}\E^{\al}_0[\rho_1^{\eta}]<\infty,\quad \forall\;t\in (0,\infty). \label{negativemom}
\end{gather}
\end{lem}
\begin{proof}
The first inequality in \eqref{negativemom} follows since
\begin{align}\label{ineq:BESint}
\quad \E_{x_0}^{\al}[\rho_t^{\eta}]= t^{\eta/2}\E^{\al}_{x_0/\sqrt{t}}[\rho_1^{\eta}]\leq t^{\eta/2}\E^{\al}_0[\rho_1^{\eta}].
\end{align}
In more detail, for \eqref{ineq:BESint}, the equality follows from the Brownian scaling property of $\BES(-\alpha)$ \cite[(1.6) Proposition on p.443]{RY}, and the inequality follows by using the comparison theorem of SDEs \cite[2.18 Proposition on p.293]{KS:BM} and the nonpositivity of $\eta$.

To see the second inequality in \eqref{negativemom}, apply the explicit formula of the transition probability density function of $\rho_1$ [cf. \eqref{density} for $x=0$ and $t=1$]. This gives the required inequality:
\begin{align}\label{alpha0:finite}
\E^{\al}_{0}[\rho_1^{\eta}]
&=\int_0^\infty \frac{2^{\alpha}}{\Gamma(1-\alpha)}\exp\left(-\frac{y^2}{2}\right) y^{1-2\alpha+\eta}\d y<\infty,
\end{align}
where the last inequality uses the assumption of $\eta$ that it satisfies $1-2\alpha+\eta>-1$.
\end{proof}

\begin{proof}[Proof of Proposition~\ref{prop:mombdd}]
($1\cc$) \; We first claim that for all $0<\alpha\leq \alpha_0<1/2$,
\begin{align}
\widehat{K}_\alpha(x)&\more 1,\quad 0\leq x\leq 2,\label{LB1:K}\\
\widehat{K}_\alpha(x) &\geq C(\alpha_0)x^{-1/2}\e^{-x},\quad 2\leq x<\infty.\label{LB2:K}
\end{align}
To get \eqref{LB1:K}, we use \eqref{Knu:int} with $\nu=\alpha$ so that for all $0<x\leq 2$ and $\alpha>0$,
\begin{align*}
\widehat{K}_\alpha(x)&\geq 2^{-1}\int_0^\infty\e^{-s}\int_{(1/s)\vee 1}^\infty \e^{-u}u^{-1}\d u\d s.
\end{align*} 
This is enough for \eqref{LB1:K} since the right-hand side is a universal constant. 
Next, for $x\geq 2$, \eqref{LB2:K} holds by using \eqref{K:3} for $\nu=\alpha$ and the integral representation of the gamma function:
\begin{align*}
\forall\;0<\alpha\leq \alpha_0<1/2,\quad
K_\alpha(x)&\more \frac{\e^{-x}}{\sqrt{x}}\cdot \frac{\int_0^\infty \e^{-t}t^{\alpha-1/2}(1+\frac{t}{4})^{\alpha-1/2}\d t}{\int_0^\infty \e^{-t}t^{\alpha-1/2}\d t}\\
&\more \frac{\e^{-x}}{\sqrt{x}}\cdot \frac{\int_0^1 \e^{-t}t^{\alpha-1/2}C(\alpha_0)\d t+\int_1^\infty \e^{-t}t^{-1}(1+\frac{t}{4})^{-1/2}\d t}{\int_0^1 \e^{-t}t^{\alpha-1/2}\d t+\int_1^\infty \e^{-t}t^{\alpha_0-1/2}\d t}\\
&\geq C(\alpha_0)x^{-1/2}\e^{-x}.
\end{align*}
The last inequality holds since $(s+a)/(s+b)\more C(a,b)$ for all $s\in\R_+$ and $a,b\in (0,\infty)$.

Now, to prove \eqref{Lap:bdd}, let $B$ denote a complex-valued planar Brownian motion with $B_0=0$ and $\E[|\Re B_1|^2]=\E[|\Im B_1|^2]=1$. Then for $\eta\leq 0$, the first inequality below holds by \eqref{LB1:K} and \eqref{LB2:K}, and the second inequality holds by the comparison theorem of SDEs:
\begin{align*}
\E^{\al}_{x_0}\big[\widehat{K}_\alpha(\sqrt{2\beta}\rho_t)^{\eta}\big]&\less 1+C(\alpha_0) \E^{\al}_{x_0}\big[(\sqrt{2\beta}\rho_t)^{-\eta/2}\e^{-\eta \sqrt{2\beta}\rho_t};\sqrt{2\beta}\rho_t\geq 2\big]\\
&\leq 1+C(\alpha_0,\beta,\eta)\E\big[\lvert x_0+B_t\rvert^{-\eta/2}\e^{-\eta\sqrt{2\beta}\lvert x_0+B_t\rvert}\big].
\end{align*}
Since $\E[\e^{q\lvert B_t\rvert}]\less \e^{q^2t}$ for all $q\in \R$, as mentioned in the proof of Lemma~\ref{lem:contexp}, the last inequality is enough to get \eqref{Lap:bdd}.\medskip 

\noindent {\rm (2$\cc$)} \; It is enough to work with the case of $q>0$. We consider the following approximations of the supremum in the necessary condition of \eqref{apbdd:1}: For all $\vep\in (0,1)$,
\begin{align*}
F_\vep(s)\,\defeq \,\sup_{x_0\in \R_+}\E^{\al}_{x_0}\left[\exp \left( \int_0^s g_\vep(\rho_r)\d r\right)f(\rho_s)\right],\quad\mbox{where}\quad
g_\vep(y)\,\defeq \,(q y^\eta)\wedge  \vep^{-1}.
\end{align*}

We first show that these functions $F_\vep$ satisfy the following inequalities:
\begin{align}\label{conv:ineq}
F_\vep(s)\leq \sup_{x_0\in\R_+}\E^{\al}_{x_0}[f(\rho_s)]+\int_0^s q r^{\eta/2}\E_0^{\al}[\rho_1^{\eta}]F_\vep(s-r)\d r,\quad 0\leq s\leq t.
\end{align}
To see \eqref{conv:ineq}, note that since $g_\vep(y)\leq q y^\eta$ and $f\geq 0$,
the elementary expansion $\e^{\int_0^s h(r)\d r}=1+\int_0^s h(r)\e^{\int_r^sh(v)\d v}\d r$ and the Markov property of the Bessel process give
\begin{align}
&\quad\; \E^{\al}_{x_0}\left[\exp \left(\int_0^s g_\vep(\rho_r)\d r\right)f(\rho_s)\right]\notag\\
&\leq \E^{\al}_{x_0}[f(\rho_s)]+\E^{\al}_{x_0}\left[\int_0^sq\rho_r^{\eta}\E^{\al}_{\rho_r}\left[\exp\left(\int_0^{s-r}g_\vep(\rho_v)\d v\right)f(\rho_{s-r})\right]\d r\right].\label{Fvep:int}
\end{align}
The last inequality leads to \eqref{conv:ineq} upon applying \eqref{negativemom}.

To complete the proof of \eqref{apbdd:1}, observe that \eqref{conv:ineq} is a convolution-type inequality where $r\mapsto q r^{\eta/2}\E_0^{\al}[\rho_1^{\eta}]$ is independent of $\vep$ and in $ L_1([0,t],\d r)$. (The $L_1$-property holds since the second inequality in \eqref{negativemom} holds and
 the assumption of $1-2\alpha+\eta>-1$ implies $\eta/2>-1+\alpha>-1$.) Also, the bounded continuity of $g_\vep$ and the sufficient condition in \eqref{apbdd:1} imply $\sup_{0\leq s\leq t}F_\vep(s)<\infty$. An extension of Gr\"onwall's lemma to convolution-type inequalities \cite[Lemma~15 on p.22--23]{Dalang} applies. We get
\begin{align}\label{Gron:bdd}
\sup_{0\leq s\leq t}F_\vep(s)\leq C(\alpha,\eta,q,f,t),\quad \forall\;\vep\in (0,1).
\end{align}
By passing $\vep\searrow 0$ and using Fatou's lemma together with $\E^{\al}_{x_0}[\int_0^\infty\1_{\{0\}}(\rho_r)\d r]=0$, \eqref{Gron:bdd} extends to the required necessary condition of \eqref{apbdd:1}. The proof is complete.
\end{proof}

\begin{rmk}
The usual methods closely related to the one used to obtain \eqref{Gron:bdd} are Khas'minskii's lemma \cite[Theorem~3]{Kha:Sol}  and Portenko's lemma~\cite[Lemma~1]{Port:Diff}. \qed 
\end{rmk}

\subsection{Explicit formulas for the local times}\label{sec:BESlt}
In this subsection, we turn to the local times of $\BES(-\alpha)$ for $\alpha\in (0,1)$.\smallskip

\noindent {1).} To characterize the local times, we first specify the speed measure $m_\alpha$ of $\BES(-\alpha)$, for $\alpha\in (0,1)$, with a normalization following the definition in \cite[Section~3 of Chapter~VII on pp.300+]{RY}. The speed measure accompanies the scale function $s_\alpha$, so we choose both of them as follows, using a fixed, but arbitrary, constant $C_\alpha\in (0,\infty)$:
\begin{align}
s_\alpha(y)&=\alpha^{-1}C_\alpha y^{2\alpha},\quad 
m_\alpha(\d y)=
C_\alpha^{-1} y^{1-2\alpha}\d y,\quad y\geq 0.\label{def:scale+speed}
\end{align}
See \cite[p.446]{RY} for the special case of $C_\alpha=\alpha$ by taking $\nu=-\alpha$ there in \cite{RY}.

With respect to $(s_\alpha,m_\alpha)$, the associated local times are given by a jointly continuous family $\{L^y_t;(y,t)\in \R^2_+\}$ such that 
\begin{itemize}[leftmargin=3\labelsep]
\item $\{L^y_t;t\geq 0\}$ is a Markovian local time at level $y$ \cite[Sections~IV.2--IV.4 on pp.105+]{Bertoin}, and
\item the occupation times formula holds:
\begin{align}\label{DLT:otf}
\int_0^t g(\rho_s)\d s=\int_0^\infty g(y)L^y_tm_\alpha(\d y),\quad\forall\; g\in \B_+(\R_+),\;t\in \R_+.
\end{align}
\end{itemize}
See It\^{o} and McKean~\cite[Section~5.4 on pp.174+]{IM:Diffusion} or Marcus and Rosen~\cite[Chapter~3 on pp.62+]{MR:Markov} for the existence of $\{L^y_t;(y,t)\in \R^2_+\}$. Note that for any $y\in \R_+$, $\{L^y_t;t\geq 0\}$ is indeed a particular case of those Markovian local times considered in \cite[Sections~IV.2--IV.4 on pp.105+]{Bertoin} since any $y\geq 0$ is regular and instantaneous under $\BES(-\alpha)$.\medskip 

\noindent {2).} The following proposition records a basic formula we have used earlier in applications of the local times.

\begin{prop}\label{prop:density}
Given any $\alpha\in(0,1)$, $C_\alpha\in(0,\infty)$ for \eqref{def:scale+speed},
 $x_0\geq 0$, $y\geq 0$, and $f\in \B_+(\R_+)$, we have
\begin{align}\label{eq:LTdensity}
\E^{\al}_{x_0}\left[\int_0^\infty f(t)\d L^y_t\right]=\int_0^\infty f(t)\bar p_t^{\al}(x_0,y)\d t.
\end{align}
Here, with $I_{-\alpha}$ denoting the modified Bessel function of the first kind of index $-\alpha$,
\begin{align}\label{def:barpt}
\begin{split}
\bar p_t^{\al} (x,y)\,&\defeq
\begin{cases}
\displaystyle \frac{C_\alpha x^{\alpha}y^{\alpha}}{t}\exp\left(-\frac{x^2+y^2}{2t}\right)I_{-\alpha}\left(\frac{xy}{t}\right),& x>0,\; y>0,\\
\vspace{-.4cm}\\
\displaystyle \frac{ C_\alpha 2^{\alpha}}{t^{1-\alpha}\Gamma(1-\alpha)}\exp\left(-\frac{x^2}{2t}\right),& x>0,\; y=0,\\
\vspace{-.4cm}\\
\displaystyle \frac{ C_\alpha 2^{\alpha}}{t^{1-\alpha}\Gamma(1-\alpha)}\exp\left(-\frac{y^2}{2t}\right),& x=0,\; y\geq 0.
\end{cases}
\end{split}
\end{align}
Moreover, it holds that $L^y_\infty=\infty$, $\P^{\al}_{x_0}$-a.s.
\end{prop}

The function $\overline{p}^{\al}_t(x,y)$ defined by \eqref{def:barpt} relates to the transition probability density function $ p^{\al}_t(x,y)$ of $\BES(-\alpha)$ in the following manner:
\begin{align}\label{pbar_p}
\bar p^{\al}_t(x,y)=\lim_{\tilde{y}\searrow y} p_t^{\al}(x,\tilde{y})\cdot C_\alpha \tilde{y}^{-1+2\alpha},\quad \forall\;x,y\geq 0.
\end{align}
This relation hold since the formula of $ p^{\al}_t(x,y)$ is as follows~\cite[p.446]{RY}:
\begin{align}
\label{density}
 p_t^{\al}(x,y)&=
 \begin{cases}
\displaystyle  \frac{x^{\alpha}y^{1-\alpha}}{t}\exp\left(-\frac{x^2+y^2}{2t}\right)I_{-\alpha}\left(\frac{xy}{t}\right),& x>0,\;y>0,\;t>0,\\
\vspace{-.4cm}\\
\displaystyle \frac{2^\alpha}{t^{1-\alpha}\Gamma(1-\alpha)}y^{1-2\alpha}\exp\left(-\frac{y^2}{2t}\right),& x=0,\;y>0,\;t>0.
\end{cases}
\end{align}
See \eqref{asymp:I} for \eqref{pbar_p} when $x>0$ and $y=0$. Note that \eqref{eq:LTdensity} and \eqref{pbar_p} imply 
\begin{align}\label{otf:extend}
\begin{split}
&\int_0^\infty  \left(\int_{0}^\infty f(t) p^{\al}_t(x_0,y)\d t\right)g(y)\d y\\
&\quad =\int_0^\infty \left(\E^{\al}_{x_0}\left[\int_0^\infty f(t)\d L_t^y\right] C_\alpha^{-1}y^{1-2\alpha}\right)g(y)\d y,
\end{split}
\end{align}
which can be deduced separately by using \eqref{DLT:otf} and the formula of $m_\alpha$ in \eqref{def:scale+speed}.

\begin{proof}[Proof of Proposition~\ref{prop:density}]
First, recall the following properties of $I_{\nu}$ for all $\nu>-1$: 
\begin{align}\label{asymp:I}
I_{\nu}(x)\sim \frac{x^\nu}{2^\nu\Gamma(1+\nu)},\quad x\searrow 0;\quad I_\nu(x)\sim\frac{\e^x}{\sqrt{2\pi x}},\quad x\nearrow \infty.
\end{align}
See \cite[(5.16.4) on p.136]{Lebedev} for \eqref{asymp:I} in the case of $\nu\geq 0$. The extension to $\nu\in (-1,0)$ follows easily by using \cite[the sixth identity in (5.7.9) on p.110]{Lebedev}.

Now, for $f(t)=\e^{-q t}$ with $q\in (0,\infty)$, \eqref{eq:LTdensity} is a particular case of \cite[Theorem~3.6.3 on p.85]{MR:Markov} by taking $m=m_\alpha$ and $P_t(x,\d y)=\bar{p}^{\al}_t(x,y)m_\alpha(\d y)$ in the setting of \cite[p.75]{MR:Markov}, where $m_\alpha$ is from \eqref{def:scale+speed}. This theorem in \cite{MR:Markov} applies since it is readily checked that $\int_0^\infty \e^{-q t}\bar p_t^{\al}(x,y)\d t $ is jointly continuous in $(x,y)\in \R_+^2$ by using \eqref{asymp:I} for a dominated convergence argument and is symmetric in $x$ and $y$. For extensions to general $f\geq 0$, the standard approximations suffice. In more detail, one starts with \eqref{eq:LTdensity} for $f$ given by linear combinations of exponential functions with negative exponents and then applies the Stone--Weierstrass theorem for more general $f$. 

Finally, the almost-sure equality $L^y_\infty=\infty$ can be justified by  \cite[Theorem~8 on p.114]{Bertoin} since 
the recurrence of $y$ \cite[p.442]{RY} holds. 
\end{proof}

The next proposition specializes \eqref{eq:LTdensity} to $y=0$. 

\begin{prop}\label{prop:LTdensity}
Given any $\alpha\in(0,1)$, $C_\alpha\in(0,\infty)$ for \eqref{def:scale+speed}, $f\in \B_+(\R_+)$ and $q\in (0,\infty)$, we have
\begin{align}
\E_{x_0}^{\al}\left[\int_0^\infty f(t)\d L^0_t\right]&=\frac{C_\alpha 2^{\alpha}}{\Gamma(1-\alpha)}\int_0^\infty f(t)t^{\alpha-1}
\exp\left(-\frac{x_0^2}{2t}\right) \d t,\quad \forall\;x_0\in \R_+,\label{eq:LTmain1}\\
\E_{x_0}^{(-\alpha)}\left[\int_0^\infty \e^{-q t}\d L^0_t\right]&=
\begin{cases}
\displaystyle \frac{1}{C'_\alpha q^\alpha} \frac{2^{1-\alpha}}{\Gamma(\alpha)}\widehat{K}_{\alpha}(\sqrt{2q}x_0)&x_0> 0,\label{eq:LTLT}\\
\vspace{-.4cm}\\
\displaystyle \frac{1}{C'_\alpha q^\alpha}&x_0=0,
\end{cases}
\end{align}
where $\widehat{K}_\alpha(x)$ is defined by \eqref{def:hK}, and
\begin{align}\label{def:Cpalpha}
C'_\alpha\;\defeq \;\frac{\Gamma(1-\alpha)}{C_\alpha 2^\alpha\Gamma(\alpha)}.
\end{align}
\end{prop}
\begin{proof}
Identity \eqref{eq:LTmain1} just combines \eqref{eq:LTdensity} and \eqref{def:barpt}, both for $y=0$.
To get \eqref{eq:LTLT} for the case of $x_0>0$, first, note that by \eqref{eq:LTmain1},
\begin{align}
\E_{x_0}^{\al}\left[\int_0^\infty \e^{- q t}\d L^0_t\right]&=
 \frac{C_\alpha 2^{\alpha}}{\Gamma(1-\alpha)}
\int_0^\infty \e^{- q t}t^{\alpha-1}\exp\left(-\frac{x_0^2}{2t}\right)\d t\notag\\
&= \frac{C_\alpha 2^{\alpha}}{\Gamma(1-\alpha) q^\alpha}
\int_0^\infty \e^{- t}t^{\alpha-1}\exp\left(-\frac{2 q x_0^2  }{4t}\right)\d t\label{eq:LTLT2-2}\\
&=\frac{C_\alpha 2^\alpha\Gamma(\alpha)}{\Gamma(1-\alpha) q^\alpha} \frac{2^{1-\alpha}}{\Gamma(\alpha)(\sqrt{2 q}x_0)^{-\alpha}}K_{-\alpha}(\sqrt{2 q}x_0),\notag
\end{align}
where the last equality uses \eqref{def:K}. Since $K_{-\alpha}(\sqrt{2 q}x_0)=K_{\alpha}(\sqrt{2 q}x_0)$ by \eqref{K:sym}, the last equality is enough to get \eqref{eq:LTLT} for the case of $x_0>0$. 
To obtain \eqref{eq:LTLT} for $x_0=0$, note that in this case, the integral in \eqref{eq:LTLT2-2} equals $\Gamma(\alpha)$. 
\end{proof}

In particular, the formula of $\E^{\al}_{x_0}[\int_0^\infty\e^{- q t}\d L^0_t]/\E^{\al}_{0}[\int_0^\infty\e^{- q t}\d L^0_t]$ by using \eqref{eq:LTLT} recovers the following formula from \cite[(2.4) on p.5239]{HM} with $\nu=-\alpha$:
\begin{align}\label{eq:T0alpha}
\E^{\al}_{x_0}[\e^{- q T_0(\rho)}]=\frac{2^{1-\alpha}}{\Gamma(\alpha)}\widehat{K}_\alpha(\sqrt{2 q}x_0).
\end{align}

\noindent {3).} 
Finally, we recall the following pathwise representations of the local times $\{L^0_t\}$. See \cite[Theorem~2.1 on p.210]{DRVY} and its proof.

\begin{prop}\label{prop:LTapprox}
Given any $\alpha\in(0,1)$, $C_\alpha\in(0,\infty)$ for \eqref{def:scale+speed}, and $x_0\geq 0$,
 the following holds under $\P^{\al}_{x_0}$:
\begin{gather}
\alpha^{-1}C_\alpha \rho_t^{2\alpha}=\alpha^{-1} C_\alpha x_0^{2\alpha}+2 C_\alpha \int_0^t \rho_s^{2\alpha-1}\d W_s + L^0_t,\label{char1:LT}\\
\sup_{0\leq t\leq T}\left\lvert\int_0^t \kappa_\vep(\rho_s)\d s-L^0_t\right\rvert\xrightarrow[\vep\to 0]{\rm a.s.}0,\quad \forall\; T\in (0,\infty),\label{char2:LT}
\end{gather}
where the stochastic integral in \eqref{char1:LT} is a martingale, and for \eqref{char2:LT},
\begin{align}\label{def:kappa}
\kappa_\vep(y)\;\defeq \;\frac{2C_\alpha(1-\alpha) \vep}{(\vep+y^2)^{2-\alpha}},\quad y\geq 0.
\end{align}  
\end{prop}

The almost-sure convergence in \eqref{char2:LT} can be viewed as a consequence of the joint continuity of $\{L^y_t;(y,t)\in\R^2_+\}$ and the occupation times formula in \eqref{DLT:otf}. This is due to the fact that $\{\kappa_\vep\}$ is an approximation to the identity under $m_\alpha(\d y)$, as the following lemma shows.

\begin{lem}
For all $\vep\in (0,\infty)$ and $\alpha\in (0,1)$, it holds that 
\[
\int_0^\infty \kappa_\vep(y)m_\alpha(\d y)=\int_{0}^\infty \frac{2(1-\alpha)\vep}{(\vep+y^2)^{2-\alpha}}  y^{1-2\alpha}\d y=1.
\]
The second equality extends to $\alpha=0$ as well.
\end{lem}
\begin{proof}
By changing variables with $z=\vep/y^2+1$, we get
\begin{align*}
\int_{0}^\infty \frac{2(1-\alpha)\vep}{(\vep+y^2)^{2-\alpha}}  y^{1-2\alpha}\d y =\int_0^\infty \frac{2(1-\alpha)\vep }{(\vep/y^2+1)^{2-\alpha}}y^{-3}\d y
&=(1-\alpha)\int_1^\infty \frac{\d z}{z^{2-\alpha}}=1,
\end{align*}
as required. 
\end{proof}

\section{Proofs of the other properties}\label{sec:other}
We prove Proposition~\ref{prop:esp} in Section~\ref{sec:esp}, Corollary~\ref{cor:FK} in Section~\ref{sec:FK}, and Propositions~\ref{prop:GEN}, \ref{prop:singular}, \ref{prop:DBM} in Section~\ref{sec:singular}, and Proposition~\ref{prop:osgood} in Section~\ref{sec:Osgood}. 

\subsection{Unique existence of the skew-product diffusions}\label{sec:esp}
\begin{proof}[Proof of Proposition~\ref{prop:esp}]
For the verification of \eqref{Zt:explosion},
see \cite[p.88]{Erickson} for $\BES(-\alpha)$, $\alpha\in (0,1/2)$, and Lemma~\ref{lem:explosion} for $\BES(-\alpha,\beta\da)$, $\alpha\in [0,1/2)$.

 The other conditions 
of Erickson's theorem are stated at the beginning of Section~2 in \cite{Erickson}. The radial processes are required to be conservative, or equivalently, show explosions in finite times, which obviously holds in all the cases. Hence, it remains to verify:
\begin{itemize}
\item [(a)] No sojourn at $0$: $\int_0^t \1_{\{0\}}(|Z_s|)\d s=0$ a.s. for $t>0$.
\item [(b)] Non-singular process (cf. \cite[p.111]{IM:Diffusion}): $\P_{x_0}(T_{x_0+}=0)=\P_{x_0}(T_{x_0-}=0)=1$ for $x_0>0$. 
\item [(c)] $\P_0(T_{0+}=0)=1$ and $\P_{0+}(T_0\leq t)=1$ for $t>0$. 
\end{itemize}
Here, $T_{x+}\defeq\inf\{t>0;|Z_t|>x\}$ and $T_{x-}\defeq\inf\{t>0;|Z_t|<x\}$. 

We first consider $\{|Z_t|\}\sim \BES(-\alpha)$ for $\alpha\in (0,1/2)$. For (a), the property is a standard result of Bessel processes \cite[(1.5) Proposition on p.442]{RY}. 
For (b) and (c), it suffices to note that for $q,x_0>0$, $\E_{x_0}[\e^{-qT_{x_0\pm}}]=\E_{0}[\e^{-qT_{0+}}]=1$, and $\int_0^\infty q\e^{-qt}\P_{0+}(T_0\leq t)\d t=\E_{0+}[\e^{-qT_0}]=1$, all of which can be deduced from the explicit formulas of the Laplace transforms of hitting times \cite[(2.2) and (2.5)]{HM}. Note that by the same argument, (b) holds when $\{|Z_t|\}\sim \BES(0)$. 

For $\BES (-\alpha,\beta\da)$, $\alpha\in(0,1/2)$, (a) holds automatically 
by \eqref{def:BESab}, and (b) and (c) can be verified by using the formulas of the Laplace transforms of hitting times recalled in \eqref{preGIG}. 

For $\BES(0,\beta\da)$, (a) holds because $0$ is instantaneously reflecting \cite[Theorem~2.1]{DY:Krein}. To see (b), take \eqref{T0-1} with $A=\{T_{x_0\pm}\leq t\}$ for $t>0$. The right-hand side of \eqref{T0-1} with this choice of $A$ converges to $1$ as $t\searrow 0$ because $\1_A=1$ $\P^{(0)}_{x_0}$-a.s. as observed above, whereas $\{T_{x_0\pm}\leq t\}\cap \{T_0>t\}\to \{T_{x_0\pm}=0\}$ $\P_{x_0}^{\zbe}$-a.s. To see (c), we use again the fact that 
$0$ is instantaneously reflecting to get the first part, and the formula of the Laplace transform of $T_0$, already recalled in \eqref{T0-2}, to get the second part. 
 \end{proof}

\subsection{The Feynman--Kac-type formula: the nonexistence of exponential forms}\label{sec:FK}
In this subsection, we fix $\beta\in(0,\infty)$ and prove Corollary~\ref{cor:FK}. The following lemma extends \eqref{FK:char} of Hypothesis~\ref{hyp:FK1}. The proof applies the standard Markovian iteration by starting with the case of $\Phi_t=\prod_{j=1}^n\phi_j(Z_{r_j})$ for $0<r_1<\cdots<r_n=t$. We omit the remaining details.

\begin{lem}\label{lem:FKcontra}
Under Hypothesis~\ref{hyp:FK1}, it holds that 
\begin{align}\label{FK:char1}
\begin{split}
&\E_{z_0}^\Bbb Q[\e^{A_t}\Phi_t;\zeta>t]=\E_{z_0}^{(0),\beta\da}\left[D^{z_0}_t\Phi_t\right],\\
&\hspace{2cm}\forall\; z_0\in \Bbb C\setminus\{0\},\;t\in (0,\infty),\;
\Phi\in \B_+(C([0,t],\Bbb C)),
\end{split}
\end{align}
where $\Phi_t$ is a shorthand notation for $ \Phi(Z_r;r\leq t)$. 
\end{lem}

\begin{proof}[Proof of Corollary~\ref{cor:FK}]
The aim of this proof is to obtain a contradiction to the existence of $(\Bbb Q,A,\zeta)$. To this end, the crucial fact we apply below is that every nonnegative c\`adl\`ag supermartingale $\{N_t\}$ vanishes at all finite times $t\geq \inf\{s\geq 0;N_s=0\}\wedge \inf\{s\geq 0;N_{s-}=0\}$. See \cite[Proposition~3.4 on p.70]{RY}. In the following argument, fix $z_0\in \Bbb C\setminus\{0\}$. \smallskip  

\noindent \hypertarget{cor:FK-1}{{\sc Step 1.}}\;
In this step, we show that the existence of $(\Bbb Q,A,\zeta)$ implies
\begin{align}\label{contradiction}
\E_{z_0}^{(0),\beta\da}\left[\left.\E^{(0),\beta\da}_{0}\left[\frac{1}{K_0(\sqrt{2\beta}\lvert Z_{t-s}\rvert )}\right]\right\rvert_{s=T_0(Z)};t\geq T_0(Z)\right]=0,\quad\forall\;t\in (0,\infty).
\end{align}

To see \eqref{contradiction}, first we apply the fact of supermartingales just recalled to the process
\[
M^{z_0}_t\,\defeq \,\e^{-A_t(Z_r;r\leq t)}D^{z_0}_t. 
\]
This process is indeed a supermartingale under $\P^{(0),\beta\da}_{z_0}$, since for all $\Phi\geq 0$ and $s,t\in (0,\infty)$, Lemma~\ref{lem:FKcontra} gives the equalities below:
\[
\E_{z_0}^{(0),\beta\da}[M^{z_0}_{t+s}\Phi_t]=\E_{z_0}^\Bbb Q[\Phi_t;\zeta>t+s]\leq \E_{z_0}^\Bbb Q[\Phi_t;\zeta>t]=\E_{z_0}^{(0),\beta\da }[M^{z_0}_{t}\Phi_t].
\]
Then it follows from the fact of supermartingales recalled above that 
\begin{align}\label{eq:FKcontra}
\E_{z_0}^{(0),\beta\da}[M^{z_0}_t\Phi_t]
=\E_{z_0}^{(0),\beta\da}[M^{z_0}_t\Phi_t;t<T_0(M^{z_0})]=\E_{z_0}^{(0),\beta\da}[M^{z_0}_t\Phi_t;t<T_0(Z)],
\end{align}
where the second equality uses the assumption that $A_t(\mathrm w_r;r\leq t)$ is real-valued so that 
 $T_0(M^{z_0})= T_0(D^{z_0})=T_0(Z)$. 
Next, by using the definition of $\{M^{z_0}_t\}$ and choosing $\Phi=\e^{-\beta t}\e^{A_t(Z_r;r\leq t)}/K_0(\sqrt{2\beta}|z_0|)$, the last equality of \eqref{eq:FKcontra} gives
\begin{align}\label{nonFK}
\E_{z_0}^{(0),\beta\da}\left[\frac{1}{K_0(\sqrt{2\beta}\lvert Z_t \rvert)}\right]
= \E_{z_0}^{(0),\beta\da}\left[\frac{1}{K_0(\sqrt{2\beta}\lvert Z_t \rvert)};t<T_0(Z)\right],
\end{align}
where the expectation on the left-hand side is finite by \eqref{id:2-1} and Lemma~\ref{lem:contexp}. 
Taking the difference of both sides of \eqref{nonFK} and then applying the strong Markov property at $T_0(Z)$ yield \eqref{contradiction}. (This strong Markov property holds under $\BES(0,\beta\da )$ and is extended to $\{Z_t\}$ under $\P^{(0),\beta\da}_{z_0}$ thanks to Erickson's theorem~\cite[Theorem~1 on pp.75--76]{Erickson}.) \smallskip 

\noindent {\sc Step 2.}\; Now, \eqref{contradiction} leads to the required contradiction since its left-hand side is actually strictly positive by the following two properties:
\begin{itemize}[leftmargin=3\labelsep]
\item The stopping time $T_0(Z)=T_0(\lvert Z\rvert)$ under $\P^{(0),\beta\da}_{z_0}$ is distributed as the generalized inverse Gaussian distribution ${\rm GIG}(0;\lvert z_0 \rvert,\sqrt{2\beta})$:
\begin{align}\label{def:GIG}
\P^{(0),\beta\da}_{z_0}(T_0(Z)\in \d t)=\frac{1}{2K_0(\sqrt{2\beta}\lvert z_0 \rvert)}\exp\left(-\beta t-\frac{\lvert z_0 \rvert^2}{2t}\right)\frac{\d t}{t},\quad t\in (0,\infty),
\end{align}
by \eqref{def:K} and \eqref{T0-2}.
Hence, $T_0(Z)$ under $\P^{\zbe}_{z_0}$ has a continuous, everywhere strictly positive probability density on $(0,\infty)$. 
\item For all $t\in(0,\infty)$,
$\E^{(0),\beta\da}_{0}[K_0(\sqrt{2\beta}\lvert Z_t \rvert)^{-1}]>0$ by \eqref{BESab:density}. 
\end{itemize}
The proof of Corollary~\ref{cor:FK} is complete.
\end{proof}

\begin{rmk}
It is pointed out by an anonymous referee that an alternative possible approach is to use the Revuz correspondence and compare the resolvents of $\{P^\beta_t\}$ with the resolvents from the standard Feynman--Kac formula. See, e.g., \cite{BB:09,Y:RevuzFK} for general methods of this direction. \qed 
\end{rmk}

\subsection{Stochastic relative motion: the generator, non-Gaussianity and singular drift}\label{sec:singular}
We give the proofs of Propositions~\ref{prop:GEN}, \ref{prop:singular} and \ref{prop:DBM} in the same order.
{\bf For the remaining of Section~\ref{sec:singular}, recall the convention for dot products and differentiations
stated below (\ref{def:bbeta}). The Fourier transforms below are understood in the same way.}

\begin{proof}[Proof of Proposition~\ref{prop:GEN}]
We compute $\ms A^{\beta\da}_0 f(z)$ as the derivative of the finite-variation part of $f(\two Z_t)-f(z)$ at $t=0$ under $\P^{\zbe}_{z/\two}$. Under the polar coordinates, we obtain from  \eqref{gen:BESab} and \eqref{Z:skewproduct} that 
\begin{align}\label{polar:2-1}
\ms A^{\beta\da}_0 f(z)=\frac{\partial^2 f}{\partial r^2}+\left(\frac{1}{r}-2\sqrt{\beta}\frac{K_1}{K_0}(\sqrt{\beta }r)\right)\frac{\partial f}{\partial r}+\frac{1}{r^2}\frac{\partial^2 f}{\partial \theta^2},\quad f\in \C_c^2,\;z\neq 0.
\end{align}
Since $\ms A^{\beta\da}_0 $ reduces to the Laplacian when $\beta=0$ and, for $x=r\cos \theta$ and $y=r\sin \theta$,
\begin{align*}
\frac{\partial f}{\partial r}=\frac{\partial f}{\partial x}\cos \theta+\frac{\partial f}{\partial y}\sin\theta=\frac{\partial f}{\partial x}\frac{x}{\sqrt{x^2+y^2}}+\frac{\partial f}{\partial y}\frac{y}{\sqrt{x^2+y^2}},
\end{align*}
we get \eqref{def:gen} from \eqref{polar:2-1}. The other formula of this proposition follows from  \eqref{def:gen} and \eqref{id:2-1} since $\Delta_\beta$ is a self-adjoint extension of the Laplacian restricted to $\C_0^\infty(\Bbb C\setminus\{0\})$ and is local on $\Bbb C\setminus\{0\}$ so that not just $f\in \C^\infty_0(\Bbb C\setminus\{0\})$ can be applied. See also
\cite[Theorem~2.3]{AGHH:2D} for a description of the operator $\Delta_\beta$ generating $\{P^\beta_t\}$. 
\end{proof}

\begin{proof}[Proof of Proposition~\ref{prop:singular}]
Recall that  $0$ is polar under planar Brownian motion \cite[(2.7) Proposition on p.191]{RY}: $\P^{(0)}(T^+_0<\infty)=0$ for $T_0^+\,\defeq\inf\{t>0;Z_t=0\}$. The required property for $\{Z_t\}_{t\in [0,t_1]}$ under $\P^{\zbe}$, $t_1>0$, then follows by decomposing it according to $\Omega=\{ T^+_0\leq t_1\}\cup\{T^+_0>t_1\}$, since $\P^{(0)}( T^+_0\leq  t_1)=0$ but $\P^{\zbe}(T^+_0\leq t_1)>0$ by \eqref{def:GIG}. 

The required property for $\{Z_t\}_{t\in [0,\infty)}$ under $\P^{\zbe}$ follows similarly. Here, the decomposition is according to $\Omega=\{T_0^+<\infty\}\cup \{T_0^+=\infty\}$.
\end{proof}

The next lemma prepares the proofs of Propositions~\ref{prop:DBM} and \ref{prop:osgood}.

\begin{lem}
For all $\alpha\in \R$ and $a>0$,
\begin{align}
\frac{\d}{\d a}\left(a\frac{K_{1-\alpha}}{K_\alpha}(a)\right)=-a+a\left(\frac{K_{1+\alpha}K_{1-\alpha}}{K_\alpha^2}\right)(a).\label{der:K1/K0}
\end{align}
\end{lem}
\begin{proof}
We inspect the computations in \eqref{d1logG} and \eqref{d2logG}. This shows
 $(\d/\d b)\log G_\alpha(b)=[-K_{1-\alpha}(x)/K_\alpha(x)]\big|_{x=\sqrt{b}}\cdot [1/(2\sqrt{b})]$ and, with the shorthand notation $K_\nu=K_\nu(a)$, 
\begin{align*}
\frac{\d}{\d a}\left(a\frac{K_{1-\alpha}}{K_\alpha}\right)&=\frac{K_{1-\alpha}}{K_\alpha}+\frac{(2\alpha-1)K_{1-\alpha}}{K_\alpha}-a+\frac{aK_{1-\alpha}^2}{K_\alpha^2}
=-a+\frac{(2\alpha K_\alpha+a K_{1-\alpha})K_{1-\alpha}}{K_\alpha^2},
\end{align*}
which gives \eqref{der:K1/K0} by using the identity $K_{\alpha-1}-K_{\alpha+1}=-2\alpha K_\alpha/a$ \cite[(5.7.9), p.110]{Lebedev}. Note that we have used the even parity of $K_\nu(\cdot)$ \cite[(5.7.10), p.110]{Lebedev} repeatedly here.
\end{proof}

\begin{proof}[Proof of Proposition~\ref{prop:DBM}]
\noindent {\rm (1$\cc$)}\;  First, the inequalities in \eqref{asymp0:drift} and \eqref{asymp0:drift1} are immediate consequences of \eqref{asymp:K0}. To prove the required properties of $\nabla_z\cdot b_\beta(z)$, 
by \eqref{der:K1/K0} for $\alpha=0$ and the definition \eqref{def:bbeta} of $b_\beta$,
\begin{align*}
\nabla_z \cdot b_\beta(z)
&=-2\nabla_z \left(\sqrt{\beta}|z|\frac{K_1}{K_0}(\sqrt{\beta}|z|)\right) \cdot \frac{z}{|z|^2}-2\left(\sqrt{\beta}|z|\frac{K_1}{K_0}(\sqrt{\beta}|z|)\right)\nabla_z\cdot \frac{z}{|z|^2}\\
&=2\beta-2\beta\left(\frac{K_1}{K_0}\right)^2(\sqrt{\beta}|z|),
\end{align*}
where the last equality uses \eqref{der:K1/K0}, $\nabla_z|z|=z/|z|$ and $\nabla_z\cdot (z/|z|^2)=0$. We have proved the second equality of the first line of \eqref{asymp:drift}, whereas the first equality there just uses the definition \eqref{def:bbeta} of $b_\beta$. Also, the second line of \eqref{asymp:drift} holds since
by \eqref{asymp:K0},
\begin{align}\label{asymp:K1K0}
\mbox{as } x\searrow 0,\quad
\frac{K_1}{K_0}(\sqrt{\beta}x)\sim \frac{1}{\sqrt{\beta}x[\log (\sqrt{\beta}x)^{-1}]}\sim \frac{1}{\sqrt{\beta}x(\log x^{-1})}.
\end{align}
By the second line of \eqref{asymp:drift} and the fact that $\int \d r/[r(\log r)^2]=-1/\log r+C$ for $0<r<1$,
$\nabla_z\cdot b_\beta(z)\notin L^2_{\rm loc}(\Bbb C)$. \smallskip 

\noindent {\rm (2$\cc$)}\; For the first part, by \eqref{asymp:K1K0}, it is enough to consider $p\in [1,\infty)$. 
In this case, by \eqref{asymp0:drift} and the equivalence of the $2$-norm and the $p$-norm in $\Bbb C$,
 $b_\beta\psi\in L^p$ if and only if the next integral for small enough $\delta=\delta(\psi)\in (0,1)$ converges:
\[
\int_{|z|<\delta} \frac{\d z}{||z|\log |z|^{-1}|^p}
\asymp \int_{0}^{\delta} \frac{\d r}{r^{p-1}|\log r|^p},
\]
where ``$\asymp$'' uses the polar coordinates. Since $\int \d r/(r\log^2 r)=-1/\log r+C$ in the critical case of $p=2$, the first required property of (2$\cc$) follows. The second required property can be obtained by almost the same argument, now using the fact that $\int_{0+}\d r/r^{p-1}<\infty$ if and only if $p<2$. 
\smallskip 

\noindent {\rm (3$\cc$)}\; First, for $f(r)\,\defeq\,(r\log r^{-1})^{-1}$ and  small enough $\delta=\delta(\psi)\in (0,\e^{-1})$, 
\eqref{asymp0:drift} gives 
\begin{align}\label{bdd:bbeta}
C_{\ref{bdd:bbeta}}(\beta,\psi)^{-1} f(|z|)\leq |b_\beta(z)\psi(z)|\leq C_{\ref{bdd:bbeta}}(\beta,\psi) f(|z|),\quad \forall\;0<|z|\leq \delta,
\end{align}
where $C_{\ref{bdd:bbeta}}(\beta,\psi)\geq 1$. 
Since $\supp(\psi)$ is bounded, by \eqref{bdd:bbeta}, it suffices to show, for $p\geq1$, 
\begin{align}\label{weakL}
 \sup_{\ell>0}\ell\cdot {\rm Leb}\{z\in B(0,\delta);f(|z|)>\ell\}^{1/p}<\infty\Longleftrightarrow p\in [1,2]. 
\end{align}

To get \eqref{weakL}, note that $f(r)$ is strictly decreasing on $(0,\e^{-1}]$. Hence, for all $\ell>0$, 
\begin{align}\label{weakL:aux}
\ell\cdot {\rm Leb}\{z\in B(0,\delta);f(|z|)>\ell\}^{1/p}&=\ell\cdot {\rm Leb}\{z\in B(0,\delta);|z|<f^{-1}(\ell)\}^{1/p}.
\end{align}
Applying \eqref{weakL:aux}, $f(0+)=\infty$, and the change of variables $\ell=f(r)$ in the same order, we get 
\begin{align*}
&\quad\;\lim_{\ell\to\infty}\ell\cdot {\rm Leb}\{z\in B(0,\delta);f(|z|)>\ell\}^{1/p}=\lim_{\ell\to \infty}\ell\cdot  \{\pi [f^{-1}(\ell)\wedge \delta]^2\}^{1/p}\\
&=\lim_{\ell\to \infty}\ell\cdot  [\pi f^{-1}(\ell)^2]^{1/p}
=\lim_{r\to 0}f(r)\cdot (\pi r^2) ^{1/p}=\lim_{r\to 0}\frac{\pi^{1/p}}{r^{1-2/p}\log r^{-1}} ,
\end{align*}
which is finite if and only if $p\leq 2$. We obtain the required property of (3$\cc$) by \eqref{weakL}. \smallskip 

\noindent {\rm (4$\cc$)}\; We recall the definition of the Besov space $B_{\infty,\infty}^s$ for $s\in \R$. See \cite[p.61 and Definition~2.68 on p.99]{BCD:FPDE}. Let $\chi,\varphi$ be nonnegative radial $\C_c^\infty$-functions on $\Bbb C$ satisfying:
\begin{itemize}
\item The support $\supp(\chi)$ of $\chi$ is contained in a ball centered at the origin,
and $\supp(\varphi)$ is contained in an annulus centered at the origin. 
\item $\chi(z)+\sum_{j\geq 0}\varphi(2^{-j}z)=1$ for all $z\in \Bbb C$. 
\item $\supp(\chi)\cap \supp(\varphi(2^{-j}\cdot))=\varnothing$ for all $j\geq 1$ and $\supp(\varphi(2^{-j}\cdot))\cap \supp(\varphi(2^{-k}\cdot))=\varnothing$ whenever $j,k\geq 0$ satisfy $|j-k|>1$.  
\end{itemize}
Let $\mc F $ denote the Fourier transform (in $\R^2$), and define the Littlewood--Paley blocks for $u$ by 
\begin{align}\label{def:block}
\Delta_{-1}u(z)&\,\defeq\, (\mathcal F^{-1}\chi)\star u(z),\quad
\Delta_ju(z)\,\defeq\,2^{2j}[\mathcal F^{-1}\varphi(2^j\cdot)]\star u(z),\quad j\in \Bbb Z_+.
\end{align}
Then $B^s_{\infty,\infty}$ is defined as the set of tempered distributions $u$ such that $\|u\|_{B^s_{\infty,\infty}}<\infty$, where
 \begin{align}\label{def:Besov}
\|u\|_{B^s_{\infty,\infty}}\defeq\sup_{j\geq -1}2^{js}\|\Delta_ju\|_\infty.
\end{align}

Below we take three steps to prove the required property of (4$\cc$).
Roughly speaking, the argument below for Steps~1 and~2 views $\|b_\beta\psi\|_{B^s_{\infty,\infty}}$
by $b_\beta(2^{-j}\cdot)$ for large $j$ and \eqref{asymp:K0}. \smallskip

\noindent {\sc Step 1.}\;  We first show that $\Re(b_\beta\psi)\in B^{-1}_{\infty,\infty}$, hence $\in B^{s}_{\infty,\infty}$ for all $s<-1$ as well. The case $\Im(b_\beta\psi)$ can be handled in the same way, so we omit the proof for it. 

The proof of $\Re(b_\beta\psi)\in B^{-1}_{\infty,\infty}$ uses the following preparation. Write \begin{align}\label{def:u0}
u_0\,\defeq\,\Re(b_\beta\psi).
\end{align}
Also, since the inverse Fourier transform of a radial function stays unchanged after an orthogonal transformation of frequency, by (2$\cc$) and \eqref{def:Besov}, it suffices to show the following inequality: for any radial Schwartz function $\theta$, 
\begin{align}
\infty>\sup_{j\in \Bbb Z_+} \sup_{z\in \Bbb C}2^{-j}|2^{2j}\theta(2^j\cdot)\star u_0(z)|&=\sup_{j\in \Bbb Z_+} \sup_{z\in \Bbb C}2^{-j}|2^{2j}\theta(2^j\cdot)\star u_0(2^{-j}z)|\notag\\
&=\sup_{j\in \Bbb Z_+} \sup_{z\in \Bbb C}2^{-j}|\theta\star u_0(2^{-j}\cdot)(z)|.\label{Besov:1}
\end{align}

Write $u_1\defeq\,\Re(b_\beta \1_{B(0,\delta)})$ and $u_2\defeq\,\Re(b_\beta\1_{B(0,\delta)^\complement})$
for fixed $\delta\in (0,1)$. Then
\begin{align}
2^{-j}\theta\star u_0(2^{-j}\cdot)(z)
=2^{-j}\theta\star u_1(2^{-j}\cdot)(z)+2^{-j}\theta\star u_2(2^{-j}\cdot)(z),\quad j\in \Bbb Z_+.\label{Besov:2}
\end{align}
For the right-hand side of \eqref{Besov:2}, each of the two terms is bounded by $C(\beta,\delta,\theta)$.
We can bound the first term this way since, by $|z''u_0(z'')|\leq C(\beta,\delta)$ for $|z''|\leq \delta$ due to \eqref{asymp:K0},
\begin{align*}
&\quad\; 2^{-j}\theta\star u_1(2^{-j}\cdot)(z)=
\int_{|z-z'|\leq 2^j\delta} \frac{\theta(z')}{|z-z'|} 2^{-j}|z-z'|u_0(2^{-j}(z-z'))\d z'\\
&\leq C(\beta,\delta) \int_{|z-z'|\leq 2^j\delta}\frac{|\theta(z')|}{|z-z'|}\d z'\leq C(\beta,\delta)\biggl(\int_{|z-z'|\leq 1}\frac{|\theta(z')|}{|z-z'|}\d z'+\int_{|z'-z|> 1}|\theta(z')|\d z'\biggr),
\end{align*}
and the first integral on the right-hand side is bounded by $C(\theta)$ thanks to H\"older's inequality.
We conclude that the inequality in \eqref{Besov:1} holds. \smallskip 

\noindent {\sc Step 2.}\; 
Next, we show that $\Re(b_\beta\psi),\Im(b_\beta\psi)\notin B^s_{\infty,\infty}$ for any $s>-1$. It suffices to consider $s\in (-1,0)$. Essentially, our goal is to sharpen the estimate in Step 1.

For the case of $\Re(b_\beta\psi)$, we use again the shorthand notation $u_0$ defined in \eqref{def:u0}. 
In this case, similar to \eqref{Besov:1}, we have 
\begin{align}\label{u1:main}
2^{js}\|\Delta_ju_0\|_{\infty}\geq2^{js} |\mc F^{-1}\varphi\star u_0(2^{-j}\cdot)(z)|,\quad \forall\;j\in \Bbb Z_+,\;z\in \Bbb C.
\end{align}
Below, we estimate the right-hand side as $j\to\infty$ via $u_1,u_2$ defined before \eqref{Besov:2}. 

First, we show that 
\begin{align}\label{u1:cond}
\lim_{j\to\infty}|2^{js} \mc F^{-1}\varphi\star u_1(2^{-j}\cdot)(z)|=\infty,\quad \forall\;s\in (-1,0),
\end{align}
under the assumption that
\begin{align}\label{u1:3}
\int_{\Bbb C} \mc F^{-1}\varphi(z')\Re\left(\frac{z-z'}{|z-z'|^2}\right)\d z'\neq 0\quad\mbox{for some $z\in \Bbb C$.}
\end{align}
We will verify \eqref{u1:3} in Step~3.  
To obtain \eqref{u1:cond}, with $\psi_1\defeq\,\psi\1_{B(0,\delta)}$, we compute
\begin{align}
&\quad\; 2^{js} \mc F^{-1}\varphi\star u_1(2^{-j}\cdot)(z)\notag\\
&=2^{js}\int_{\Bbb C}\mc F^{-1}\varphi(z')\Re (b_\beta)(2^{-j}z-2^{-j}z')\psi_1(2^{-j}z-2^{-j}z')\d z'\label{u1:1-0}\\
&=\frac{2^{js+j}}{j}\int_{\Bbb C}\mc F^{-1}\varphi(z')\frac{j \cdot  \Re(b_\beta)(\zeta)|\zeta|\log|\zeta|^{-1}|_{\zeta=2^{-j}(z-z')}}{ |z-z'|\log (2^j |z-z'|^{-1})}\psi_1(2^{-j}z-2^{-j}z')\d z'\notag,
\end{align}
and
\begin{align}
&\int_{\Bbb C}\mc F^{-1}\varphi(z')\frac{j \cdot  \Re(b_\beta)(\zeta)|\zeta|\log|\zeta|^{-1}|_{\zeta=2^{-j}(z-z')}}{ |z-z'|\log (2^j |z-z'|^{-1})}\psi_1(2^{-j}z-2^{-j}z')\d z'\notag\\
&\xrightarrow[j\to\infty]{}\int_{\Bbb C}\mc F^{-1}\varphi(z')\frac{-2\Re(z-z')}{|z-z'|^2}\d z',\label{u1:2}
\end{align}
so that \eqref{u1:cond} follows if we recall \eqref{u1:3} and $s\in (-1,0)$. Here, we obtain \eqref{u1:2} by dominated convergence, using \eqref{asymp:K1K0}, the assumption $\psi(0)=1$, and the fast decay of $\mathcal F^{-1}\varphi(z')$ as $|z'|\to\infty$. Also, $2^{js} \mc F^{-1}\varphi\star u_2(2^{-j}\cdot)(z)\to 0$  as $j\to\infty$ by \eqref{u1:1-0} with $(u_1,\psi_1)$ replaced by $(u_2,\psi_2)$, where $\psi_2\defeq\,\psi\1_{B(0,\delta)^\complement}$.

Up to this point, for any $s\in (-1,0)$, we have \eqref{u1:main}, \eqref{u1:cond} under the assumption of \eqref{u1:3}, the limit $2^{js} \mc F^{-1}\varphi\star u_2(2^{-j}\cdot)(z)\to 0$, and the fact $u_0=u_1+u_2$. Accordingly, under \eqref{u1:3}, 
\[
\lim_{j\to\infty}2^{js}\|\Delta_ju_0\|_\infty=\infty,\quad \forall\;s\in (-1,0),
\] 
so by \eqref{def:Besov}, $u_0\notin B^s_{\infty,\infty}$ for any $s\in (-1,0)$.  
Similarly, $\Im(b_\beta \psi)\notin B^s_{\infty,\infty}$ for any $s\in (-1,0)$ under the assumption of
\begin{align}\label{u1:4}
\int_{\Bbb C} \mc F^{-1}\varphi(z')\Im\left(\frac{\tilde{z}-z'}{|\tilde{z}-z'|^2}\right)\d z'\neq 0\quad\mbox{for some $\tilde{z}\in \Bbb C$.}
\end{align}

\noindent {\sc Step 3.}\; 
It remains to verify \eqref{u1:3} and \eqref{u1:4}. 
Note that \eqref{u1:3} and \eqref{u1:4} are equivalent since we have the following
where all the objects are viewed as their counterparts in $\R^2$:
\begin{align*}
 \int_{\Bbb C}\mc F^{-1}\varphi(z') \Re\left(\frac{z- z'}{|z-z'|^2}\right)\d z'
&= \int_{\Bbb C}\mc F^{-1}\varphi(\Im z'+\i \Re z')\Im \left(\frac{\i z- z'}{| \i z- z'|^2}\right)\d z'\\
&=\int_{\Bbb C}\mc F^{-1}\varphi( z')\Im \left(\frac{\i z- z'}{| \i z- z'|^2}\right)\d z'.
\end{align*}
Here, the first equality uses a change of variables in $\R^2$ by replacing $(\Re z',\Im z')$ with $ (\Im z',\Re z')$, which is an orthogonal transformation, and the last equality uses the radial symmetry of $\varphi$.
By their equivalence, both of \eqref{u1:3} and \eqref{u1:4} hold if we can verify that
\begin{align}\label{def:z2}
\int_{\Bbb C}\mc F^{-1}\varphi(z')\frac{z_0-z'}{|z_0-z'|^2}\d z\neq 0\quad\mbox{for some }z_0\in \Bbb C.
\end{align}

To prove \eqref{def:z2}, suppose the converse so that equality holds for all $z_0\in \Bbb C$. 
Then  
\[
0=\int_{\Bbb C} \mc F^{-1}\varphi(z')\frac{z_0-z'}{|z_0-z'|^2}\d z'=\nabla_{z_0}\int_{\Bbb C} \mathcal F^{-1}\varphi(z')\log |z_0-z'|\d z',\quad \forall\;z_0\in \Bbb C,
\]
where the second equality holds by dominated convergence since $\mc F^{-1}\varphi$ is a Schwartz function.
By the fundamental theorem of calculus, the last equality implies that
$z_0\mapsto \int \mathcal F^{-1}\varphi(z')\log |z_0-z'|\d z'$ is only a constant. Then we get a contradiction to the fact that $z_0\mapsto (2\pi)^{-1}\int \mathcal F^{-1}\varphi(z')\log |z_0-z'|\d z'$ is a solution to the two-dimensional Poisson equation $\Delta_{z'} u=\mc F^{-1}\varphi(z')$ \cite[Lemma~4.2 on p.55]{GT:PDE}, and this PDE does not allow a constant solution since $\mc F^{-1}\varphi\neq 0$. Hence,  \eqref{def:z2} must hold. The proof of (4$\cc$) is complete.
\end{proof}

\subsection{Stochastic relative position: the failing concave Osgood condition}\label{sec:Osgood}

\begin{proof}[Proof of Proposition~\ref{prop:osgood}]
We proceed with three steps.\smallskip\\
\noindent {\sc Step 1.}\; To verify the increasing monotonicity and concavity of
 of $\kappa_\alpha$ defined in \eqref{def:kappao}, it suffices to consider the case of $\alpha=0$.  First, $\kappa_0$ is increasing since $(\d/\d x)(-1/\log x)=1/(x\log^2 x)$. To see the concavity of $\kappa_0$, note that the concavity of $u\mapsto -1/\log (u\wedge \frac{\e^{-2}}{2})$ holds since $u\mapsto -1/\log u$ is increasing and concave on $[0,\e^{-2}/2]$. Here, the concavity of  $u\mapsto -1/\log u$ holds since $(\d^2/\d u^2)(-1/\log u)=-(\log u+2)/(u^2\log^3 u)$,
which is strictly negative in $(0, \e^{-2}]$.\smallskip \\
\smallskip 
\noindent {\sc Step 2.}\; To see \eqref{Osgood} for all $\alpha\in [0,1/2)$, note that 
\begin{align}\label{mubeta:scaling}
\mu^{\beta\da}_\alpha(x)=\mu_\alpha( |2\beta x|),\quad \forall\;x\in \R,
\end{align}
by \eqref{def:mu}, where $\mu_\alpha(x)\,\defeq\, \mu^{(1/2)\da}_\alpha(x)$. Since $\kappa_\alpha$ is increasing, \eqref{Osgood} holds for all $\beta>0$ a soon as we have
\begin{align}\label{Osgood0-1}
|\mu_\alpha(y)-\mu_\alpha(x)|\leq \kappa_\alpha(|y-x|),\quad \forall\;0\leq x\leq y<\infty.
\end{align}
\indent To prove \eqref{Osgood0-1}, 
first, note that 
\begin{align}
\mu_\alpha(y)-\mu_\alpha(x)=(y-x)-\int_x^y \frac{K_{1+\alpha}K_{1-\alpha}}{K_\alpha^2}(\sqrt{ a}\,)\d a,\quad \forall\; 0\leq x\leq y,\label{Osgood:diff}
\end{align}
since by \eqref{def:mu} with $\beta=1/2$
and by \eqref{der:K1/K0}, 
\begin{align*}
\forall\;a>0, \quad
\frac{\d}{\d a}\mu_\alpha(a)&=-2\biggl(-b+b\frac{K_{1+\alpha}K_{1-\alpha}}{K_\alpha^2}(b)\biggr)\biggr|_{b=\sqrt{a}}\cdot \frac{1}{2 \sqrt{a}}
=1-  \frac{K_{1+\alpha}K_{1-\alpha}}{K_\alpha^2}(\sqrt{ a}).
\end{align*}
Second, for $\alpha\in (0,1/2)$, \eqref{Osgood0-1}  follows immediately by applying \eqref{asymp:K0} and \eqref{asymp:Kinfty} to \eqref{Osgood:diff} since $\int (\sqrt{a})^{-(2-2\alpha)}\d a=a^{\alpha}/\alpha+C$. For $\alpha=0$, \eqref{Osgood0-1} can be seen by using 
again
\eqref{Osgood:diff} and the following asymptotic representations implied by \eqref{asymp:K0} and~\eqref{asymp:Kinfty}:
\begin{align}\label{K1K0:finalasymp}
(K_1/K_0)^2(x)\sim 1/(x^2\log^2 x),\quad x\to 0;\quad (K_1/K_0)^2(x)\sim 1,\quad x\to\infty,
\end{align}
where $\log^nx\,\defeq\,(\log x)^n$. In more detail, for $y-x\leq \e^{-2}/2$ and $ y\leq \e^{-2}$, \eqref{Osgood0-1} holds since
\begin{align}
\int_{x}^y  \left(\frac{K_1}{K_0}\right)^2(\sqrt{a}\,)\d a&\less \int_{x}^y\frac{\d a}{ a\log^2 a}
\leq \int_{0}^{y-x}\frac{\d a}{ a\log^2a}=-\frac{1}{\log (y-x)},\label{mumod:1}
\end{align}
where
the $\less$-inequality uses the first asymptotic representation in \eqref{K1K0:finalasymp}, and
the last inequality hold since $t\mapsto1/( t\log^2t)$ decreases on $(0, \e^{-2}]$.
To complete the proof of \eqref{Osgood0-1} with $\alpha=0$, one may consider the following cases separately: (1) $y-x\leq \e^{-2}/2$ and $y>\e^{-2}$, (2) $y-x>\e^{-2}/2$ and $x\leq \e^{-2}/8$, and (3) 
$y-x>\e^{-2}/2$ and $x>\e^{-2}/8$. We omit the details.\smallskip

\noindent {\sc Step 3.}\; Finally, we verify \eqref{Osgood1} for $\alpha\in [0,1/2)$. To get \eqref{Osgood1} for $\beta=1/2$, we use \eqref{Osgood:diff} and note that by L'H\^opital's rule and \eqref{asymp:K0},  as $x\searrow 0$, 
\begin{align*}
\begin{split}
 \int_0^x \frac{K_{1+\alpha}K_{1-\alpha}}{K_\alpha^2}(\sqrt{a}\,)\d a&\sim \int_0^x \frac{\Gamma(1+\alpha)\Gamma(1-\alpha)\sqrt{a}^{-2}}{[2^{\alpha-1}\Gamma(\alpha)\sqrt{a}^{-\alpha}]^2}\d a=
 \frac{\Gamma(1+\alpha)\Gamma(1-\alpha)x^{\alpha}}{2^{2\alpha-2}\Gamma(\alpha)^2\alpha},\\
 \int_0^x \left(\frac{K_1}{K_0}\right)^2(\sqrt{a}\,)\d a&\sim \int_0^x \frac{\d a}{a\log^2\sqrt{a}}=\frac{-4}{\log x}.
 \end{split}
\end{align*}
Then  \eqref{Osgood1} for general $\beta>0$ follows upon using \eqref{mubeta:scaling}. The proof is complete.
\end{proof}

\noindent {\bf Acknowledgments.} This paper is to appear in the \emph{Annals of Applied Probability}. The author would like to thank the anonymous Associate Editor and referees for giving constructive comments and pointing out reference \cite{Eberle}.

\end{document}